\documentclass[11pt,a4paper,reqno]{amsart}
\usepackage[english]{babel}
\usepackage[T1]{fontenc}
\usepackage{verbatim}
\usepackage{palatino}
\usepackage{amsmath}
\usepackage{mathabx}
\usepackage{amssymb}
\usepackage{amsthm}
\usepackage{amsfonts}
\usepackage{graphicx}
\usepackage{esint}
\usepackage{color}
\usepackage{mathtools}
\usepackage{overpic}

\usepackage[colorlinks = true, citecolor = black]{hyperref}
\author{Tuomas Orponen}
\title[Projections of Ahlfors-regular sets]{On the projections of Ahlfors regular sets in the plane}
\address{Department of Mathematics and Statistics\\ University of Jyv\"askyl\"a,
P.O. Box 35 (MaD)\\
FI-40014 University of Jyv\"askyl\"a\\
Finland}
\email{tuomas.t.orponen@jyu.fi}
\date{\today}
\subjclass[2010]{28A80 (primary) 28A78 (secondary)}
\keywords{Projections, Ahlfors regular sets}
\thanks{T.O. is supported by the European Research Council (ERC) under the European Union’s Horizon Europe research and innovation programme (grant agreement No 101087499), and by the Research Council of Finland via the project \emph{Approximate incidence geometry}, grant no. 355453.}

\newcommand{\R}{\mathbb{R}}

\newcommand{\N}{\mathbb{N}}
\newcommand{\Q}{\mathbb{Q}}

\newcommand{\Z}{\mathbb{Z}}
\newcommand{\tn}{\mathbb{P}}

\newcommand{\calS}{\mathcal{S}}

\newcommand{\spt}{\operatorname{spt}}
\newcommand{\Hd}{\dim_{\mathrm{H}}}

\newcommand{\Bd}{\overline{\dim}_{\mathrm{B}}}

\newcommand{\diam}{\operatorname{diam}}

\newcommand{\m}{\mathfrak{m}}
\newcommand{\M}{\mathfrak{M}}
\newcommand{\Div}{\mathfrak{D}}

\def\Barint_#1{\mathchoice
          {\mathop{\vrule width 6pt height 3 pt depth -2.5pt
                  \kern -8pt \intop}\nolimits_{#1}}%
          {\mathop{\vrule width 5pt height 3 pt depth -2.6pt
                  \kern -6pt \intop}\nolimits_{#1}}%
          {\mathop{\vrule width 5pt height 3 pt depth -2.6pt
                  \kern -6pt \intop}\nolimits_{#1}}%
          {\mathop{\vrule width 5pt height 3 pt depth -2.6pt
                  \kern -6pt \intop}\nolimits_{#1}}}

\numberwithin{equation}{section}

\theoremstyle{plain}
\newtheorem{thm}[equation]{Theorem}
\newtheorem{"thm"}[equation]{"Theorem"}

\newtheorem{lemma}[equation]{Lemma}
\newtheorem{"lemma"}[equation]{"Lemma"}

\newtheorem{cor}[equation]{Corollary}
\newtheorem{"proposition"}[equation]{"Proposition"}
\newtheorem{proposition}[equation]{Proposition}

\theoremstyle{definition}

\newtheorem{definition}[equation]{Definition}
\newtheorem{assumption}[equation]{Assumption}
\newtheorem{notation}[equation]{Notation}

\theoremstyle{remark}
\newtheorem{remark}[equation]{Remark}

\newcommand{\nref}[1]{(\hyperref[#1]{#1})}

\DeclareMathSymbol{\intop}  {\mathop}{mathx}{"B3}

\begin{document}

\begin{abstract} This paper contains the following $\delta$-discretised projection theorem for Ahlfors regular sets in the plane.

For all $C,\epsilon > 0$ and $s \in [0,1]$, there exists $\kappa > 0$ such that the following holds for all $\delta > 0$ small enough. Let $\nu$ be a Borel probability measure on $S^{1}$ satisfying $\nu(B(x,r)) \leq Cr^{\epsilon}$ for all $x \in S^{1}$ and $r > 0$. Let $K \subset B(1) \subset \R^{2}$ be Ahlfors $s$-regular with constant at most $C$. Then, there exists a vector $\theta \in \spt \nu$ such that
\begin{displaymath} |\pi_{\theta}(F)|_{\delta} \geq \delta^{\epsilon - s} \end{displaymath}
for all $F \subset K$ with $|F|_{\delta} \geq \delta^{\kappa - s}$. Here $\pi_{\theta}(z) = \theta \cdot z$ for $z \in \R^{2}$.

\end{abstract}

\maketitle

\tableofcontents

\section{Introduction}

This paper is concerned with the orthogonal projections of Ahlfors regular subsets of the plane. Orthogonal projections are the maps defined by $\pi_{\theta}(z) := \theta \cdot z$ for $\theta \in S^{1}$ and $z \in \R^{2}$. Occasionally it is also convenient to use the alternative parametrisation $\pi_{\theta}(x,y) := x + \theta y$, where $\theta \in [0,1]$, but this will be mentioned separately.

Marstrand \cite{mar} in 1954 proved that if $K \subset \R^{2}$ is a Borel set, then 
\begin{equation}\label{form73a} E_{H} := \{\theta \in S^{1} : \Hd \pi_{\theta}(K) < \min\{\Hd K,1\}\} \end{equation}
has $\mathcal{H}^{1}(E_{H}) = 0$ (the subscript "$H$" stands for Hausdorff dimension, see Remark \ref{rem5} for a discussion concerning a few other notions of dimension). In the generality of Borel sets, Marstrand's result is sharp in the sense that the conclusion $\mathcal{H}^{1}(E_{H}) = 0$ cannot be improved to $\Hd E_{H} \leq \sigma$ for any $\sigma < 1$, see \cite{KM}.

Under additional hypotheses on $K$, it is reasonable to expect stronger versions of Marstrand's theorem. This has been demonstrated for self-similar, self-conformal, and self-affine sets \cite{MR3955707,MR3903263,MR3263957,Ho,MR3276605,PS}. These papers show that either $\Hd E_{H} = 0$, or even $|E_{H}| < \infty$ under suitable "irrationality" hypotheses on the underlying dynamics.

\cite[Question 1]{MR4388762} proposed that $\Hd E_{H} = 0$ whenever $K$ is \emph{Ahlfors regular}:
\begin{definition}[Ahlfors $(s,C)$-regularity] Let $C,s > 0$. A Borel measure $\mu$ on $\R^{d}$ is called \emph{Ahlfors $(s,C)$-regular} if
\begin{displaymath} C^{-1}r^{s} \leq \mu(B(x,r)) \leq Cr^{s}, \qquad x \in \spt \mu, \, 0 < r \leq \diam (\spt \mu). \end{displaymath}
A closed set $K \subset \R^{d}$ is called Ahlfors $(s,C)$-regular if $K = \spt \mu$ for some non-trivial Ahlfors $(s,C)$-regular measure $\mu$. \end{definition}

 The motivation for \cite[Question 1]{MR4388762} was that Ahlfors regular sets are not necessarily generated by any dynamical system, but they are "self-similar as a set family" in the sense that Ahlfors regularity is invariant under rescaling. While \cite[Question 1]{MR4388762} remains open as stated, in this paper I obtain the following $\delta$-discretised variant (for the terminology and notation, see Definition \ref{def:regularity} and Notation \ref{not1}):
\begin{thm}\label{mainThm2} For every $C > 0$, $s,\tau \in (0,1]$, and $\underline{s} \in [0,s)$, there exist $\kappa = \kappa(C,s,\underline{s},\tau) > 0$ and $\delta_{0} = \delta_{0}(C,s,\underline{s},\tau) > 0$ such that the following holds for all $\delta \in (0,\delta_{0}]$.

Let $\nu$ be a Borel probability measure on $S^{1}$ satisfying $\nu(B(x,r)) \leq Cr^{\tau}$ for all $x \in S^{1}$ and $r > 0$. Let $\mu$ be an Ahlfors $(s,C)$-regular measure on $\R^{2}$. Then, there exists a vector $\theta \in \spt \nu$ such that 
\begin{equation}\label{form78a} |\pi_{\theta}(F)|_{\delta} \geq \delta^{-\underline{s}} \end{equation}
for all Borel sets $F \subset B(1)$ with $\mu(F) \geq \delta^{\kappa}$. \end{thm}

The following weak answer to \cite[Question 1]{MR4388762} follows easily:
\begin{cor}\label{mainCor} Let $K \subset \R^{2}$ be Ahlfors $s$-regular for some $s \in [0,2]$. Then $\Hd E_{\bar{B}} = 0$, where
\begin{displaymath} E_{\bar{B}} := \{\theta \in S^{1} : \Bd \pi_{\theta}(K) < \min\{s,1\}\}, \end{displaymath}
and $\Bd$ is the \emph{upper box dimension}. \end{cor}

\begin{proof} One may assume $s \leq 1$, since Ahlfors $t$-regular sets always contain Ahlfors $s$-regular subsets for $0 < s < t$ by a result of Mattila and Saaranen \cite{MS08}. Now, assume to the contrary that $\Hd E_{\bar{B}} > \tau > 0$. Then also $E_{\bar{B}}(\underline{s}) := \{\theta : \Bd \pi_{\theta}(K) < \underline{s}\}$ has $\Hd E_{\bar{B}}(\underline{s}) > \tau$ for some $\underline{s} < s \leq 1$. Using Frostman's lemma, find a Borel probability measure $\nu$ on $E_{\bar{B}}(\underline{s})$ satisfying $\nu(B(x,r)) \lesssim r^{\tau}$ for all $x \in S^{1}$ and $r > 0$. By the definition of upper box dimension, the union of the nested sets
\begin{displaymath} E_{\bar{B}}(\underline{s},\delta_{0}) := \{\theta \in S^{1} : |\pi_{\theta}(K)|_{\delta} \leq \delta^{\underline{s}} \text{ for all } \delta \in (0,\delta_{0}]\}, \qquad \delta_{0} \in 2^{-\N}, \end{displaymath}
contains $E_{\bar{B}}(\underline{s})$. In particular $\nu(E_{\bar{B}}(\underline{s},\delta_{0})) \geq \tfrac{1}{2}$ for $\delta_{0} > 0$ sufficiently small. This contradicts Theorem \ref{mainThm2} for $\delta_{0} > 0$ sufficiently small.
\end{proof} 

\begin{remark}\label{rem4} Theorem \ref{mainThm2} looks similar to other $\delta$-discretised projection theorems in the literature, e.g. \cite[Theorem 3]{Bo2} or \cite[Theorem 1]{He20} or \cite[Corollary 4.9]{2023arXiv230110199O}. Why does it not imply that $\Hd E_{H} = 0$? One would need a slightly stronger version, where the hypothesis $\nu(B(x,r)) \leq Cr^{\tau}$ is relaxed to $\nu(B(x,r)) \leq C(\delta)r^{\tau}$, where $n \mapsto 1/C(2^{-n})$ is summable (e.g. $\nu(B(x,r)) \leq (\log \tfrac{1}{\delta})^{2}r^{\tau}$).

On the other hand, Theorem \ref{mainThm2} is stronger than Corollary \ref{mainCor} in the sense that strictly weaker versions of Theorem \ref{mainThm2} would already imply Corollary \ref{mainCor}. To name one example, in place of \eqref{form78a}, it would suffice to know that $|\pi_{\theta}(F)|_{\Delta} \geq \Delta^{-\underline{s}}$ for some $\Delta \in [\delta,\delta^{\kappa}]$.
 \end{remark} 

\begin{remark}\label{rem5} This remark discusses what is known, or plausible, regarding variants of \cite[Question 1]{MR4388762} for other notions of dimension. First of all, M. Wu has announced \cite{Wu24} a result much stronger than Corollary \ref{mainCor}: $E_{\bar{B}}$ is countable whenever $K$ is Ahlfors regular! Wu's method is different than the one used in this paper, and it does not yield Theorem \ref{mainThm2}. It does not seem plausible that one could deduce Theorem \ref{mainThm2} from information on $\Bd \pi_{\theta}(K)$ alone, for the reason mentioned at the end of Remark \ref{rem4}.

The set $E_{H}$ (as in \eqref{form73a}) may not be countable, even if $K$ is Ahlfors regular: \cite[Theorem 1.5]{Or1} contains the construction of an Ahlfors $1$-regular set $K \subset \R^{2}$ such that $E_{H}(0) = \{\theta \in S^{1} : \Hd \pi_{\theta}(K) = 0\}$ is a dense $G_{\delta}$-set, in particular $E_{H}$ has full packing dimension. In fact, the short proof of \cite[Theorem 1.5]{Or1} shows \emph{a fortiori} that $E_{\underline{B}}(0) = \{\theta \in S^{1} : \underline{\dim}_{B} \pi_{\theta}(K) = 0\}$ is a dense $G_{\delta}$-set, where $\underline{\dim}_{B}$ is the lower box dimension. So, even when $K$ is Ahlfors regular, neither $E_{H}$ nor $E_{\underline{B}}$ need be countable. 

If $K$ is a general (non-Ahlfors regular) compact set, $E_{\bar{B}}$ may have positive packing dimension, see \cite[Theorem 1.17]{Or1} (so $E_{\bar{B}}$ need not be countable for general compact sets). On the other hand, the weaker conclusion $\Hd E_{\bar{B}} = 0$ of Corollary \ref{mainCor} seems plausible for all compact sets. However, Ahlfors regularity (or at least some extra hypothesis of this nature) is necessary for the the $\delta$-discretised Theorem \ref{mainThm2}, see e.g. the construction in \cite[Appendix A]{MR4745881}. The $\delta$-discretised counterpart of the question $\Hd E_{\bar{B}} = 0$, for general compact sets, would rather have to be formulated in the way suggested in Remark \ref{rem4}.

Finally, for general (even non-compact) sets, \cite{MR4218963} shows that $\Hd E_{A} = 0$, where $E_{A} = \{\theta \in S^{1} : \dim_{\mathrm{A}} \pi_{\theta}(K) < \min\{\dim_{\mathrm{A}} K,1\}\}$, and $\dim_{\mathrm{A}}$ is the Assouad dimension. Wu has announced that $E_{A}$ is always countable.  \end{remark}

\begin{remark} In addition to the sizes of $E_{H},E_{\bar{B}},E_{\underline{B}},E_{A}$, one may ask for the sharp upper bounds on the sizes of $E_{H}(s),E_{\bar{B}}(s),E_{\underline{B}}(s),E_{A}(s)$, where for example
\begin{displaymath} E_{H}(s) := \{\theta \in S^{1} : \Hd \pi_{\theta}(K) \leq s\}, \qquad s \in [0,\min\{\Hd K,1\}). \end{displaymath}
The sharp upper bounds on $\Hd E_{H}(s)$, for all $s \in [0,\min\{\Hd K,1\})$, were established in 2023 by Ren and Wang \cite{2023arXiv230808819R}. For recent progress and more references regarding this problem in higher dimensions, see \cite{He20,2023arXiv230904097R}. \end{remark} 

I now proceed to define a slightly weaker variant of Ahlfors regularity.

\begin{definition}[Upper $(s,C)$-regularity]\label{def:upperRegularity} Let $C,s >0$. A set $K \subset \R^{d}$ is called \emph{upper $(s,C)$-regular} if 
\begin{displaymath} N_{r}(K \cap B(x,R)) \leq C\left(\tfrac{R}{r} \right)^{s}, \qquad 0 < r \leq R < \infty, \, x \in \R^{d}. \end{displaymath}
Here $N_{r}(A)$ is the least number of balls of radius $r$ required to cover $A \subset \R^{d}$. \end{definition}

\begin{definition}[$(s,C)$-Frostman measure] Let $C,s > 0$. A Borel measure $\mu$ on $\R^{d}$ is called \emph{$(s,C)$-Frostman} if $\mu(B(x,r)) \leq Cr^{s}$ for all $x \in \R^{d}$ and $r > 0$. \end{definition} 

\begin{definition}[$(s,C)$-regularity]\label{def:regularity} Let $C,s > 0$. A Borel measure $\mu$ on $\R^{d}$ is called \emph{$(s,C)$-regular} if $\mu$ is $(s,C)$-Frostman, and $K = \spt \mu$ is upper $(s,C)$-regular.  A closed set $K \subset \R^{d}$ is called $(s,C)$-regular if $K$ coincides with the support of a non-trivial $(s,C)$-regular measure. Notably, all (Ahlfors) $(s,C)$-regular sets are assumed to be closed and non-empty without further remark. \end{definition}

\begin{remark} Theorem \ref{mainThm2} concerns the (stronger) notion of Ahlfors $s$-regularity, but the weaker $s$-regularity is sufficient for a major part of the paper, namely Theorem \ref{mainThm} below. The weaker notion might well be sufficient for all the results, but the stronger Ahlfors $s$-regularity is technically convenient: it guarantees the existence of a system of \emph{David cubes}, see Definition \ref{def:DavidCubes}. If desired, generalising all the results to $s$-regular sets only hinges on generalising the "stand-alone" Lemma \ref{lemma9}.

Note that if $\mu$ is an Ahlfors $(s,C)$-regular measure, then $\mu$ is $(s,2^{s}C)$-Frostman, and $K = \spt \mu$ is upper $(s,\mathbf{C})$-regular with $\mathbf{C} \lesssim_{d,s} C^{2}$. Therefore, Ahlfors $(s,C)$-regularity is stronger than $(s,C)$-regularity.
\end{remark} 

\subsection{Proof sketch and paper structure}\label{s:outline}

Theorem \ref{mainThm2} will be deduced from the following variant concerning \emph{high-multiplicity sets} $H_{\theta}(K,\delta^{-\sigma},[\delta,1])$. For the precise definition, see Notation \ref{def:highMultiplicity}. Informally, $H_{\theta}(K,M,[\delta,1])$ contains all the points of $K$ contained on lines of the form $\pi_{\theta}^{-1}\{t\}$, which hit the $\delta$-neighbourhood of $K$ at least $M$ times.

\begin{thm}\label{mainThm} For every $\mathbf{C},\epsilon,\sigma > 0$, and $s \in [0,1]$, there exists $\delta_{0} = \delta_{0}(\mathbf{C},\epsilon,\sigma) > 0$ such that the following holds for all $\delta \in (0,\delta_{0}]$. 

Let $\mu$ be an $(s,\mathbf{C})$-regular measure on $\R^{2}$, and let $\nu$ be a Borel measure on $S^{1}$ satisfying $\nu(B(x,r)) \leq \mathbf{C}r^{\epsilon}$ for all $x \in S^{1}$ and $r > 0$. Then,
\begin{displaymath} \int_{S^{1}} \mu(B(1) \cap H_{\theta}(\spt \mu,\delta^{-\sigma},[\delta,1])) \, d\nu(\theta) \leq \epsilon. \end{displaymath}  \end{thm}

The proof of Theorem \ref{mainThm} will be completed in Section \ref{s:statements}.

\begin{remark} Theorem \ref{mainThm2} is a formal corollary of Theorem \ref{mainThm}, but this is not straightforward: the biggest difference between the statements is that Theorem \ref{mainThm2}  yields information about subsets $F \subset B(1)$ of $\mu$ measure $\delta^{\kappa}$. \emph{A priori}, Theorem \ref{mainThm} only seems powerful enough to deal with subsets of $\mu$ measure $\epsilon$. The main point of the argument, contained in Section \ref{s:selfImprovement}, is that Theorem \ref{mainThm} may be applied to suitable rescalings of the "given" measure $\mu$, since such renormalisations remain (Ahlfors) $(s,C)$-regular.  
\end{remark}

Sections \ref{s:technicalLemmas}-\ref{s:branching} contain the proof of Theorem \ref{mainThm}. The argument has significant similarities to the proof of \cite[Theorem 3.6]{MR4388762}. In particular, it is also based on an iterative scheme, where the "multiplicity parameter" $\sigma$ is gradually reduced towards $0$, see Proposition \ref{mainProp} for the formal statement, and "Proposition" \ref{mainPropToy} for an informal version.

For the remainder of Section \ref{s:outline}, I will outline the proof of Theorem \ref{mainThm}, and then explain the differences to \cite[Theorem 3.6]{MR4388762} in Remark \ref{rem3}. The discussion will be highly informal and heuristic: the reader should assume that none of the statements are literally true, and check (if desired) the accurate counterparts from the body of the paper.

\begin{notation}\label{not1} Below, for $A \subset \R^{d}$, the notation $\mathcal{D}_{r}(A)$ refers to the standard dyadic cubes of side-length $r \in 2^{\Z}$ intersecting $A$, and $|A|_{r} := |\mathcal{D}_{r}(A)|$. This \emph{dyadic $r$-covering number} is comparable to the $r$-covering number $N_{r}(A)$ already used above. \end{notation}

Here is an informal version of Proposition \ref{mainProp}:

\begin{"proposition"}\label{mainPropToy} Let $s \in (0,1]$ and $\sigma_{0},\tau > 0$. Assume that for \underline{all} $\epsilon > 0$ there exists $\Delta_{\epsilon} > 0$ such that for \underline{all} $s$-regular sets $K \subset B(1)$:
\begin{equation}\label{form9t} \mathcal{H}_{\infty}^{\tau}\{\theta : |\pi_{\theta}(K)|_{\Delta} \leq \Delta^{\sigma_{0} - s}\} \leq \epsilon, \qquad \Delta \in (0,\Delta_{\epsilon}]. \end{equation}
Let $\sigma \in (0,\sigma_{0})$ be so close to $\sigma_{0}$ that $\sigma_{0}(1 - \sigma) < \sigma$. Then, for \underline{all} $\epsilon > 0$ there exists $\delta_{\epsilon} > 0$ such that for \underline{all} $s$-regular sets $K \subset B(1)$:
\begin{equation}\label{form8t} \mathcal{H}_{\infty}^{\tau}\{\theta : |\pi_{\theta}(K)|_{\delta} \leq \delta^{\sigma - s}\} \leq \epsilon. \end{equation} 
\end{"proposition"} 

It is easy to see that the hypothesis \eqref{form9t} is valid, at least, with $\sigma_{0} = 1$.

A key ingredient in the proof of "Proposition" \ref{mainPropToy} is Shmerkin's inverse theorem \cite[Theorem 2.1]{Sh}, stated rigorously as Theorem \ref{shmerkin}. An informal version is "Theorem" \ref{t:inverse} below, but even for that we need some notation.

\begin{definition}[Special case of uniform sets] Let $\Delta \in 2^{-\N}$ and $n \in \N$. A set $P \subset [0,1)$ is called \emph{$\{\Delta^{j}\}_{j = 0}^{n - 1}$-uniform} if there exists a sequence $\{N_{j}\}_{j = 0}^{n - 1} \subset \N^{n}$ such that 
\begin{displaymath} |P \cap Q|_{\Delta^{j + 1}} = N_{j}, \qquad Q \in \mathcal{D}_{\Delta^{j}}(P). \end{displaymath}
 \end{definition} 
 
\begin{"thm"}[Inverse theorem]\label{t:inverse} For every $\Delta \in 2^{-\N}$ there exists $\zeta > 0$ and $n_{0} \in \N$ such that the following holds for $n \geq n_{0}$. Write $\delta := \Delta^{n}$. Let $A,B \subset [0,1]$ be $\{\Delta^{j}\}_{j = 0}^{n - 1}$-uniform sets with branching numbers $\{N_{j}^{A}\}$ and $\{N_{j}^{B}\}$. Assume that
\begin{equation}\label{form11t} |A + B|_{\delta} \leq \delta^{-\zeta}|A|_{\delta}. \end{equation}
Then, $\{0 \leq j \leq n - 1 : N_{j}^{B} > 1\} \subset \{0 \leq j \leq n - 1 : N_{j}^{A} = \Delta^{-1}\}$.  \end{"thm"}

We are then equipped to "sketch" the proof of "Proposition" \ref{mainPropToy}.

\begin{proof}["Sketch" of a proof for "Proposition" \ref{mainPropToy}] Let $K \subset B(1)$ be $s$-regular, fix $\epsilon > 0$. We want to prove that $\mathcal{H}_{\infty}^{\tau}(E) \leq \epsilon$, where 
\begin{displaymath} E = \{\theta \in [0,1] : |\pi_{\theta}(K)|_{\delta} \leq \delta^{\sigma - s}\}. \end{displaymath}
and $\delta > 0$ can be taken as small as we like. In this discussion it is useful to use the parametrisation $\pi_{\theta}(x,y) = x + \theta y$ for orthogonal projections, with $\theta \in [0,1]$. Abbreviate $\nu := \mathcal{H}^{\tau}_{\infty}$ (or let $\nu$ be a $\tau$-dimensional Frostman measure inside $E$). Start with a counter assumption: $\nu(E) > \epsilon$. Then, use the hypothesis \eqref{form8t} with parameter $\epsilon^{2}$ to deduce the existence of a scale $\Delta \in (\delta,1]$ such that
\begin{equation}\label{form13} \mathcal{H}^{\tau}\{\theta : |\pi_{\theta}(\bar{K})|_{\Delta} \leq \Delta^{\sigma_{0} - s}\} \leq \epsilon^{2} \leq \epsilon \cdot \nu(E), \end{equation}
whenever $\bar{K} \subset B(1)$ is $s$-regular. Thus, for every $s$-regular $\bar{K} \subset B(1)$, at most an $\epsilon$-fraction (w.r.t. $\nu$) of the points $\theta \in E$ are "bad" in the sense $|\pi_{\theta}(\bar{K})|_{\Delta} \geq \Delta^{\sigma_{0} - s}$. For this exposition, I will now assume (quite unrealistically) that this fraction is $0$:
\begin{assumption}\label{A1} If $\bar{K} \subset B(1)$ is $s$-regular, then $|\pi_{\theta}(\bar{K})|_{\Delta} \geq \Delta^{\sigma_{0} - s}$ for all $\theta \in E$. \end{assumption}

I proceed to draft a list of further suspicious extra assumptions. They are never exactly correct, but substitutes of them will be obtained in the body of the paper.

\begin{assumption} Assume that $\{0,1\} \subset E$ (so in particular Assumption \ref{A1} applied to $0$). This is quite reasonable, since $\mathcal{H}^{\tau}(E) \geq \epsilon$ by the counter assumption. Therefore $E$ certainly contains a pair of well-separated points, and why not $\{0,1\}$. \end{assumption}
\begin{assumption}\label{A2} Assume that $K = A \times B$. This is too technical to justify in this proof sketch. On the other hand, reducing to the case $K \approx A \times B$ is an old idea, first used in \cite{MR4055989}, and later \cite{MR4388762,OS23,2023arXiv230110199O}. Since $\{0,1\} \subset E$:
\begin{equation}\label{form10t} |A|_{\delta} = |\pi_{0}(A \times B)|_{\delta} \leq \delta^{\sigma - s} \quad \text{and} \quad |A + B|_{\delta} = |\pi_{1}(A \times B)|_{\delta} \leq \delta^{\sigma - s}. \end{equation}
Moreover, since $|K|_{\delta} \sim \delta^{-s}$ by $s$-regularity, the upper bound for $|A|_{\delta}$ implies $|B|_{\delta} \geq \delta^{-\sigma}$. \end{assumption}

\begin{assumption} Assume that $\delta = \Delta^{n}$ for some $n \in \N$, where $\Delta$ is the scale from \eqref{form13}. Assume that $A,B$ are $\{\Delta^{j}\}_{j = 0}^{n - 1}$-uniform. This assumption does not have a counterpart in the real proof: it is only needed here apply the sloppily stated "Theorem" \ref{t:inverse}. In the true version of the inverse theorem $A,B$ need not be uniform to begin with, but one locates uniform subsets of them as a consequence of the theorem. \end{assumption}

\begin{"lemma"}\label{lemma14} The branching numbers $\{N_{j}^{A}\}_{j = 0}^{N - 1}$ of $A$ satisfy 
\begin{equation}\label{form12t} N_{j}^{A} \geq \Delta^{\sigma_{0} - s}, \qquad 0 \leq j \leq n - 1. \end{equation} \end{"lemma"}
\begin{proof} This follows from Assumption \ref{A1} applied to all "renormalised" sets of the from $K^{Q} = T_{Q}(K)$, where $Q = I \times J \in \mathcal{D}_{\Delta^{j}}(K)$, and $T_{Q}$ is a rescaling map sending $Q$ to $[0,1)^{2}$. Note that
\begin{displaymath} K^{Q} = T_{I}(A) \times T_{J}(B) =: A^{I} \times B^{J}.\end{displaymath}
The sets $K^{Q}$ remain $s$-regular (with the original constant), so Assumption \ref{A1} applies to them: $N_{j} = |A \cap I|_{\Delta^{j + 1}} = |A^{I}|_{\Delta} = |\pi_{0}(K^{Q})|_{\Delta} \geq \Delta^{\sigma_{0} - s}$. \end{proof}

As a corollary of "Lemma" \ref{lemma14}, we could deduce that $|A|_{\delta} \geq \delta^{\sigma_{0} - s}$, which is an approximate converse to \eqref{form10}. However, this lower bound is not quite strong enough, as discussed further below, and we need the following stronger version:

\begin{assumption}\label{A3} Assume that $|A|_{\delta} \geq \delta^{\sigma - s + \zeta}$ for small number $\zeta = \zeta(\Delta) > 0$, determined by applying the inverse theorem at scale $\Delta$. Achieving this lower bound is technically tedious, but eventually a matter of pigeonholing: this is accomplished by Proposition \ref{mainTechnicalProp}. In order to achieve simultaneously the upper and lower bounds $\delta^{\sigma - s + \zeta} \leq |A|_{\delta} \leq \delta^{\sigma - s}$, the "given" parameter $\sigma$ may have to be replaced by $\bar{\sigma} > \sigma$. This is ultimately harmless, because the key relation $\sigma_{0}(1 - \bar{\sigma}) < \bar{\sigma}$ remains valid for $\bar{\sigma} \geq \sigma$.  \end{assumption}

After this list of assumptions, we have
\begin{displaymath} |A + B|_{\delta} \stackrel{\eqref{form10t}}{\leq} \delta^{\sigma - s} \leq \delta^{-\zeta}|A|_{\delta}. \end{displaymath}
Thus, the hypothesis \eqref{form11t} of "Theorem" \ref{t:inverse} is met for $A$ and $B$. (Note that the information $|A + B|_{\delta} \leq \delta^{\sigma - \sigma_{0}}|A|_{\delta}$, implied by "Lemma" \ref{lemma14}, would not suffice here, as the difference $\sigma_{0} - \sigma$ is not allowed to depend on $\Delta$.) We may conclude that $\{0 \leq j \leq n - 1 : N_{j}^{B} > 1\} \subset \{0 \leq j \leq n - 1 : N_{j}^{A} = \Delta^{-1}\}$. To apply this information, write 
\begin{displaymath} \lambda := \tfrac{1}{n}|\{0 \leq j \leq n - 1 : N_{j}^{B} > 1\}. \end{displaymath} 
Since $|B|_{\delta} \geq \delta^{-\sigma}$, as observed in Assumption \ref{A2}, necessarily $\lambda \geq \sigma$. We may now estimate $|A|_{\delta}$ from below as follows, using $\delta = \Delta^{n}$:
\begin{displaymath} |A|_{\delta} = \prod_{N_{j}^{B} = 1} N_{j}^{A} \times \prod_{N_{j}^{B} > 1} N_{j}^{A} \stackrel{\eqref{form12t}}{\geq} (\Delta^{-s + \sigma_{0}})^{(1 - \lambda)n} (\Delta^{-1})^{\lambda n} = \delta^{(1 - \lambda)(\sigma_{0} - s) - \lambda} \geq \delta^{(1 - \sigma)(\sigma_{0} - s) - \sigma}. \end{displaymath} 
On the other hand $|A|_{\delta} \leq \delta^{\sigma - s}$ by \eqref{form10t}. Since $s \in (0,1]$ and $\sigma_{0}(1 - \sigma) < \sigma$, these inequalities are incompatible for $\delta > 0$ small enough depending on $\sigma,\sigma_{0}$. This contradiction completes the "sketch" for the proof of "Proposition" \ref{mainPropToy}. \end{proof}

\begin{remark}\label{rem3} Compared to \cite[Theorem 3.6]{MR4388762}, the main differences in the proof sketch above are (i) the application of "Lemma" \ref{lemma14} and (ii) the addition of Assumption \ref{A3}. At the time of writing  \cite[Theorem 3.6]{MR4388762}, I only knew the bound $|A + B|_{\delta} \leq \delta^{\sigma - \sigma_{0}}|A|_{\delta}$, which was too weak to apply the inverse theorem with scale parameter $\Delta$. As a substitute, I applied the inverse theorem with some "absolute" $\Delta_{0} \in 2^{\N}$ (not stemming from the "inductive" hypothesis \eqref{form9t}). Since $\Delta_{0}$ had nothing to do with \eqref{form9t}, I resorted to a "trivial" substitute of "Lemma" \ref{lemma14}, valid only for the product measures considered in \cite[Theorem 3.6]{MR4388762}. Plugging this "trivial" lower bound into the final computation in the sketch above gave the partial results in \cite{MR4388762}.   \end{remark}

\subsection*{Notation} The meaning of $\mathcal{D}_{r}(A)$ and $|A|_{r}$ was explained in Notation \ref{not1}. Additionally,
\begin{displaymath} \mathcal{D}_{r} := \mathcal{D}_{r}(\R^{d}). \end{displaymath}
Open balls in either $\R^{d}$ will be denoted $B(x,r)$; in this paper always $d \in \{1,2\}$. If $f,g$ are real-valued non-negative functions of some parameter $x \in X$, the notation $f \lesssim g$ means that there exists an absolute constant $C \geq 1$ such that $f(x) \leq Cg(x)$ for all $x \in X$. If the constant $C$ is allowed to depend on a parameter "$p$", this is signified by writing $A \lesssim_{p} B$. If $A \subset \R^{d}$ is a bounded set, and $\delta > 0$, then $N_{\delta}(A)$ refers to the smallest number of balls $B(x,\delta) \subset \R^{d}$ required to cover $A$. If $A \subset \R^{d}$ is a finite set, its cardinality is denoted $|A|$. For $\delta > 0$, the open Euclidean $\delta$-neighbourhood of a set $A \subset \R^{d}$ is denoted $[A]_{\delta}$.  

\subsection*{Acknowledgements} I am most grateful to Pablo Shmerkin for discussions and ideas throughout the project. On my account he could be considered a co-author.

\section{Technical lemmas}\label{s:technicalLemmas}

\subsection{Lemmas on high multiplicity sets} I start by defining the notion of high multiplicity sets which already appeared in the statement of Theorem \ref{mainThm}.

\begin{notation}\label{def:highMultiplicity} For a set $K \subset \R^{2}$, a pair of dyadic scales $\delta \leq \Delta$, and a direction $\theta \in S^{1}$, we define the multiplicity function
\begin{displaymath} \mathfrak{m}_{K,\theta}(x \mid [\delta,\Delta]) := |K_{\delta} \cap B(x,\Delta) \cap \pi_{\theta}^{-1}\{\pi_{\theta}(x)\}|_{\delta}, \end{displaymath}
where $K_{\delta} := \cup \mathcal{D}_{\delta}(K)$, and $|\cdot|_{\delta}$ refers to the dyadic $\delta$-covering number. We also write
\begin{displaymath} H_{\theta}(K,N,[\delta,\Delta]) := \{x \in \R^{2} : \mathfrak{m}_{K,\theta}(x \mid [\delta,\Delta]) \geq N\}, \qquad N \geq 1. \end{displaymath}
\end{notation}

The next lemma on the effect of rescaling on high multiplicity sets is \cite[Lemma 2.11]{MR4388762}:
\begin{lemma}\label{lemma7} Let $K \subset \R^{2}$ be arbitrary, let $0 < r \leq R \leq \infty$, $M > 0$, and $\theta \in [0,1]$. Then, 
\begin{displaymath} T_{z_{0},r_{0}}(H_{\theta}(K,M,[r,R])) = H_{\theta}(T_{z_{0},r_{0}}(K),M,[\tfrac{r}{r_{0}},\tfrac{R}{r_{0}}]), \qquad z_{0} \in \R^{2}, \, r_{0} > 0, \end{displaymath}
where $T_{z_{0},r_{0}}(z) = (z - z_{0})/r_{0}$. \end{lemma}

The following observations from \cite[Lemma 2.6]{MR4388762} are frequently useful:
\begin{lemma}\label{lemma13} Let $K \subset \R^{2}$, $\theta \in S^{1}$, and $C \geq 1$ and $M \leq N$. Assume $C\delta \leq \Delta$. Then: 
\begin{itemize}
\item[(i)] $H_{\theta}(K,N,[\delta,\Delta]) \subset H_{\theta}(K,M,[\delta,\Delta])$.
\item[(ii)] $H_{\theta}(K,M,[\delta,\Delta]) \subset H_{\theta}(K,M,[\delta,C\Delta])$.
\item[(iii)] $H_{\theta}(K,M,[\delta,\Delta]) \subset H_{\theta}(K,\tfrac{M}{C},[C\delta,\Delta])$.
\end{itemize}
\end{lemma}

The following lemma is essentially \cite[Proposition 5.1]{MR4388762} (and \cite[Proposition 5.1]{MR4388762} would also suffice in this paper). The proof in \cite{MR4388762} was so unnecessarily long that I decided to give here a self-contained streamlined proof.

\begin{lemma}\label{lemma15} Let $A,\mathbf{C} \geq 1$, $s \in [0,2]$. Let $\mu$ be an $(s,\mathbf{C})$-regular measure, and let $K \subset \R^{2}$ be an $(s,\mathbf{C})$-regular set (not necessarily $\spt \mu$). Let $\kappa \in (0,1]$, $1 \leq M \leq N$, and $\delta,\Delta \in 2^{-\N}$ with $\delta \leq \Delta$. Then, for every $\theta \in S^{1}$ fixed,
\begin{align*} \mu(B(1) \cap H_{\theta}(K,N,[\delta,A])) & \leq (1 + \kappa)  \mu(B(1) \cap H_{\theta}(K,M,[3\delta,3\Delta])) + \kappa,\\
& + \mu(B(1) \cap H_{\theta}(K,N,[\delta,A]) \cap H_{\theta}(K,\mathbf{c}(\kappa) \cdot \tfrac{N}{M},[\Delta,3])), \end{align*}
where $\mathbf{c}(\kappa) = c\kappa^{2}/(A^{2}\mathbf{C}^{3})$ for an absolute constant $c > 0$.
 \end{lemma}
 

\begin{proof} Fix $\theta \in S^{1}$. In this proof, we abbreviate $H_{\theta}(K,\ldots)$ to $H(\ldots)$. We also abbreviate $\mu(B(1) \cap F)$ to $\mu_{1}(F)$.

Fix $\kappa > 0$ as in the statement, and let $\eta = \eta(A,\mathbf{C},\kappa) > 0$ be so small that 
\begin{equation}\label{form14} (1 - \eta)^{-1} \leq 1 + \kappa \quad \text{and} \quad A^{s}C\mathbf{C}\eta \leq \kappa \end{equation}
for a suitable absolute constant $C > 0$ to be determined in \eqref{form10} below.

Let $\mathcal{T} := \mathcal{T}_{\theta}$ be a minimal cover of $B(1) \cap H(N,[\delta,A])$ by disjoint (half-open) tubes of width $\delta$ parallel to $\pi_{\theta}^{-1}\{0\}$. In particular, every tube $T \in \mathcal{T}$ contains at least one point $x_{T} \in B(1) \cap H(N,[\delta,A])$. Moreover, let us check that
\begin{equation}\label{form10} |\mathcal{T}| \leq A^{s}C\mathbf{C}\delta^{-s}/N. \end{equation}
By the definition of $x_{T} \in H(N,[\delta,A]) \cap T$, it holds $|K_{\delta} \cap B(x_{T},A) \cap \pi_{\theta}^{-1}\{\pi_{\theta}(x_{T})\}|_{\delta} \geq N$. In particular, there are $\geq N$ squares in $\mathcal{D}_{\delta}(B(3A) \cap K)$ intersecting $T$. Since the $\delta$-tubes in $\mathcal{T}$ are disjoint, and $|B(3A) \cap K|_{\delta} \leq \mathbf{C}(3A/\delta)^{s}$ by the $(s,\mathbf{C})$-regularity of $K$, the claim follows.

A tube $T \in \mathcal{T}$ is called \emph{good} if 
\begin{displaymath} \mu_{1}(T \cap H(M,[3\delta,3\Delta])) \geq (1 - \eta)\mu_{1}(T \cap H(N,[\delta,A])), \end{displaymath}
and otherwise \emph{bad}. The good and bad tubes are denoted $\mathcal{T}_{\mathrm{good}}$ and $\mathcal{T}_{\mathrm{bad}}$, respectively. With this notation,
\begin{align*} \mu_{1}(H(N,[\delta,A])) & \leq (1 - \eta)^{-1} \sum_{T \in \mathcal{T}_{\mathrm{good}}} \mu_{1}(T \cap H(M,[3\delta,3\Delta])) + \sum_{T \in \mathcal{T}_{\mathrm{bad}}} \mu_{1}(T \cap H(N,[\delta,A]))\\
& \stackrel{\eqref{form14}}{\leq} (1 + \kappa) \mu_{1}(H(M,[3\delta,3\Delta])) + \sum_{T \in \mathcal{T}_{\mathrm{bad}}} \mu_{1}(T \cap H(N,[\delta,A])).  \end{align*}

To bound the sum over the bad tubes, we split them further into $\mathcal{T}_{\mathrm{bad}}^{\mathrm{light}}$ and $\mathcal{T}_{\mathrm{bad}}^{\mathrm{heavy}}$, where
\begin{displaymath} \mathcal{T}_{\mathrm{bad}}^{\mathrm{light}} := \{T \in \mathcal{T}_{\mathrm{bad}} : \mu_{1}(T \cap H(N,[\delta,A])) \leq \eta N \cdot \delta^{s}\}, \end{displaymath}
and $\mathcal{T}_{\mathrm{bad}}^{\mathrm{heavy}} = \mathcal{T}_{\mathrm{bad}} \, \setminus \, \mathcal{T}_{\mathrm{bad}}^{\mathrm{light}}$. By \eqref{form10}, the light bad tubes only cover a small set:
\begin{displaymath} \sum_{T \in \mathcal{T}^{\mathrm{light}}_{\mathrm{bad}}} \mu_{1}(T \cap H(N,[\delta,A])) \leq A^{s}C\mathbf{C}\eta \stackrel{\eqref{form14}}{\leq} \kappa. \end{displaymath}
After these estimates, the lemma will be proven once we show that
\begin{equation}\label{form15} \sum_{T \in \mathcal{T}_{\mathrm{bad}}^{\mathrm{heavy}}} \mu_{1}(T \cap H(N,[\delta,A])) \leq \mu_{1}(H(N,[\delta,A]) \cap H(\mathbf{c}(\kappa) \cdot \tfrac{N}{M},[\Delta,3])). \end{equation}
In fact, we will prove this with $\mathbf{c}(\kappa) \sim \mathbf{C}^{-1}\eta^{2} \sim \kappa^{2}/(A^{2}\mathbf{C}^{3})$, as claimed in the lemma.

Fix $T \in \mathcal{T}^{\mathrm{heavy}}_{\mathrm{bad}}$. We claim that 
\begin{equation}\label{form16} B(1) \cap T \subset H(c\eta^{2} \cdot \tfrac{N}{M},[\Delta,3]) \end{equation}
for a suitable constant $c \sim \mathbf{C}^{-1}$. Clearly \eqref{form16} implies \eqref{form15}. 

To prove \eqref{form16}, first using the definition of badness, and then heaviness,
\begin{equation}\label{form11} \mu_{1}(T \, \setminus \, H(M,[3\delta,3\Delta])) \geq \eta \cdot \mu_{1}(T \cap H(N,[\delta,A])) \geq \eta^{2}N \cdot \delta^{s}. \end{equation}
Let 
\begin{equation}\label{form17} \{x_{1},\ldots,x_{m}\} \subset B(1) \cap K \cap T \, \setminus \, H(M,[3\delta,3\Delta]) \end{equation}
be a maximal $\delta$-separated set. Then 
\begin{equation}\label{form13} m \geq \eta^{2}N/\mathbf{C} \end{equation}
by \eqref{form11}, and since $\mu_{1}(B(x_{j},\delta)) \leq \mathbf{C}\delta^{s}$ for all $1 \leq j \leq m$. We will now use the points $\{x_{1},\ldots,x_{m}\}$ to prove \eqref{form16}. The rough idea is that since the points $x_{j}$ lie in the complement of $H(M,[2\delta,2\Delta])$, it takes $\gtrsim m/M \gtrsim_{\mathbf{C}} \eta^{2}N/M$ balls of radius $\Delta$ to cover them, and this eventually gives \eqref{form16}.

Unwrapping the definitions, \eqref{form16} is equivalent to
\begin{equation}\label{form12} |K_{\Delta} \cap B(x,3) \cap \pi_{\theta}^{-1}\{\pi_{\theta}(x)\}|_{\Delta} \geq c\eta^{2}\tfrac{N}{M}, \qquad x \in B(1) \cap T. \end{equation}
Fix $x \in B(1) \cap T$, and let $Q_{1},\ldots,Q_{n}$ be the family of all dyadic $\Delta$-squares intersecting $K_{\Delta} \cap B(x,3) \cap \pi_{\theta}^{-1}\{\pi_{\theta}(x)\}$. 

Since $x_{j} \in B(1) \cap T$ for each $1 \leq j \leq m$, and $\pi_{\theta}^{-1}\{\pi_{\theta}(x)\}$ is a line contained in $T$, we may pick points $y_{j} \in B(x_{j},\delta) \cap \pi_{\theta}^{-1}\{\pi_{\theta}(x)\}$, $1 \leq j \leq m$. Since the points $x_{j}$ are $\delta$-separated, the set $\{y_{1},\ldots,y_{m}\}$ still has cardinality $\sim m$; we pretend here that $|\{y_{1},\ldots,y_{m}\}| = m$ for simplicity, this only affects absolute constants. Since 
\begin{displaymath} y_{j} \in K_{\delta} \cap B(x_{j},1) \subset K_{\Delta} \cap B(x,3), \end{displaymath}
each point $y_{j}$ is covered by one of the squares $Q_{1},\ldots,Q_{n}$.

Now assume that \eqref{form12} fails, thus $n \leq c\eta^{2}N/M$. This implies that at least one of the squares $Q_{1},\ldots,Q_{n}$, say $Q_{i_{0}}$, contains a subset $Y \subset \{y_{1},\ldots,y_{m}\}$ with cardinality
\begin{displaymath} |Y| \gtrsim \frac{m}{c\eta^{2}N/M} \stackrel{\eqref{form13}}{\gtrsim} \frac{M}{c \cdot \mathbf{C}}. \end{displaymath}
In particular, $Y \subset \pi_{\theta}^{-1}\{\pi_{\theta}(x)\}$, and $\diam(Y) \leq \diam Q_{i_{0}} \leq 2\Delta$. Let $X := \{x_{j} : y_{j} \in Y\}$, thus also $|X| > CM$ (for a suitable absolute constant $C \geq 1$), provided that we choose $c \sim \mathbf{C}^{-1}$ sufficiently small. 

We now finally claim that $X \subset H(M,[3\delta,3\Delta])$. Since also $X \subset \{x_{1},\ldots,x_{m}\}$, this will contradict the choice of $\{x_{1},\ldots,x_{m}\}$ at \eqref{form17}.

The claim $X \subset H(M,[3\delta,3\Delta])$ is equivalent to
\begin{equation}\label{form18} |K_{3\delta} \cap B(x_{j},3\Delta) \cap \pi_{\theta}^{-1}\{\pi_{\theta}(x_{j})\}|_{3\delta} \geq M, \qquad x_{j} \in X. \end{equation} 
To see this, fix $x_{j} \in X$, and note that all the other elements $x_{i} \in X$ also lie inside $K_{3\delta} \cap B(x_{j},3\Delta)$. Since these points are $\delta$-separated, and $|X| \geq CM$, it is easy to check that \eqref{form18} holds, provided that $C \geq 1$ is sufficiently large. This completes the proof. \end{proof}

The next lemma states that if a set satisfies local multiplicity upper bounds on "most scales", then it also satisfies a global multiplicity upper bound.
 
\begin{lemma}\label{lemma5} Let $\Delta \in 2^{-\N} \cap (0,\tfrac{1}{10}]$, $N \geq 1$, and $\eta > 0$. Write $\delta := \Delta^{N}$. Fix $\theta \in S^{1}$, and let $F \subset B(1)$ be a set with the following properties.

Let $\{a_{i}\}_{i = 0}^{n} \subset \{0,\ldots,N\}$ with $0 = a_{0} < a_{1} < \ldots < a_{n} = N$. Assume that for every $x \in F$ it holds
\begin{displaymath} \sum_{j \in \mathcal{B}(x)} (a_{j + 1} - a_{j}) \leq \eta N, \end{displaymath} 
where 
\begin{equation}\label{form24} \mathcal{B}(x) = \{0 \leq j \leq n - 1 : x \in H_{\theta}(F,(\Delta^{a_{j + 1} - a_{j}})^{-\sigma},[50\Delta^{a_{j + 1}},50\Delta^{a_{j}}])\}. \end{equation}
Then, for an absolute constant $C \geq 1$,
\begin{equation}\label{form49} |F_{5\delta} \cap \pi_{\theta}^{-1}\{t\}|_{5\delta} \leq C^{n}\delta^{-\sigma - \eta}, \qquad t \in \R, \end{equation} 
\end{lemma}

\subsubsection{Entropy} The proof of Lemma \ref{lemma5} uses the notion of entropy, so let us briefly recall the bits of information we need. If $\mu$ is a probability measure on some space $\Omega$, and $\mathcal{F}$ is a $\mu$ measurable partition of $\Omega$, the $\mathcal{F}$-entropy of $\mu$ is defined by
\begin{displaymath} H(\mu,\mathcal{F}) := \sum_{F \in \mathcal{F}} \mu(F)\log \tfrac{1}{\mu(F)}. \end{displaymath}
Here $0 \cdot \log \tfrac{1}{0} := 0$, and "$\log$" refers to logarithm in base $2$. The measures of interest to us are probability measures $\R$ or $\R^{2}$, and $\mathcal{F} = \mathcal{D}_{\Delta}$ for some $\Delta \in 2^{-\N}$. In addition to $\mathcal{F}$-entropy, we will also need the \emph{conditional entropy of $\mu$}:
\begin{displaymath} H(\mu,\mathcal{F} \mid \mathcal{E}) := \sum_{E \in \mathcal{E}} \mu(E)H(\mu_{E},\mathcal{F}). \end{displaymath}
Here $\mu_{E} := \mu(E)^{-1} \cdot \mu|_{E}$, and $\mathcal{E},\mathcal{F}$ are $\mu$ measurable partitions of $\Omega$. In all applications below, $\mathcal{F}$ \emph{refines} $\mathcal{E}$, or in other words every $E \in \mathcal{E}$ is a finite union of sets in $\mathcal{F}$. In this special case, the conditional entropy can be rewritten as
\begin{equation}\label{entropy} H(\mu,\mathcal{F} \mid \mathcal{E}) = H(\mu,\mathcal{F}) - H(\mu,\mathcal{E}), \end{equation}
We finally record that $H(\mu,\mathcal{F}) \leq \log |\mathcal{F}|$, as a consequence of Jensen's inequality.

\begin{proof}[Proof of Lemma \ref{lemma5}] Fix $t \in \R$, and let $\{y_{1},y_{2},\ldots,y_{m}\}$ be a maximal $(5\delta)$-separated subset of $F_{5\delta} \cap \pi_{\theta}^{-1}\{t\}$. The claim is that $m \leq C^{n}\delta^{-\sigma - \eta}$ for some absolute constant $C \geq 1$. For each $y_{j}$, find a point $x_{j} \in F$ with $|x_{j} - y_{j}| \leq 5\delta$. The points $x_{j}$ may not be $\delta$-separated, but they contain a $\delta$-separated subset of cardinality $\sim m$. So, we assume with no loss of generality that the points $x_{j}$ are $\delta$-separated, and in fact that each dyadic $\delta$-square contains at most one point $x_{j}$. We write $X := \{x_{1},\ldots,x_{m}\}$. Let us record that
\begin{displaymath} X \subset T := \pi_{\theta}^{-1}([t - 5\delta,t + 5\delta]). \end{displaymath}
Let $\eta$ be the uniformly distributed probability measure on $X \subset F \cap T$. Write $\mathcal{D}_{j} := \mathcal{D}_{\Delta^{a_{j}}}$ for the dyadic squares of side-length $\Delta^{a_{j}}$ in $\R^{2}$, so in particular 
\begin{displaymath} \mathcal{D}_{0} = \mathcal{D}_{\Delta^{a_{0}}} = \mathcal{D}_{1} \quad \text{and} \quad \mathcal{D}_{n} = \mathcal{D}_{\Delta^{a_{n}}} = \mathcal{D}_{\delta}. \end{displaymath}
Note that $H(\eta,\mathcal{D}_{0}) \leq \log |\{Q \in \mathcal{D}_{0} \cap B(1) \neq \emptyset\}| = 2$, since $\spt \eta \subset F \subset B(1)$. Then, using the assumption that every square in $\mathcal{D}_{n}$ contains at most one point from $X$,
\begin{align} \log m - 2 & \leq H(\eta,\mathcal{D}_{n}) - H(\eta,\mathcal{D}_{0}) \stackrel{\eqref{entropy}}{=} H(\eta,\mathcal{D}_{n} \mid \mathcal{D}_{0}) \stackrel{\eqref{entropy}}{=} \sum_{j = 0}^{n - 1} H(\eta,\mathcal{D}_{j + 1} \mid \mathcal{D}_{j}) \notag\\
&\label{form40} = \sum_{j = 0}^{n - 1} \sum_{Q \in \mathcal{D}_{j}} \eta(Q)H(\eta_{Q},\mathcal{D}_{j + 1}) = \int \sum_{Q \ni x} H(\eta_{Q},\mathcal{D}_{\mathrm{gen}(Q) + 1}) \, d\eta(x). \end{align}
Here $\eta_{Q} = \eta(Q)^{-1} \cdot \eta|_{Q}$. Fix 
\begin{displaymath} x \in \spt \eta = X \subset \pi_{\theta}^{-1}([t - 5\delta,t + 5\delta]), \end{displaymath}
and $Q \in \mathcal{D}_{j}$, $0 \leq j \leq n - 1$, with $x \in Q$. Then
\begin{displaymath} H(\eta_{Q},\mathcal{D}_{\mathrm{gen}(Q) + 1}) = H(\eta_{Q},\mathcal{D}_{j + 1}) \leq \log |\{Q_{j + 1} \in \mathcal{D}_{j + 1} : Q_{j + 1} \cap \spt \eta_{Q} \neq \emptyset\}|, \end{displaymath}
and we claim that
\begin{equation}\label{form39} |\{Q_{j + 1} \in \mathcal{D}_{j + 1} : Q_{j + 1} \cap \spt \eta_{Q} \neq \emptyset\}| \leq C \cdot \m_{F,\theta}(x \mid [50\Delta^{a_{j + 1}},50\Delta^{a_{j}}]) \end{equation}
for some absolute constant $C \geq 1$. Indeed, let $Q_{1}',\ldots,Q_{p}'$ be an enumeration of the squares in $\mathcal{D}_{j + 1}$ which intersect $\spt \eta_{Q}$, see Figure \ref{fig1}. Thus, for each $Q_{k}'$, there exists a point $x_{k}' \in \spt \eta_{Q} \cap Q_{k}' \subset F \cap Q \cap T$. On the other hand, by definition
\begin{displaymath} \m_{\theta}(x \mid [50\Delta^{a_{j + 1}},50\Delta^{a_{j}}]) = |B(x,50\Delta^{a_{j}}) \cap F_{50\Delta^{a_{j + 1}}} \cap \pi_{\theta}^{-1}\{\pi_{\theta}(x)\}|_{50\Delta^{a_{j + 1}}}. \end{displaymath}
Now, the key observation is that for each $x_{k}' \in Q_{k}'$, the line $\pi_{\theta}^{-1}\{\pi_{\theta}(x)\}$ intersects the ball $B(x_{k}',50\Delta^{a_{j + 1}})$ at some point 
\begin{displaymath} y_{k}' \in B(x,50\Delta^{a_{j}}) \cap F_{50\Delta^{a_{j + 1}}} \cap \pi_{\theta}^{-1}\{\pi_{\theta}(x)\}, \qquad 1 \leq k \leq p, \end{displaymath}
and in fact even in a segment of length $\sim \Delta^{a_{j + 1}}$, see Figure \ref{fig1}.
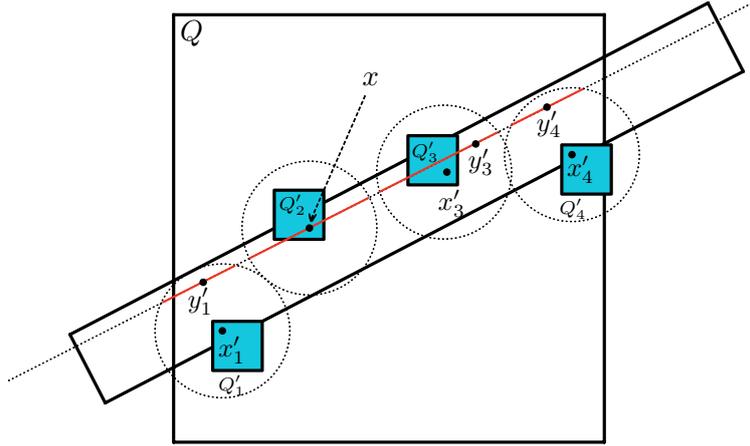
\begin{figure}[h!]
\begin{center}
\begin{overpic}[scale = 1]{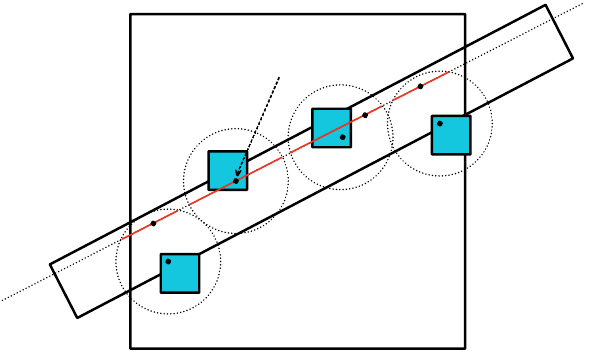}
\put(23,54){$Q$}
\put(24,18.5){\small{$y_{1}'$}}
\put(28,12){\small{$x_{1}'$}}
\put(28,7.5){\tiny{$Q_{1}'$}}
\put(47,48){$x$}
\put(36,31.5){\tiny{$Q_{2}'$}}
\put(53.5,38.5){\tiny{$Q_{3}'$}}
\put(73,31){\tiny{$Q_{4}'$}}
\put(61,37){\small{$y_{3}'$}}
\put(57,31.5){\small{$x_{3}'$}}
\put(74,36){\small{$x_{4}'$}}
\put(70,42){\small{$y_{4}'$}}
\end{overpic}
\caption{The squares $Q_{1}',\ldots,Q_{p}' \in \mathcal{D}_{j + 1}$ intersecting $\eta_{Q}$, the points $x_{1}',\ldots,x_{p}' \subset F \cap Q \cap T$ and the points $y_{1}',\ldots,y_{p}' \subset F_{50\Delta^{a_{j + 1}}} \cap \pi_{\theta}^{-1}\{\pi_{\theta}(x)\}$. The larger ball $B(x,50\Delta^{a_{j}})$ has not been drawn in the figure, but it is large enough to contain $Q \in \mathcal{D}_{j}$, and especially all the points $y_{1}',\ldots,y_{p}'$.}\label{fig1}
\end{center}
\end{figure}
Moreover, since $Q_{k}' \in \mathcal{D}_{j + 1}$, we have
\begin{displaymath} p \lesssim |\{y_{1}',\ldots,y_{p}'\}|_{50\Delta^{a_{j + 1}}} \leq |B(x,50\Delta^{a_{j}}) \cap F_{10\Delta^{a_{j + 1}}} \cap \pi_{\theta}^{-1}\{\pi_{\theta}(x)\}|_{50\Delta^{a_{j + 1}}}. \end{displaymath} This proves \eqref{form39}.

From \eqref{form40}-\eqref{form39}, we conclude that
\begin{equation}\label{form48} \log m \leq \int \sum_{Q \ni x} \log \m_{F,\theta}(x \mid [50\Delta^{a_{\mathrm{gen}(Q) + 1}},50\Delta^{a_{\mathrm{gen}(Q)}}]) \, d\eta(x) + \log C^{n}, \end{equation}
where the term "$2$" was absorbed into the term "$n\log C$". A trivial bound for $\m_{F,\theta}(x \mid [r,R]) = |B(x,R) \cap F \cap \pi_{\theta}^{-1}\{\pi_{\theta}(x)\}|_{r}$ is 
\begin{displaymath} \m_{F,\theta}(x \mid [r,R]) \leq C(R/r), \end{displaymath}
where $C \geq 1$ is an absolute constant. We will apply this bound for $j \in \mathcal{B}(x)$, recall the definition from \eqref{form24}. For $j \in \mathcal{G}(x) := \{0,\ldots,n - 1\} \, \setminus \, \mathcal{B}(x)$, we use the better estimate $\m_{F,\theta}(x \mid [50\Delta^{a_{j + 1}},50\Delta^{a_{j}}]) \leq \Delta^{-\sigma}$. Altogether,
 \begin{align*} \sum_{Q \ni x}  & \log \m_{F,\theta}(x \mid [50\Delta^{a_{\mathrm{gen}(Q) + 1}},50\Delta^{a_{\mathrm{gen}(Q)}}])\\
 & \leq \sum_{j \in \mathcal{G}} \log \Delta^{-\sigma(a_{j + 1} - a_{j})} + \sum_{j \in \mathcal{B}} \log C \cdot \Delta^{-(a_{j + 1} - a_{j})}\\
 & \leq n \log C + \sigma \log \tfrac{1}{\Delta} \sum_{j = 0}^{n - 1} (a_{j + 1} - a_{j}) + \log \tfrac{1}{\Delta} \sum_{j \in \mathcal{B}} (a_{j + 1} - a_{j})\\
 & \leq \log C^{n} + \sigma N \log \tfrac{1}{\Delta} + \eta N \log \tfrac{1}{\Delta} = \log C^{n} + (\sigma + \eta) \log \tfrac{1}{\delta}.  \end{align*}
 Recalling from the beginning of the proof that $m \sim |F_{5\delta} \cap \pi_{\theta}^{-1}\{t\}|_{5\delta}$, we have now shown 
 \begin{displaymath} |F_{5\delta} \cap \pi_{\theta}^{-1}\{t\}|_{5\delta} \lesssim C^{n}\delta^{-\sigma - \eta}, \end{displaymath}
 and the proof is complete.  \end{proof}
 
Lemma \ref{lemma4} shows that a "single-scale" upper bound on the size of high-multiplicity sets, which is true for all $(s,\mathbf{C})$-regular measures, implies a "multi-scale" upper bound for the size of high-multiplicity sets of $(s,\mathbf{C})$-regular measures $\mu$. The result follows by appling the single-scale bound to various \emph{renormalisations} $\mu^{B}$, defined below:
 \begin{definition}[Renormalised measure]\label{def:renormalisation} Let $\mathbf{C},s > 0$. Given an $(s,\mathbf{C})$-regular measure $\mu$ on $\R^{d}$, and a ball $B \subset \R^{d}$ with $\mathrm{rad}(B) = r$, we write
\begin{displaymath} \mu^{B} := r^{-s}T_{B}\mu. \end{displaymath}
Here $T_{B}$ is the affine map which sends $B$ to $B(1)$. Specifically, if $B = B(x_{0},r_{0})$, then $T_{B}$ is defined by $T_{B}(x) = (x - x_{0})/r_{0}$ for $x \in \R^{d}$. \end{definition} 

A key property of renormalisation is that it exactly preserves $(s,\mathbf{C})$-regularity.

\begin{lemma} Let $\mathbf{C},s > 0$, and let $\mu$ be an $(s,\mathbf{C})$-regular measure on $\R^{d}$. If $B \subset \R^{d}$ is an arbitrary ball, then $\mu^{B}$ is also $(s,\mathbf{C})$-regular.  The same conclusions hold if $(s,\mathbf{C})$-regularity is replaced by Ahlfors $(s,\mathbf{C})$-regularity. \end{lemma}

\begin{proof} Note that if $K = \spt \mu$, then $K^{B} := \spt \mu^{B} = T_{B}(K)$. Write $B = B(x_{0},r_{0})$ with $x_{0} \in \R^{2}$ and $r_{0} > 0$. Then, $K^{B}$ is upper $(s,\mathbf{C})$-regular, since
\begin{displaymath} N_{r}(K^{B} \cap B(x,R)) = N_{rr_{0}}(K \cap B(T^{-1}(x),r_{0}R)) \leq \mathbf{C}\left(\tfrac{r_{0}R}{r_{0}r} \right)^{s} = \mathbf{C}\left(\tfrac{R}{r} \right)^{s} \end{displaymath} 
for all $x \in \R^{d}$ and $0 < r \leq R < \infty$.

The $(s,\mathbf{C})$-Frostman property of $\mu^{B}$ follows from the next calculation:
\begin{displaymath} \mu^{B}(B(x,r)) = r_{0}^{-s} \cdot \mu(B(T^{-1}(x),rr_{0})) \leq r_{0}^{-s} \cdot \mathbf{C} (rr_{0})^{s} = \mathbf{C} r^{s}, \end{displaymath}
This completes the proof of the $(s,\mathbf{C})$-regular case. For the Ahlfors $(s,\mathbf{C})$-regular case, note that the previous displayed upper bound also holds as a lower bound for $x = T_{B}(x') \in \spt \mu^{B}$ (with constant $\mathbf{C}^{-1}$).   \end{proof}
 
 \begin{lemma}\label{lemma4} Let $A \geq 1$, $s \in [0,2]$, and $\Delta \in (0,\tfrac{1}{A}]$. Let $\mathbf{C},\epsilon,\sigma_{0} > 0$ be parameters such that the inequality
\begin{equation}\label{form63a} \int_{S^{1}} \mu(B(1) \cap H_{\theta}(K,\Delta^{-\sigma_{0}},[\Delta,1])) \, d\nu(\theta) \leq \epsilon \end{equation}
holds for all $(s,\mathbf{C})$-regular measures $\mu$ with $K = \spt \mu$, and some finite Borel measure $\nu$ on $S^{1}$. Let $M,N \in \N \, \setminus \, \{0\}$, and let $\mu$ be an $(s,\mathbf{C})$-regular measure with $K := \spt \mu \subset \R^{2}$.

Then, there exists a subset $S \subset S^{1}$ of $\nu$-measure 
\begin{displaymath} \nu(S) \geq \nu(S^{1}) - \left(\tfrac{\mathbf{C}}{M} \right)^{1/2}, \end{displaymath}
and for all $\theta \in S$ there exists a set $L_{\theta} \subset B(1) \cap K$ with $\mu(L_{\theta}) \geq \mu(B(1)) - \left(\tfrac{\mathbf{C}}{M} \right)^{1/2}$ such that
\begin{equation}\label{form30a} \tfrac{1}{N}|\left\{1 \leq j \leq N : x \in H_{\theta}(K,\Delta^{-\sigma_{0}},[A\Delta^{j + 1},A\Delta^{j}])\right\}| \leq M\epsilon, \qquad x \in L_{\theta}. \end{equation}
\end{lemma}

\begin{proof} We define the space 
\begin{displaymath} \Omega := \{1,\ldots,N\} \times [K \cap B(1)] \times S^{1}, \end{displaymath}
the measure $\tn := \tfrac{1}{N}\mathcal{H}^{0} \times \mu|_{B(1)} \times \nu$ on $\Omega$ (not necessarily a probability measure, despite the notation), and the subset
\begin{displaymath} H := \{(j,x,\theta) \in \Omega : x \in H_{\theta}(K,\Delta^{-\sigma_{0}},[A\Delta^{j + 1},A\Delta^{j}])\}. \end{displaymath}
After abbreviating $H_{\theta}(j) := H_{\theta}(K,\Delta^{- \sigma_{0}},[A\Delta^{j + 1},A\Delta^{j}])$, we claim that
\begin{equation}\label{form66a} \tn(H) = \frac{1}{N} \sum_{j = 1}^{N} \int_{S^{1}} \mu(B(1) \cap H_{\theta}(j)) \, d\nu(\theta) \leq \mathbf{C} \epsilon. \end{equation}
The first equation follows immediately from the definition of $H$ and $\tn$. The second inequality might look like a consequence of the assumption \eqref{form63a}, but the set $H_{\theta}(j)$ is not exactly the one which appears in \eqref{form63a}. To bring the correct set out, let $\mathcal{K}_{A\Delta^{j}} =: \mathcal{K}_{j}$ be a minimal cover of $K \cap B(1)$ by discs of radius $A\Delta^{j}$ (note that $A\Delta^{j} \leq 1$ for $j \in \{1,\ldots,N\}$). By the $(s,\mathbf{C})$-regularity of $\mu$, we have 
\begin{equation}\label{form6a} |\mathcal{K}_{j}| = N_{A\Delta^{j}}(K \cap B(1)) \leq \tfrac{\mathbf{C}}{A^{s}} \cdot \Delta^{-js}, \qquad 1 \leq j \leq N. \end{equation}
Then, for $1 \leq j \leq N$ fixed, we start by decomposing
\begin{displaymath} \int_{S^{1}} \mu(B(1) \cap H_{\theta}(j)) \, d\nu(\theta) \leq \sum_{B_{j} \in \mathcal{K}_{j}} \int_{S^{1}} \mu(B_{j} \cap H_{\theta}(j)) \, d\nu(\theta). \end{displaymath}
To estimate the terms on the right individually, fix $B_{j} \in \mathcal{K}_{j}$, and let $T_{B_{j}} \colon B_{j} \to B(1)$ be the rescaling map. Since $\mathrm{rad}(B_{j}) = A\Delta^{j}$, we may apply Lemma \ref{lemma7} to deduce that
\begin{displaymath} B_{j} \cap H_{\theta}(j) = T_{B_{j}}^{-1}[B(1) \cap H_{\theta}(T_{B_{j}}(K),\Delta^{-\sigma_{0}},[\Delta,1])], \qquad \theta \in S^{1}. \end{displaymath} 
Therefore, denoting by $\mu^{B_{j}} := (A\Delta^{j})^{-s} \cdot T_{B_{j}}\mu$ be the renormalised measure associated with $B_{j}$ (as in Definition \ref{def:renormalisation}), we find
\begin{displaymath} \mu(B_{j} \cap H_{\theta}(j)) = (A\Delta^{j})^{s} \cdot \mu^{B_{j}}(B(1) \cap H_{\theta}(T_{B_{j}}(K),\Delta^{-\sigma_{0}},[\Delta,1])), \qquad \theta \in S^{1}. \end{displaymath} 
The measure $\mu^{B_{j}}$ is $(s,\mathbf{C})$-regular, and $\spt \mu^{B_{j}} = T_{B_{j}}(K)$, so \eqref{form63a} yields
\begin{displaymath} \int_{S^{1}} \mu(B_{j} \cap H_{\theta}(j)) \, d\nu(\theta) = (A\Delta^{j})^{s} \int \mu^{B_{j}}(B(1) \cap H_{\theta}(T_{B_{j}}(K),\Delta^{-\sigma_{0}},[\Delta,1])) \, d\nu(\theta) \leq (A\Delta^{j})^{s}\epsilon \end{displaymath}
for all $B_{j} \in \mathcal{K}_{j}$. Combining this estimate with \eqref{form6a},
\begin{displaymath} \int_{S^{1}} \mu(B(1) \cap H_{\theta}(j)) \, d\nu(\theta) \leq \sum_{B_{j} \in \mathcal{K}_{j}} \int_{S^{1}} \mu(B_{j} \cap H_{\theta}(j)) \, d\nu(\theta) \leq \mathbf{C} \epsilon. \end{displaymath}
This proves \eqref{form66a} after taking an average over $j \in \{1,\ldots,N\}$.

Now, fix a constant $M \geq 1$, and for $\theta \in S^{1}$, consider the set
\begin{displaymath} L_{\theta} := \left\{x \in K : \tfrac{1}{N}|\{1 \leq j \leq N : x \in H_{\theta}(j)\}| \leq M\epsilon\right\}. \end{displaymath}
Then,
\begin{displaymath} \int_{S^{1}} \mu(B(1) \, \setminus \, L_{\theta}) \, d\nu(\theta) \leq  \int_{S^{1}} \int_{B(1)} \frac{1}{MN\epsilon}\sum_{j = 1}^{N} \mathbf{1}_{H_{\theta}(j)}(x) \, d\mu(x) \, d\nu(\theta) = \frac{\tn(H)}{M\epsilon} \stackrel{\eqref{form66a}}{\leq} \frac{\mathbf{C}}{M}. \end{displaymath}
It follows that 
\begin{displaymath} \nu(\{\theta \in S^{1} : \mu(B(1) \, \setminus \, L_{\theta}) \geq \left(\tfrac{\mathbf{C}}{M} \right)^{1/2}\}) \leq \left(\tfrac{\mathbf{C}}{M} \right)^{1/2}. \end{displaymath}
Therefore, there exists a subset $S \subset S^{1}$ of measure 
\begin{displaymath} \nu(S) \geq \nu(S^{1}) - \left(\tfrac{\mathbf{C}}{M} \right)^{1/2} \quad \text{ such that } \quad \mu(L_{\theta}) \geq \mu(B(1)) - \left(\tfrac{\mathbf{C}}{M} \right)^{1/2} \text{ for } \theta \in S. \end{displaymath}
This is what we claimed. \end{proof}

\subsection{A Lebesgue differentiation lemma on $[0,1]$} If $E \subset [0,1]$ is a Lebesgue measurable set of measure $|E| = \epsilon < 1$, then by the Lebesgue differentiation theorem, almost all of $[0,1] \, \setminus \, E$ can be covered by intervals $J$ where the density of $E$ is arbitrarily low, and in particular $|E \cap I| \lesssim \epsilon |I|$ for all sub-intervals $I \subset J$ of length $|I| \sim |J|$. The next lemma is a more quantitative version of this fact.

\begin{lemma}\label{lemma2} For every $\gamma \in 2^{-\N} \cap (0,\tfrac{1}{2}]$ and $C \geq 1$ there exists $\rho = \rho(C,\gamma) > 0$ such that the following holds. Let $\epsilon > 0$, and let $E \subset [0,1]$ be a Lebesgue measurable set of measure $|E| \leq \epsilon$. 

Then, there exists a set $G \subset [0,1]$ of measure $|G| \geq 1 - \tfrac{1}{C}$, which is a union of finitely many disjoint dyadic intervals $J \subset [0,1]$, of length $|J| \geq \rho$, with the property that $|E \cap I| < (8C\epsilon) |I|$ for all sub-intervals $I \subset J$ with $|I| \geq \gamma |J|$. \end{lemma}

\begin{proof} We first define the constants. Let $n(C,\gamma) \in \N$ be the smallest integer such that 
\begin{equation}\label{form19} (1 - \gamma)^{n(C,\gamma)} \leq \tfrac{1}{4C}. \end{equation}
Next, let $\rho = \rho(C,\gamma) > 0$ be so small that $\gamma^{-(n(C,\gamma) + 1)}\rho \leq \tfrac{1}{4C}$.

We perform a construction for $n = 0,1,\ldots$ where we maintain sets $B_{n},G_{n},S_{n}$, and $T_{n}$. The letters stand for bad, good, short, and transient. For every $n \geq 0$, the sets $B_{n},G_{n},S_{n}$ and $T_{n}$ partition $[0,1)$. Initially $T_{n} = [0,1)$, while the other sets are empty. 

Assume that the sets $B_{n},G_{n},S_{n},T_{n}$ have been constructed for some $n \geq 0$, and are unions of dyadic intervals such that the intervals constituting $S_{n}$ have length $< \rho$. 

We will now construct the sets $B_{n + 1},G_{n + 1},S_{n + 1},T_{n + 1}$. In fact, we immediately include the set $B_{n}$ into $B_{n + 1}$, and $G_{n}$ into $G_{n + 1}$, and $S_{n}$ into $S_{n + 1}$. So, it only remains to describe how to (re-)distribute the points of $T_{n}$ into the sets $B_{n + 1},G_{n + 1},S_{n + 1}$, and $T_{n + 1}$. For this purpose, let $\mathcal{T}_{n}$ the intervals which constitute $T_{n}$. 

Fix $J \in \mathcal{T}_{n}$. If there are \textbf{no} dyadic sub-intervals $I \subset J$ with 
\begin{equation}\label{form21} |I| \geq \gamma|J| \quad \text{and} \quad |E \cap I| \geq (2C\epsilon)|I|, \end{equation}
we include $J$ into $G_{n + 1}$, thus $J \cap T_{n + 1} = \emptyset$. In the opposite case, there exists at least one dyadic sub-interval $I \subset J$ with $|I| \geq \gamma|J|$ and $|E \cap I| \geq (2C\epsilon)|I|$. We include $I$ into $B_{n + 1}$. The set $J \, \setminus \, I$ can now be expressed as a union of dyadic intervals of length (exactly) $\gamma |J|$. If $\gamma |J| < \rho$, we include all of these intervals into $\mathcal{S}_{n}$. Otherwise they are included into $T_{n + 1}$. This completes the construction of the sets $B_{n + 1},G_{n + 1},S_{n + 1},T_{n + 1}$.

We now claim that $|T_{n}| \leq (1 - \gamma)^{n}$ for all $n \geq 0$. This is clear for $n = 0$. The inductive step follows from the inequality
\begin{displaymath} |J \cap T_{n + 1}| \leq (1 - \gamma)|J|, \qquad J \in \mathcal{T}_{n}, \end{displaymath}
which is clear by inspecting the construction of $T_{n + 1}$ above. 

Another easy fact is that $|\mathcal{T}_{n}| \leq \gamma^{-n}$ for all $n \geq 0$. As a consequence, let us check that $S_{n}$ consists of at most $\gamma^{-(n + 1)}$ intervals. Indeed, if $\mathcal{S}_{n}$ is the collection of intervals constituting $S_{n}$, then we may assume (inductively) that $|\mathcal{S}_{n}| \leq \gamma^{-(n + 1)}$, and we note that $\mathcal{S}_{n + 1}$ consists of all the intervals in $\mathcal{S}_{n}$, plus additionally $\leq \gamma^{-1}$ new intervals for each element of $\mathcal{T}_{n}$. Thus, 
\begin{displaymath} |\mathcal{S}_{n + 1}| \leq |\mathcal{S}_{n}| + \gamma^{-1}|\mathcal{T}_{n}| \leq \gamma^{-(n + 1)} + \gamma^{-1} \cdot \gamma^{-n} = 2 \cdot \gamma^{-(n + 1)} \leq \gamma^{-(n + 2)}. \end{displaymath}
Since all the intervals in $\mathcal{S}_{n}$ have length $< \rho$, we may now infer that
\begin{displaymath} 1 - |B_{n}| = |G_{n} \cup S_{n} \cup T_{n}| \leq |G_{n}| + \gamma^{-(n + 1)}\rho + (1 - \gamma)^{n}, \qquad n \in \N. \end{displaymath}
In particular, this estimate is true for $n = n(C,\gamma) \in \N$ defined at \eqref{form19}, and for this index
\begin{equation}\label{form20} 1 - |B_{n}| \leq |G_{n}| + \tfrac{1}{2C} \quad \Longrightarrow \quad |G_{n}| \geq 1 - \tfrac{1}{2C} - |B_{n}|. \end{equation} 
Now, a key observation is that, on the other hand, $|B_{n}| \leq \tfrac{1}{2C}$. This is because $B_{n}$ is a union of disjoint (up to endpoints) intervals $I \subset [0,1]$ such that $|E \cap I| \geq (2C\epsilon)|I|$. Since $|E| \leq \epsilon$, we have $|B_{n}| \leq \tfrac{1}{2C}$. Inserting this bound to \eqref{form20}, we find $|G_{n}| \geq 1 - \tfrac{1}{C}$. 

Finally, by construction, $G_{n}$ consists of intervals $J \subset [0,1]$ of length $\geq \rho$ such \eqref{form21} holds for \textbf{no} dyadic subintervals $I \subset J$. This completes the proof. (The original claim about arbitrary intervals $I \subset J$ vs. dyadic intervals follows by covering $I$ by two dyadic intervals of length $\leq 2|I|$). \end{proof} 

\subsection{Finding branching scales for Frostman measures} The idea of the technical lemma in this section is the following. If $\nu$ is a Borel measure on $\R^{d}$ satisfying $\nu(B(x,r)) \lesssim r^{\tau}$, then for some scales $\delta > 0$ it may happen that $\nu$ has "no branching" in the sense that, say, $\nu(B(x,\delta)) \sim \nu(B(x,\delta^{2}))$ for $x \in \spt \nu$. However, the Frostman property $\nu(B(x,r)) \lesssim r^{\tau}$ prevents this from happening on "many" consecutive scales, where "many" depends only on $\tau > 0$. In Lemma \ref{lemma8a}, this fact is quantified by fixing a long sequence of scales $\{\delta_{j}\}$ of a special \emph{a priori} form, which depends only on $\tau$, and another technical parameter $\mathfrak{d} > 1$. The conclusion is then (roughly speaking) that every measure $\nu$ with $\nu(B(x,r)) \lesssim r^{\tau}$ must have significant branching on at least one scale in $\{\delta_{j}\}$.

To describe the sequence $\{\delta_{j}\}$ we find useful, we pose the following definition:

\begin{definition}[$(\tau,\mathfrak{d})$-numbers]\label{def:tauRationals} Let $\tau > 0$ and $\mathfrak{d} > 1$. Let $\mathfrak{n} \lesssim_{\mathfrak{d},\tau} 1$ be the smallest natural number satisfying
\begin{equation}\label{form65a} \frac{1}{2\mathfrak{d}^{\mathfrak{n} - 1}} \leq \frac{\tau}{4}. \end{equation}
 Then, let $\mathcal{Q}_{0}(\mathfrak{d},\tau) = \{\tfrac{1}{2}\mathfrak{d}^{-j} : 0 \leq j \leq \mathfrak{n}\}$.  \end{definition} 
 
 \begin{lemma}\label{lemma8a} For every $\tau > 0$ and $\mathfrak{d} > 1$ there exists $\eta > 0$ and $\delta_{0} = \delta_{0}(\mathfrak{d},\tau) > 0$ such that the following holds for all $\delta \in (0,\delta_{0}]$.
 
 Let $\nu$ be a Borel probability measure on $S^{1}$ satisfying $\nu(B(x,r)) \leq \delta^{-\eta}r^{\tau}$ for all $x \in S^{1}$ and $r > 0$. Then, there exists $p \in \mathcal{Q}_{0}(\mathfrak{d},\tau)$ and a subset $G \subset S^{1}$ with $\nu(G) \geq \tau^{2}/(150\mathfrak{n}^{2}) \gtrsim_{\mathfrak{d},\tau} 1$ such that if $I \in \mathcal{D}_{\delta^{p}}(S^{1})$ has $\nu(I \cap G) > 0$, then
\begin{equation}\label{form68a} \nu(J \cap G) \leq \delta^{\tau/(20\mathfrak{n})}\nu(I \cap G), \qquad J \in \mathcal{D}_{\delta^{\mathfrak{d} p}}(I). \end{equation} 
 \end{lemma} 
 
 This lemma is actually \cite[Lemma 4.17]{MR4388762}, but since that lemma was not formulated as a "stand-alone" result, we repeat the proof here. 
 
 \begin{proof}[Proof of Lemma \ref{lemma8a}] We define the following increasing scale sequence: $\delta_{0} := \delta^{\mathfrak{d}/2}$, and
 \begin{displaymath} \delta_{j + 1} := \delta_{j}^{1/\mathfrak{d}}, \qquad 0 \leq j \leq \mathfrak{n}. \end{displaymath}
 Thus actually $\{\delta_{j}\}_{j = 0}^{\mathfrak{n}} = \{\delta^{p}\}_{p \in \mathcal{Q}_{0}(\mathfrak{d},\tau)}$. Note that, a little unusually, the sequence $\{\delta_{j}\}$ is increasing in $j$. Below, we identify $S^{1} \cong [0,1)$ in the sense that we use dyadic intervals on $S^{1}$. We assume with no loss of generality that $\delta > 0$, as in the statement, has the property that $\delta_{j} \in 2^{-\N}$ for all $0 \leq j \leq \mathfrak{n}$. It is easy to check that the general case of Lemma \ref{lemma8a} can be reduced to the special case with this extra assumption.
 
 Regarding the constant $\eta = \eta(\tau,\mathfrak{d}) > 0$, the only requirement is that
 \begin{equation}\label{form67a} \eta < \tfrac{\tau(\mathfrak{d} - 1)}{2}. \end{equation}
 
 Using the hypothesis $\nu(B(x,r)) \leq \delta^{-\eta}r^{\tau}$, we have the uniform bound
\begin{displaymath} \nu(I) \leq \delta^{-\eta}\delta_{0}^{\tau} \qquad I \in \mathcal{D}_{\delta_{0}}, \end{displaymath} 
which gives the following lower bound for the entropy of $\nu$ at scale $\delta_{0}$:
\begin{equation}\label{form21c} H(\nu,\mathcal{D}_{\delta_{0}}) = \sum_{I \in \mathcal{D}_{\delta_{0}}} \nu(I) \log \tfrac{1}{\nu(I)} \geq \log \delta_{0}^{-\tau} - \log \delta^{-\eta} \geq \log \delta^{-\tau/2}. \end{equation}
In the final inequality, we used that $\delta_{0}^{-\tau} = \delta^{-\tau \mathfrak{d}/2} = \delta^{-\tau/2} \cdot \delta^{-(\mathfrak{d} - 1) \tau/2}$, and the logarithm of the second factor exceeds the error term $\log \delta^{-\eta}$ if $\delta > 0$ is sufficiently small.

On the other hand, observing that $\mathcal{D}_{\delta_{\mathfrak{n}}}$ consists of intervals of length $\delta^{1/[2\mathfrak{d}^{n - 1}]}$,
\begin{displaymath} H(\nu,\mathcal{D}_{\delta_{\mathfrak{n}}}) \leq \log |\mathcal{D}_{\delta_{\mathfrak{n}}}| \leq \log \delta^{-1/[2\mathfrak{d}^{n - 1}]} \stackrel{\eqref{form65a}}{\leq} \log \delta^{-\tau/4}. \end{displaymath} 
Combining this estimate with \eqref{form21c},
\begin{displaymath} \log \delta^{-\tau/4} \leq H(\nu,\mathcal{D}_{\delta_{0}}) - H(\nu,\mathcal{D}_{\delta_{\mathfrak{n}}}) \stackrel{\eqref{entropy}}{=} \sum_{j = 0}^{\mathfrak{n} - 1} H(\nu,\mathcal{D}_{\delta_{j}} \mid \mathcal{D}_{\delta_{j + 1}}). \end{displaymath}
Consequently, we may fix an index $j \in \{0,\ldots,\mathfrak{n} - 1\}$ with the property
\begin{equation}\label{form23c} \sum_{I \in \mathcal{D}_{\delta_{j + 1}}} \nu(I) \cdot H(\nu_{I},\mathcal{D}_{\delta_{j}}) = H(\nu,\mathcal{D}_{\delta_{j}} \mid \mathcal{D}_{\delta_{j + 1}}) \geq \log \delta^{-\tau/(4\mathfrak{n})}. \end{equation} 
Here $\nu_{I} = \nu(I)^{-1} \nu|_{I}$ for any $I \in \mathcal{D}_{\delta_{j + 1}}$ with $\nu(I) \neq 0$. Motivated by \eqref{form23c}, set
\begin{equation}\label{form69a} \bar{\tau} := \frac{\tau}{4\mathfrak{n}} \gtrsim_{\tau,\mathfrak{d}} 1.  \end{equation}
The choice of the index $j \in \{0,\ldots,\mathfrak{n} - 1\}$ at \eqref{form23} roughly tells us that for many "long" intervals $I \in \mathcal{D}_{\delta_{j + 1}}$ with $\nu(I) > 0$, the re-normalised restriction $\nu_{I}$ is not concentrated on few "short" sub-intervals of $I$ in the family $\mathcal{D}_{\delta_{j}}$. By restricting $\nu$ a bit, we may replace the word "many" by "all", as we will see next. We claim that there exists a collection of intervals $\mathcal{G}_{j + 1} \subset \mathcal{D}_{\delta_{j + 1}}$ with
\begin{equation}\label{form239} \nu(\cup \mathcal{G}_{j + 1}) \geq \bar{\tau}/2 \end{equation} 
with the properties $\nu(I) > 0$ and 
\begin{equation}\label{form24c} H(\nu_{I},\mathcal{D}_{\delta_{j}}) \geq \log \delta^{-\bar{\tau}/2}, \qquad I \in \mathcal{G}_{j + 1}. \end{equation} Indeed, if this failed, then, using the trivial bound $H(\nu_{I},\mathcal{D}_{\delta_{j}}) \leq \log \delta^{-1}$,
\begin{align*} \sum_{I \in \mathcal{D}_{\delta_{j + 1}}} \nu(I) H(\nu_{I},\mathcal{D}_{\delta_{j}}) & \leq \sum_{H(\ldots) \geq \log \delta^{-\bar{\tau}/2}} \nu(I) \cdot \log \delta^{-1} + \sum_{H(\ldots) < \log \delta^{-\bar{\tau}/2}} \nu(I) \cdot \log \delta^{-\bar{\tau}/2}\\
& < \frac{\bar{\tau}}{2} \cdot \log \delta^{-1} + \log \delta^{-\bar{\tau}/2} = \log \delta^{-\bar{\tau}}, \end{align*} 
contradicting \eqref{form23c}. 

Now, \eqref{form24c} means that $\nu_{I} = \nu(I)^{-1}\nu|_{I}$ cannot be fully concentrated inside any single interval of length $\delta_{j} = |I|^{\mathfrak{d}}$. We want, \emph{a fortiori}, that $\nu_{I}$ satisfies a Frostman bound at the smaller scale $\delta_{j}$, see \eqref{form68a}. This will be achieved by choosing a suitable subset $G_{I} \subset I$.

Fix $I \in \mathcal{G}_{j + 1}$, write $\mathcal{D}_{\delta_{j}}(I) := \{J \in \mathcal{D}_{\delta_{j}} : J \subset I\}$, and then estimate
\begin{displaymath} \log \delta^{-\bar{\tau}/2} \stackrel{\eqref{form24c}}{\leq} \sum_{J \in \mathcal{D}_{\delta_{j}}(I)} \nu_{I}(J)\log \tfrac{1}{\nu_{I}(J)} \leq \sum_{\nu_{I}(J) \geq \delta^{\bar{\tau}/4}} \nu_{I}(J) \log \delta^{-\bar{\tau}/4} + \sum_{\nu_{I}(J) < \delta^{\bar{\tau}/4}} \nu_{I}\log \tfrac{1}{\nu_{I}(J)}. \end{displaymath}
Since $\nu_{I}(S^{1}) = 1$, the first term is bounded from above by $\tfrac{1}{2} \log \delta^{-\bar{\tau}/2}$, and consequently the second term has the lower bound
\begin{equation}\label{form22c} \sum_{\nu_{I}(J) < \delta^{\bar{\tau}/4}} \nu_{I}(J)\log \tfrac{1}{\nu_{I}(J)} \geq \log \delta^{-\bar{\tau}/4}. \end{equation}
Write $\mathcal{G}_{I} := \{I \in \mathcal{D}_{\delta_{j}}(I) : \nu_{I}(J) < \delta^{\bar{\tau}/4}\}$ and $G_{I} := \cup \mathcal{G}_{I}$. Then,
\begin{align*} \log \delta^{-\bar{\tau}/4} & \stackrel{\eqref{form22c}}{\leq} \sum_{J \in \mathcal{G}_{I}} \nu_{I}(J)\log \tfrac{1}{\nu_{I}(J)} = \nu_{I}(G_{I})\sum_{J \in \mathcal{G}_{I}} \frac{\nu_{I}(J)}{\nu_{I}(G_{I})} \log \tfrac{1}{\nu_{I}(J)}\\
&  \leq \nu_{I}(G_{I}) \log \left(\sum_{J \in \mathcal{G}_{I}} \tfrac{1}{\nu_{I}(G_{I})} \right) \leq \nu_{I}(G_{I}) \log |\mathcal{G}_{I}| \leq \delta^{-\nu_{I}(G_{I})} \end{align*}
by Jensen's inequality applied to the discrete probability measure $J \mapsto \nu_{I}(J)/m_{I}$ on $\mathcal{G}_{I}$, and since $|\mathcal{G}_{I}| \leq |\mathcal{D}_{\delta_{j}}(I)| \leq \delta^{-1}$. Consequently, 
 $\nu_{I}(G_{I}) \geq \bar{\tau}/4$. Define finally
\begin{displaymath} G := \bigcup_{I \in \mathcal{G}_{j + 1}} G_{I}. \end{displaymath}
Then,
\begin{displaymath} \tfrac{1}{\nu(G_{j + 1})} \cdot \nu(G) = \nu(G) = \sum_{I \in \mathcal{G}_{j + 1}} \nu(I) \cdot \nu_{I}(G_{I}) \geq \bar{\tau}/4, \end{displaymath}
so $\nu(G) \geq (\bar{\tau}/4) \cdot \nu(G_{j + 1}) \geq \bar{\tau}^{2}/8 \geq \tau^{2}/(150\mathfrak{n}^{2})$ according to \eqref{form239} and \eqref{form69a}.
Further, note that if $I \in \mathcal{D}_{\delta_{j + 1}}$ with $\nu(I \cap G) \neq 0$, then $I \in \mathcal{G}_{j + 1}$, and therefore
\begin{displaymath} \nu(J \cap G) = \nu(I) \cdot \nu_{I}(J \cap G) \leq \delta^{\bar{\tau}/4}\nu(I) \leq \tfrac{4}{\bar{\tau}} \cdot \delta^{\bar{\tau}/4}\nu(I \cap G) , \qquad J \in \mathcal{D}_{\delta_{j}}(I). \end{displaymath}
This completes the proof, noting that $\delta_{j + 1} = \delta^{p}$ for some $p \in \mathcal{Q}_{0}(\mathfrak{d},\tau)$, and $\delta_{j} = \delta^{\mathfrak{d}p}$. \end{proof}

\section{The main proposition}\label{s:statements}

The purpose of this section is to complete the proof of Theorem \ref{mainThm}, or more precisely to reduce it to Proposition \ref{mainProp}. To state that proposition, we introduce the following notation. Let $\mathbf{C} \geq 1$ and $s,\tau \in (0,2]$. For $\delta,\sigma > 0$, write
\begin{displaymath} \iota(\sigma)[\delta] := \iota(\mathbf{C},s,\sigma,\tau)[\delta] := \sup_{\mu,\nu} \int_{S^{1}} \mu(B(1) \cap H_{\theta}(K,\delta^{-\sigma},[\delta,1])) \, d\nu(\theta), \end{displaymath}
where the "$\sup$" runs over all $(s,\mathbf{C})$-regular measures $\mu$ on $\R^{2}$ and $(\tau,\mathbf{C})$-Frostman measures on $\nu$ on $S^{1}$. While $\iota(\sigma)[\delta]$ evidently depends on $\mathbf{C},s,\tau$, these quantities will remain fixed throughout the proof of Proposition \ref{mainProp}, whereas $\sigma$ and $\delta$ will vary.

\begin{proposition}\label{mainProp} Let $0 < \sigma < \sigma_{0}$ with $\sigma \leq s$ and and $\sigma_{0}(1 - \tfrac{\sigma}{4}) < \sigma$. Assume that for every $\epsilon_{0} > 0$ there exists $\Delta_{0} > 0$ such that $\iota(\sigma_{0})[\Delta] \leq \epsilon_{0}$ for all $\Delta \in (0,\Delta_{0}]$.

Then, for every $\epsilon > 0$, there exists $\delta_{0} > 0$ such that $\iota(\sigma)[\delta] \leq \epsilon$ for all $\delta \in (0,\delta_{0}]$.
\end{proposition} 

\begin{remark} The proof of Proposition \ref{mainProp} is effective in principle. Assuming that we have access to information about which $\Delta_{0}$ works for given $\epsilon_{0}$, it is possible, with some effort, to answer the question: which threshold $\delta_{0}$ works for given $\epsilon > 0$? \end{remark} 

We then complete the proof of Theorem \ref{mainThm}:

\begin{proof}[Proof of Theorem \ref{mainThm} assuming Proposition \ref{mainProp}] Fix $\mathbf{C},\tau > 0$ and $s \in [0,1]$. Let
\begin{displaymath} \Sigma(\mathbf{C},s,\tau) := \inf \{\sigma > 0 : \forall \epsilon > 0 \, \exists \delta_{0} > 0 \text{ such that } \iota(\mathbf{C},s,\sigma,\tau)[\delta] \leq \epsilon \text{ for all } \delta \in (0,\delta_{0}]\}. \end{displaymath}
We claim that $\Sigma(\mathbf{C},s,\tau) = 0$. Let us check that this proves Theorem \ref{mainThm}. Indeed, for Theorem \ref{mainThm}, we need to fix $\mathbf{C},\epsilon,s,\sigma > 0$, and prove that $\iota(\mathbf{C},s,\sigma,\epsilon)[\delta] \leq \epsilon$ for all $\delta > 0$ small enough. But this follows from $\Sigma(\mathbf{C},s,\epsilon) = 0 < \sigma$.

The proof that $\Sigma(\mathbf{C},s,\tau) = 0$ could be accomplished by a direct argument, but let us use a counter assumption: $\Sigma(\mathbf{C},s,\tau) > 0$. Clearly $\Sigma(\mathbf{C},s,\tau) \leq s$, because if $\sigma > s$, then $|K \cap B(x,1)|_{\delta} \leq \mathbf{C}\delta^{-s}$ (for every $x \in \R^{2}$) implies that
\begin{displaymath} H_{\theta}(K,\delta^{-\sigma},[\delta,1]) = \emptyset \end{displaymath}
for all $\delta > 0$ sufficiently small. In particular then $\iota(\mathbf{C},s,\sigma,\tau)[\delta] = 0$.

It is therefore possible to fix $0 < \sigma < \Sigma(\mathbf{C},s,\tau) < \sigma_{0}$ such that $\sigma \leq s$, and $\sigma_{0}(1 - \tfrac{\sigma}{4}) < \sigma$. Since $\sigma_{0} > \Sigma(\mathbf{C},s,\tau)$, for all $\epsilon_{0} > 0$ there exists $\Delta_{0} > 0$ such that $\iota(\mathbf{C},s,\sigma_{0},\tau)[\Delta] \leq \epsilon_{0}$ for all $\Delta \in (0,\Delta_{0}]$. But now Proposition \ref{mainProp} implies that for all $\epsilon > 0$ there exists $\delta_{0} > 0$ such that $\iota(\mathbf{C},s,\sigma,\tau)[\delta] \leq \epsilon$ for all $\delta \in (0,\delta_{0}]$. But then $\Sigma(\mathbf{C},s,\tau) \leq \sigma$ by definition, which contradicts our choice $\sigma < \Sigma(\mathbf{C},s,\tau)$. \end{proof}

The proof of Proposition \ref{mainProp} will mostly be accomplished in Section \ref{s:mainProof}, but it requires a major auxiliary result, Proposition \ref{mainTechnicalProp}. To make the statement of that proposition appear more readable, we briefly initialise the proof of Proposition \ref{mainProp}: this way we will already introduce the objects to which Proposition \ref{mainTechnicalProp} will be applied. We will recall the beginning of the proof of Proposition \ref{mainProp}, once more, at the head of Section \ref{s:mainProof}.

The proof of Proposition \ref{mainProp} proceeds by making a counter assumption: for some $\epsilon > 0$, there are arbitrarily small scales $\delta > 0$ such that $\iota(\sigma)[\delta] \geq \epsilon$. To spell out the definition, for arbitrarily small $\delta > 0$, we may find an $(s,\mathbf{C})$-regular measure $\mu$ on $\R^{2}$, and a $(\tau,\mathbf{C})$-Frostman measure $\nu$ on $S^{1}$ such that 
\begin{equation}\label{form22} \int_{S^{1}} \mu(B(1) \cap H_{\theta}(K,\delta^{-\sigma},[\delta,1])) \, d\nu(\theta) \geq \epsilon. \end{equation}
In particular, there exist subsets $S_{1} \subset S^{1}$ with $\nu(S_{1}) \geq \tfrac{\epsilon}{2}$ such that 
\begin{equation}\label{form65} \mu(B(1) \cap H_{\theta}(K,\delta^{-\sigma},[\delta,1])) \geq \tfrac{\epsilon}{2}, \qquad \theta \in S_{1}. \end{equation}
We then fix a constant $\lambda = \lambda(\sigma,\sigma_{0},\tau) > 0$, depending only on $\sigma,\sigma_{0},\tau$. The precise value is not important yet, but will be fixed at \eqref{choiceLambda}. We then apply the hypothesis of Proposition \ref{mainProp} with constant
\begin{equation}\label{form35} \epsilon_{0} := \frac{\lambda\epsilon^{2}}{16 \mathbf{C}} > 0 \end{equation} 
to find a scale $\Delta \in (0,\tfrac{1}{500}]$ such that $\Delta$ is small in a manner depending on $\mathbf{C},\sigma,\sigma_{0},\tau$ (the precise requirement is not yet important, but will be discussed right below \eqref{form133}), and
\begin{displaymath} \sup_{\bar{\mu}} \int \bar{\mu}(B(1) \cap H_{\theta}(K,\Delta^{-\sigma_{0}},[\Delta,1])) \, d\nu(\theta) \leq \epsilon_{0}, \end{displaymath}
where the $\sup$ runs over $(s,\mathbf{C})$-regular measures $\bar{\mu}$ on $\R^{2}$. With these choices, we claim to find a contradiction to \eqref{form22}.

\begin{remark}  We may assume with no loss of generality that $\delta = \Delta^{N}$ for some large integer $N \in \N$. Why? Note that the choice of $\Delta$ only depends on $\epsilon,\sigma,\sigma_{0},\tau$. When we write "$\delta = \Delta^{N}$ without loss of generality", we actually mean that if we manage to deduce a contradiction from \eqref{form22} for all $\delta = \Delta^{N}$ with $N \geq N_{0}$, say, then we do the same for all $\delta > 0$ sufficiently small. To see this, fix $0 < \delta \leq \Delta^{N_{0}}$ arbitrary, assume \eqref{form22} for this $\delta$, and let $N \in \N$ such that $\Delta^{N + 1} \leq \delta \leq \Delta^{N}$. Then 
\begin{equation}\label{form72a} H_{\theta}(K,\delta^{-\sigma},[\delta,1]) \subset H_{\theta}(K,\Delta\delta^{-\sigma},[\Delta^{N},1]) \end{equation}
by Lemma \ref{lemma13}(iii) applied with $C := \Delta^{N}/\delta$.

Now, we take $N \in \N$ (also) so large that $\Delta\delta^{-\sigma} \geq \delta^{-\sigma'} \geq (\Delta^{N})^{-\sigma'}$, where $\sigma' < \sigma$ is so close to $\sigma$ that still $\sigma_{0}(1 - \tfrac{\sigma'}{4}) < \sigma'$. Now the inclusion \eqref{form72a} combined with \eqref{form22} yields
\begin{displaymath} \int \mu(B(1) \cap H_{\theta}(K,(\Delta^{N})^{-\sigma'},[\delta,1])) \, d\nu(\theta) \geq \int_{S^{1}} \mu(B(1) \cap H_{\theta}(K,\delta^{-\sigma},[\delta,1])) \, d\nu(\theta) \geq \epsilon. \end{displaymath}
We have now managed to replace the "given" $\delta$ by one of the form $\Delta^{N}$ at the cost of changing $\sigma$ by an insignificant amount. In the sequel, similar arguments will be often omitted, and we will assume that "$\delta$" has some convenient form where necessary. \end{remark}

Write
\begin{equation}\label{form73} M := \frac{16\mathbf{C}}{\epsilon^{2}} \quad \text{and} \quad A := 500 \end{equation}
and for $\theta \in S_{1}$, consider the "low multiplicity sets" 
\begin{equation}\label{form41} L_{\theta} := \{x \in B(1) \cap K : \tfrac{1}{N}|\{0 \leq j \leq N - 1 : x \in H_{\theta}(K,\Delta^{-\sigma_{0}},[A\Delta^{j + 1},A\Delta^{j}])\}| \leq \lambda\} \end{equation}
familiar form Lemma \ref{lemma4}. Note that $M\epsilon_{0} = \lambda$. With the choice of "$M$" above, Lemma \ref{lemma4} implies that there exists a further subset $S_{2} \subset S_{1}$ with $\nu(S_{2}) \geq \tfrac{\epsilon}{4}$ such that $\mu(L_{\theta}) \geq \mu(B(1)) - \tfrac{\epsilon}{4}$ for all $\theta \in S_{2}$. Combining this with \eqref{form65}, we find $\mu(B(1) \cap H_{\theta}(K,\delta^{-\sigma},[\delta,1]) \cap L_{\theta}) \geq \tfrac{\epsilon}{4}$ for all $\theta \in S_{2}$, and as a corollary
\begin{equation}\label{form23} \int_{S^{1}} \mu(B(1) \cap H_{\theta}(K,\delta^{-\sigma},[\delta,1]) \cap L_{\theta}) \, d\nu(\theta) \geq \tfrac{\epsilon}{16}. \end{equation}
To derive a contradiction directly from \eqref{form23} seems difficult, but would be possible if $\mu$ and $\nu$ satisfied a few additional properties. It turns out that these properties can be achieved after replacing $\mu$ by a suitable renormalisation $\bar{\mu} = \mu^{\mathbf{B}}$ to some disc $\mathbf{B} \subset \R^{2}$. The next proposition -- the most technical statement of the paper -- formalises this idea:

\begin{proposition}\label{mainTechnicalProp} Let $A,\mathbf{C},\Delta,\epsilon,\eta,\lambda,s,\sigma,\sigma_{0},\tau,\zeta \in (0,1]$ with $\eta \leq \sigma \leq s$, and $\mathfrak{d} > 1$. Assume that $\delta = \Delta^{N}$ for some $N \in \N$ sufficiently large, depending on all the previously listed constants. Let $\nu$ be a $(\tau,\mathbf{C})$-Frostman measure on $S^{1}$. Let $\mu$ be an $(s,\mathbf{C})$-regular measure with $\spt \mu = K \subset \R^{2}$ satisfying
\begin{equation}\label{form66} \int_{S^{1}} \mu(B(1) \cap H_{\theta}(K,\delta^{-\sigma},[\delta,1]) \cap L_{\theta}) \, d\nu(\theta) \geq \epsilon, \end{equation}
where
\begin{equation}\label{form67} L_{\theta} = \{x \in \R^{2} : \tfrac{1}{N}|\{0 \leq j \leq N - 1 : x \in H_{\theta}(K,\Delta^{-\sigma_{0}},[A\Delta^{j + 1},A\Delta^{j}]\}| \leq \lambda\}. \end{equation} 
Then, there exists an $(s,\mathbf{C})$-regular measure $\bar{\mu}$ with $\spt \bar{\mu} = \bar{K}$ and a $(\tau,\mathbf{C})$-Frostman measure $\bar{\nu}$ on $S^{1}$ such that 
\begin{displaymath} \int_{S^{1}} \bar{\mu}(B(1) \cap H_{\theta}(\bar{K},\bar{\delta}^{-\bar{\sigma}},[\bar{\delta},1]) \cap \bar{L}_{\theta}) \, d\bar{\nu}(\theta) = \bar{\epsilon} \gtrsim_{\mathbf{C},\epsilon,\eta,\mathfrak{d},\zeta,\tau} 1, \end{displaymath} 
where
\begin{displaymath} \bar{\sigma} \in [\sigma - \eta,s] \quad \text{and} \quad \bar{\delta} = \Delta^{\bar{N}} \end{displaymath}
for some $\bar{N} \gtrsim_{\mathfrak{d},\eta,\tau,\zeta} N$, and 
\begin{displaymath} \bar{L}_{\theta} = \{x \in \R^{2} : \tfrac{1}{\bar{N}} |\{0 \leq j \leq \bar{N} - 1 : x \in H_{\theta}(\bar{K},\Delta^{-\sigma_{0}},[\tfrac{A}{50}\Delta^{j + 1},\tfrac{A}{50}\Delta^{j}])\}| \leq C(\mathfrak{d},\tau)\eta^{-1}\lambda\}, \end{displaymath}
where $C(\mathfrak{d},\tau) \geq 1$ is a constant depending only on $\mathfrak{d},\tau$. Moreover,
\begin{equation}\label{form68} \int_{S^{1}} \bar{\mu}(B(1) \cap H_{\theta}(\bar{K},(\bar{\delta}^{q - p})^{-\bar{\sigma} - \zeta},[\bar{\delta}^{q},\bar{\delta}^{p}]) \cap \bar{L}_{\theta}) \, d\bar{\nu}(\theta) \leq \tfrac{\bar{\epsilon}}{10} \end{equation} 
for all pairs $p,q \in \{0,\tfrac{1}{2},1\}$ with $p < q$. Finally, if $I \subset S^{1}$ is a dyadic arc of length $\bar{\delta}^{1/2}$ with $\bar{\nu}(I) > 0$, then $\bar{\nu}(J) \leq \bar{\delta}^{c(\mathfrak{d},\tau)}\bar{\nu}(I)$ for every sub-arc $J \subset I$ with length $|I| \leq \bar{\delta}^{\mathfrak{d}/2}$, where $c(\mathfrak{d},\tau) > 0$ is a constant depending only on $\mathfrak{d} > 1$ and $\tau > 0$. \end{proposition} 

\begin{remark} The statement \eqref{form68} is the formal version of Assumption \ref{A3} in the sketchy proof outline in the introduction. The parameter $\zeta$ will eventually be chosen small enough in terms of the scale $\Delta$.  \end{remark} 

\section{Proof of Proposition \ref{mainTechnicalProp}}

The proof of Proposition \ref{mainTechnicalProp} splits to Lemmas \ref{lemma3} and \ref{lemma6}, and then the main argument in Section \ref{lemma6}. Lemma \ref{lemma3} replaces $\mu$ by another $(s,\mathbf{C})$-regular measure whose associated "local low-multiplicity sets" $L_{1,\theta}$ have an additional stability property \eqref{form42}.

The proof of Lemma \ref{lemma3} uses the renormalised measures from Definition \ref{def:renormalisation}. We note here that these renormalisations satisfy the following "chain rule": if $B,B' \subset \R^{2}$ are arbitrary discs, then
\begin{equation}\label{chainRule} (\mu^{B})^{B'} = \mu^{B''}, \end{equation}
where $B'' = (T_{B'} \circ T_{B})^{-1}(B(1))$. The proof is straightforward (or see \cite[(3.6)]{MR4218963}). By \cite[Lemma 3.7]{MR4218963}, if $\mu$ is $(s,\mathbf{C})$-regular, then every renormalisation $\mu^{B}$ is also $(s,\mathbf{C})$-regular.

\begin{lemma}\label{lemma3} For all $A,\mathbf{C},\Delta,\epsilon,\lambda,\eta,\gamma,\sigma,\sigma_{0} > 0$, there exist constants $\rho = \rho(\eta,\gamma) > 0$ and $N_{0} = N_{0}(\mathbf{C},\eta,\gamma,\epsilon) > 0$ such that the following holds for all $\delta = \Delta^{N}$ with $N \geq N_{0}$.

Let $\mu$ be an $(s,\mathbf{C})$-regular measure satisfying
\begin{equation}\label{form66b} \int_{S^{1}} \mu(B(1) \cap H_{\theta}(K,\delta^{-\sigma},[\delta,1]) \cap L_{\theta}) \, d\nu(\theta) \geq \epsilon, \end{equation}
where $K = \spt \mu$, the sets $L_{\theta}$ are given by \eqref{form67} (the parameters $\sigma_{0},\lambda$ appear in \eqref{form67}), and $\nu$ is a Borel measure with $\nu(S^{1}) \leq \mathbf{C}$. Then, there exists a scale $\delta_{1} = \Delta^{N_{1}}$, where $N_{1} \geq \rho N$, and an $(s,\mathbf{C})$-regular measure $\mu_{1}$ such that
\begin{equation}\label{form36} \int \mu_{1}(B(1) \cap H_{\theta}(K_{1},\delta_{1}^{-\sigma_{1}},[50\delta_{1},50]) \cap L_{1,\theta}) \, d\nu(\theta) = \epsilon_{1} \gtrsim_{\mathbf{C},\eta,\gamma,\epsilon} 1. \end{equation}
Here $\sigma_{1} \geq \sigma - \eta$, and $K_{1} = \spt \mu_{1}$, and for each $\theta \in S^{1}$, the set $L_{1,\theta}$ has the following property: if $I \subset \{0,\ldots,N_{1} - 1\}$ is any interval of length $|I| \geq \gamma N_{1}$, then
\begin{equation}\label{form42} |\{j \in I : x \in H_{\theta}(K_{1},\Delta^{-\sigma_{0}},[A\Delta^{j + 1},A\Delta^{j}]\}| \leq 16\eta^{-1}\lambda|I|, \qquad x \in L_{1,\theta}. \end{equation}
\end{lemma}

\begin{remark} To highlight the main findings of Lemma \ref{lemma3}, note that the proportion of "local high multiplicity scales" of points $x \in L_{1,\theta}$ may be somewhat higher (by a factor $16\eta^{-1}$). than that of $x \in L_{\theta}$. On the bright side, this proportion stays small even when the full interval $\{0,\ldots,N_{1} - 1\}$ is replaced by a sub-interval $I$ with $|I| \geq \gamma N_{1}$.  \end{remark}


\begin{proof}[Proof of Lemma \ref{lemma3}] It follows from the hypothesis \eqref{form66b}, combined with $\mu(B(1)) \leq \mathbf{C}$ and $\nu(S^{1}) \leq \mathbf{C}$, that there exists a set $S \subset S^{1}$ with $\nu(S) \geq \epsilon/(2\mathbf{C})$ such that $\mu(G_{\theta}) \geq \epsilon/(2\mathbf{C})$ for all $\theta \in S$, where
\begin{displaymath} G_{\theta} := B(1) \cap K \cap H_{\theta}(K,\delta^{-\sigma},[\delta,1]) \cap L_{\theta}. \end{displaymath}
Fix $\theta \in S$, and $x \in G_{\theta} \subset L_{\theta}$, thus by \eqref{form67},
\begin{displaymath} |\{0 \leq j \leq N - 1 : x \in H_{\theta}(\Delta^{-\sigma_{0}},[A\Delta^{j + 1},A\Delta^{j}])\}| \leq \lambda N. \end{displaymath}
Rescaling by $\tfrac{1}{N}$, we now view the set $E := E_{x,\theta} := \tfrac{1}{N}\{0 \leq j \leq N - 1 : x \in \ldots\}$ appearing above as a subset of $[0,1]$ of Lebesgue measure $\leq \lambda$. We apply to $E$ Lemma \ref{lemma2} with constants $\gamma > 0$ and $C := 2\eta^{-1}$ (given by the statement). The conclusion is that there exists a family $\mathcal{G}_{x,\theta} \subset \{0,\ldots,N - 1\}$ of disjoint dyadic intervals\footnote{We mean intervals of the form $\{1 \leq j \leq N : j/N \in I\}$, where $I \subset [0,1]$ is a standard dyadic interval.} of length $\geq \rho N = \rho(\eta,\gamma)N > 0$ with the following properties. First, 
\begin{equation}\label{form26} |\cup \mathcal{G}_{x,\theta}| \geq (1 - \tfrac{\eta}{2})N. \end{equation}
Second, if $J \in \mathcal{G}_{x,\theta}$, then 
\begin{equation}\label{form32} |\{j \in I : x \in H_{\theta}(K,\Delta^{-\sigma_{0}},[A\Delta^{j + 1},A\Delta^{j}])\}| \leq 16\eta^{-1}\lambda|I| \end{equation}
for all sub-intervals $I \subset J$ with $|I| \geq \gamma |J|$. This argument (identifying $E$ with a subset of $[0,1]$ and applying Lemma \ref{lemma2}) implicitly needed $N \geq 1$ to be so large that $\rho(\eta,\gamma)N \geq 1$.

The collection $\mathcal{G}_{x,\theta}$ depends on both $\theta \in S$ and $x \in G_{\theta}$. However, $\mathcal{G}_{x,\theta}$ is a collection of dyadic intervals of length $\geq \rho = \rho(\eta,\gamma)$.  There are only in total $\leq 2^{2^{1/\rho + 1}} \lesssim_{\eta,\gamma} 1$ distinct collections of dyadic intervals of length $\geq \rho$. Thus, there exist new subsets $\bar{S} \subset S$ and $\bar{G}_{\theta} \subset G_{\theta}$ with 
\begin{equation}\label{form31} \nu(\bar{S}) \gtrsim_{\eta,\gamma} \nu(S) \gtrsim_{\mathbf{C},\epsilon} 1 \quad \text{and} \quad \mu(\bar{G}_{\theta}) \gtrsim_{\eta,\gamma} \mu(G_{\theta}) \gtrsim_{\mathbf{C},\epsilon} 1 \text{ for } \theta \in \bar{S} \end{equation}
such that $\mathcal{G}_{x,\theta} \equiv \mathcal{G}$ for all pairs $(\theta,x)$ with $\theta \in \bar{S}$ and $x \in \bar{G}_{\theta}$. The collection $\mathcal{G}$ satisfies \eqref{form26} and \eqref{form32} for all $\theta \in \bar{S}$ and $x \in \bar{G}_{\theta}$.

Enumerate the intervals in $\mathcal{G}$ as $\mathcal{G} = \{[a_{i},b_{i}]\}_{i \in I}$. Thus $(b_{i} - a_{i}) \geq \rho N$ for all $i \in I$, and consequently
\begin{equation}\label{form25} |I| \leq \rho^{-1}. \end{equation}

 The next goal is now to find a fixed index $i \in I$ such that
\begin{equation}\label{form27} \int_{\bar{S}} \mu(B(1) \cap H_{\theta}(K,(\Delta^{b_{i} - a_{i}})^{-\sigma + \eta},[50\Delta^{b_{i}},50\Delta^{a_{i}}]) \cap \bar{G}_{\theta}) \, d\nu(\theta) \gtrsim_{\mathbf{C},\eta,\gamma,\epsilon} 1.  \end{equation}
From this, \eqref{form36} will follow by renormalising $\mu$ to a disc of radius $\Delta^{a_{i}}$. For proving \eqref{form27}, we need the next "hereditary" property of high multiplicity sets:

\begin{lemma}\label{l:hereditary} Let $\mu$ be an $(s,\mathbf{C})$-regular measure, $M \geq 1$, and let $F \subset B(1) \cap H_{\theta}(K,M,[\delta,1])$ be a Borel set with $\mu(F) \geq \kappa > 0$. Then, there exists a subset $G \subset F$ with $\mu(G) \geq \tfrac{1}{2}\mu(F)$ such that
\begin{displaymath} \m_{G,\theta}(x \mid [4\delta,4]) \geq M', \qquad x \in G, \end{displaymath}
where $M' = c\mathbf{C}^{-2}\kappa M$, and $c > 0$ is an absolute constant. \end{lemma}

\begin{proof} Since $\emptyset \neq F \subset B(1) \cap H_{\theta}(K,M,[\delta,1])$, it is easy to check that $M \lesssim \mathbf{C}\delta^{-s}$, and
\begin{displaymath} |\pi_{\theta}(F)|_{\delta} \lesssim \frac{|K \cap B(1)|_{\delta}}{M} \leq \mathbf{C}\delta^{-s}/M. \end{displaymath}
Consequently, there exists a collection $\mathcal{T}$ of disjoint $\delta$-tubes parallel to $\pi_{\theta}^{-1}\{0\}$ whose union covers $F$, with $|\mathcal{T}| \lesssim \mathbf{C}\delta^{-s}/M$. Write $M' := c\mathbf{C}^{-2}\kappa M$, where $c > 0$ is a suitable absolute constant to be determined in a moment. Let 
\begin{displaymath} \mathcal{T}_{\mathrm{light}} := \{T \in \mathcal{T} : |\{p \in \mathcal{D}_{\delta}(F) : p \cap T \neq \emptyset\}| \leq M'\} \quad \text{and} \quad \mathcal{T}_{\mathrm{heavy}} := \mathcal{T} \, \setminus \, \mathcal{T}_{\mathrm{light}}. \end{displaymath}
Let $F_{\mathrm{light}}$ be the part of $F$ covered by the tubes in $\mathcal{T}_{\mathrm{light}}$. Now,
\begin{displaymath} \mu(F_{\mathrm{light}}) \lesssim \mathbf{C}M'\delta^{s} \cdot |\mathcal{T}| \lesssim c\kappa. \end{displaymath} 
In particular, $\mu(F \, \setminus \, F_{\mathrm{light}}) \geq \tfrac{1}{2}\mu(F)$ if the constant $c > 0$ is chosen appropriately. Define
\begin{equation}\label{form28} G := F \cap \bigcup_{T \in \mathcal{T}_{\mathrm{heavy}}} 2T. \end{equation}
Then $G \supset F \, \setminus F_{\mathrm{light}}$, so $\mu(G) \geq \tfrac{1}{2}\mu(F)$.

Finally, we claim that $m_{G,\theta}(x \mid [4\delta,4]) \gtrsim M'$ for all $x \in G$. To see this, fix $x \in G$, and let $T \in \mathcal{T}_{\mathrm{heavy}}$ be a tube with $x \in 2T$. By the definition of $\mathcal{T}_{\mathrm{heavy}}$,
\begin{displaymath} |\{p \in \mathcal{D}_{\delta}(F) : p \cap T \neq \emptyset\}| \geq M'. \end{displaymath}
Each square $p \in \mathcal{D}_{\delta}(F)$ with $p \cap T \neq \emptyset$ is contained in $2T$, and contains a point $x_{p} \in F \cap p \subset 2T$. Then also $x_{p} \in G$ according to \eqref{form28}. Moreover, since $x,x_{p} \in B(1) \cap 2T$, and $T$ is a $\delta$-tube, the disc $B(x_{p},4\delta)$ intersects $B(x,4) \cap \pi_{\theta}^{-1}\{\pi_{\theta}(x)\}$. This implies that
\begin{displaymath} |G_{4\delta} \cap B(x,4) \cap \pi_{\theta}^{-1}\{\pi_{\theta}(x)\}|_{4\delta} \geq |\{x_{p} : p \in \mathcal{D}_{\delta}(F) \text{ and } p \cap T \neq \emptyset\}|_{4\delta} \gtrsim M', \end{displaymath}
as claimed. \end{proof} 

We then proceed with finding the index satisfying \eqref{form27}. Fix $\theta \in \bar{S}$, and recall from \eqref{form31} that $\mu(\bar{G}_{\theta}) \gtrsim_{\mathbf{C},\eta,\gamma,\epsilon} 1$. We start by finding an index $i = i(\theta) \in I$ such that
\begin{equation}\label{form30} \mu(B(1) \cap H_{\theta}(K,(\Delta^{b_{i} - a_{i}})^{-\sigma + \eta},[50\Delta^{b_{i}},50\Delta^{a_{i}}]) \cap \bar{G}_{\theta}) \geq \tfrac{\rho}{2} \cdot \mu(\bar{G}_{\theta}) \gtrsim_{\mathbf{C},\eta,\gamma,\epsilon} 1. \end{equation}
If this failed for all indices $i \in I$, then by \eqref{form25}, the part of $\bar{G}_{\theta}$ in the complement of all the sets $H_{\theta}(K,(\Delta^{b_{i} - a_{i}})^{-\sigma + \eta},[50\Delta^{b_{i}},50\Delta^{a_{i}}])$, with $i \in I$, denoted 
\begin{displaymath} F_{\theta} := \bar{G}_{\theta} \, \setminus \, \bigcup_{i \in I} H_{\theta}(K,(\Delta^{b_{i} - a_{i}})^{-\sigma + \eta},[50\Delta^{b_{i}},50\Delta^{a_{i}}]), \end{displaymath}
would still have measure $\mu(F_{\theta}) \geq \tfrac{1}{2}\mu(\bar{G}_{\theta}) \gtrsim_{\mathbf{C},\eta,\gamma,\epsilon} 1$. Let us justify that the set $F_{\theta}$ satisfies the assumptions of Lemma \ref{lemma5}. Indeed, since $F_{\theta} \subset K$, but
\begin{displaymath} F_{\theta} \subset \R^{2} \, \setminus \, H_{\theta}(K,(\Delta^{b_{i} - a_{i}})^{-\sigma + \eta},[50\Delta^{b_{i}},50\Delta^{a_{i}}]), \qquad i \in I, \end{displaymath}
we in particular have
\begin{equation}\label{form16a} H_{\theta}(F_{\theta},(\Delta^{b_{i} - a_{i}})^{-\sigma + \eta},[50\Delta^{b_{i}},50\Delta^{a_{i}}]) = \emptyset, \qquad i \in I. \end{equation}
Furthermore, recall from \eqref{form26} that the good intervals $[a_{i},b_{i}]$, $i \in I$, cover all of $\{1,\ldots,N\}$ except a fraction "$\eta/2$". To render Lemma \ref{lemma5} formally applicable, we need to write $0 = \bar{a}_{0} < \bar{a}_{1} < \ldots < \bar{a}_{n} = N$, where the partition $\{\bar{a}_{j}\}_{j = 0}^{n}$ is determined by the good intervals $[a_{i},b_{i}]$ (thus each $[a_{i},b_{i}] = [\bar{a}_{j},\bar{a}_{j + 1}]$ for some $0 \leq j \leq n - 1$) and their complementary intervals, denoted $\mathcal{B}$. With this notation, $n \leq |I| + 1$, and
\begin{displaymath} \sum_{j \in \mathcal{B}} (\bar{a}_{j + 1} - \bar{a}_{j}) \leq \tfrac{\eta}{2}N, \end{displaymath}
and by \eqref{form16a},
\begin{displaymath} \{0 \leq j \leq n - 1 : x \in H_{\theta}(F_{\theta},(\Delta^{\bar{a}_{j + 1} - \bar{a}_{j}})^{-\sigma + \eta},[50\Delta^{\bar{a}_{j + 1},\bar{a}^{j}}])\} \subset \mathcal{B}, \qquad x \in F_{\theta} \end{displaymath} 
Therefore, Lemma \ref{lemma5} (applied with parameter $\sigma - \eta$) implies
\begin{equation}\label{form29} |(F_{\theta})_{4\delta} \cap \pi_{\theta}^{-1}\{t\}|_{4\delta} \leq C^{2(|I| + 1)}\delta^{-\sigma + \eta/2}, \qquad t \in \R, \end{equation}
where $|I| \leq \rho^{-1} \lesssim_{\eta,\gamma} 1$, recall \eqref{form25}. 

On the other hand, recall also that $F_{\theta} \subset \bar{G}_{\theta} \subset H_{\theta}(K,\delta^{-\sigma},[\delta,1])$, and $\mu(F_{\theta}) \gtrsim_{\mathbf{C},\eta,\gamma,\epsilon} 1$. Now it follows from Lemma \ref{l:hereditary} with $M := \delta^{-\sigma}$ and $\kappa = \kappa(\mathbf{C},\eta,\gamma,\epsilon) = \mu(F_{\theta})$ that there exists a non-empty subset $\bar{F} \subset F_{\theta}$ such that 
\begin{displaymath} |(F_{\theta})_{4\delta} \cap \pi_{\theta}^{-1}\{\pi_{\theta}(x)\}|_{4\delta} \geq m_{\bar{F},\theta}(x \mid [4\delta,4]) \geq c\mathbf{C}^{-2} \cdot \kappa\delta^{-\sigma}, \qquad x \in F'. \end{displaymath}
This estimate contradicts \eqref{form29} for $\delta > 0$ small enough, depending only on $\mathbf{C},\eta,\gamma,\epsilon$. 

From this contradiction, we conclude that for $\theta \in \bar{S}$ fixed, there exists an index $i(\theta) \in I$ such that \eqref{form30} holds. Since $\nu(\bar{S}) \gtrsim_{\mathbf{C},\eta,\gamma,\epsilon} 1$ by \eqref{form31}, and $|I| \leq \rho^{-1} \lesssim_{\eta,\gamma} 1$, the existence of a fixed index $i \in I$ such that \eqref{form27} holds follows from the pigeonhole principle.

We finally use the index $i \in I$ satisfying \eqref{form27} to complete the proof of Lemma \ref{lemma3}. We abbreviate $N_{1} := b_{i} - a_{i}$ and $\delta_{1} := \Delta^{N_{1}}$. Recall from above \eqref{form26} that $N_{1} \geq \rho N$, as claimed in Lemma \ref{lemma3}. The measure $\mu_{1}$ will have the form $\mu_{1} = \mu^{\mathbf{B}}$ for a suitable disc of radius $\Delta^{a_{i}}$. To find $\mathbf{B}$, let $\mathcal{B}$ be a cover of $B(1) \cap K$ by boundedly overlapping discs of radius $\Delta^{a_{i}}$. By \eqref{form27},
\begin{displaymath} \sum_{\mathbf{B} \in \mathcal{B}} \int_{\bar{S}} \mu(\mathbf{B} \cap H_{\theta}(K,\delta_{1}^{-\sigma + \eta},[50\Delta^{b_{i}},50\Delta^{a_{i}}]) \cap \bar{G}_{\theta}) \, d\nu(\theta) \gtrsim_{\mathbf{C},\eta,\gamma,\epsilon} 1. \end{displaymath}
Using the definition of push-forward measures, and Lemma \ref{lemma7},
\begin{align*} \mu(\mathbf{B} & \cap H_{\theta}(K,\delta_{1}^{-\sigma + \eta},[50\Delta^{b_{i}},50\Delta^{a_{i}}]) \cap \bar{G}_{\theta})\\
& = (\Delta^{a_{i}})^{s}\mu^{\mathbf{B}}(B(1) \cap H_{\theta}(T_{\mathbf{B}}(K),\delta_{1}^{-\sigma + \eta},[50\delta_{1},50]) \cap T_{\mathbf{B}}(\bar{G}_{\theta})). \end{align*} 
Since on the other hand $|\mathcal{B}| \leq \mathbf{C}(\Delta^{a_{i}})^{-s}$ by the $(s,\mathbf{C})$-regularity of $\mu$, we conclude that there exists a disc $\mathbf{B} \in \mathcal{B}$ such that
\begin{displaymath} \int_{S_{3}} \mu^{\mathbf{B}}(B(1) \cap H_{\theta}(T_{\mathbf{B}}(K),\delta_{1}^{-\sigma + \eta},[50\delta_{1},50]) \cap T_{\mathbf{B}}(\bar{G}_{\theta})) \, d\nu(\theta) \gtrsim_{\mathbf{C},\eta,\gamma,\epsilon} 1. \end{displaymath}
This is \eqref{form36}, once we define
\begin{displaymath} \mu_{1} := \mu^{\mathbf{B}} \quad \text{and} \quad L_{1,\theta} := T_{\mathbf{B}}(\bar{G}_{\theta}) \text{ for } \theta \in \bar{S}. \end{displaymath}
We also write $K_{1} = \spt \mu_{1} = T_{\mathbf{B}}(K)$. For $\theta \in S^{1} \, \setminus \, \bar{S}$, we set $L_{1,\theta} := \emptyset$.

It only remains to verify \eqref{form42}, namely that whenever $\theta \in S^{1}$ and $y \in L_{1,\theta}$, and $I \subset \{0,\ldots,N_{1} - 1\}$ is an interval of length $|I| \geq \gamma N_{1}$, then 
\begin{equation}\label{form44}  |\{j \in I : y \in H_{\theta}(K_{1}\Delta^{-\sigma_{0}},[A\Delta^{j + 1},A\Delta^{j}]\}| \leq 16\eta^{-1}\lambda |I|. \end{equation}
This is a consequence of \eqref{form32}, and the rescaling Lemma \ref{lemma7}, but let us check the details. Fix $\theta \in \bar{S}$, and $y = T_{\mathbf{B}}(x) \in L_{1,\theta} = T_{\mathbf{B}}(\bar{G}_{\theta})$, thus $x \in \bar{G}_{\theta}$. Fix also an interval $I \subset \{0,\ldots,N_{1} - 1\}$ of length $|I| \geq \gamma N_{1}$. Recall that $N_{1} = b_{i} - a_{i}$, and that $J := [a_{i},b_{i}] \in \mathcal{G}$. Since $\theta \in \bar{S}$ and $x \in \bar{G}_{\theta}$, in particular $J \in \mathcal{G}_{x,\theta}$, recalling the definition of $\mathcal{G}$ below \eqref{form31}.

 Now $I' := I + a_{i} \subset J$ is an interval of length $|I'| \geq \gamma N_{1} = \gamma (b_{i} - a_{i})$, so according to \eqref{form32} (and since $J \in \mathcal{G}_{x,\theta}$),
\begin{equation}\label{form43} |\{j \in I' : x \in H_{\theta}(K,\Delta^{-\sigma_{0}},[A\Delta^{j + 1},A\Delta^{j}])\}| \leq 16\eta^{-1}\lambda |I|. \end{equation}
Finally, it follows from Lemma \ref{lemma7} that $x \in H_{\theta}(K,\Delta^{-\sigma_{0}},[A\Delta^{j + 1},A\Delta^{j}])$ if and only if $y \in H_{\theta}(K_{1},\Delta^{-\sigma_{0}},[A\Delta^{j + 1 - a_{i}},A\Delta^{j - a_{i}}])$. In combination with \eqref{form43}, this proves \eqref{form44}. The proof of Lemma \ref{lemma3} is complete. \end{proof}

\subsection{Finding a tangent with exact dimensional projections}\label{s:tangent}

Lemma \ref{lemma3} did not yet make much visible progress towards Proposition \ref{mainTechnicalProp}. In particular, it did nothing to bring us closer to \eqref{form68}. Making such progress is the content of Lemma \ref{lemma6}.

\begin{lemma}\label{lemma6} Let $\mathbf{C} \geq 1$, $\zeta_{1},\sigma,s \in (0,1]$ with $\sigma \leq s$, and let $\mathcal{Q} \subset [0,1]$ be finite. Then, there exists $\delta_{0} = \delta_{0}(\mathbf{C},\mathcal{Q},\zeta_{1}) > 0$ such that the following holds for all $\delta \in (0,\delta_{0}]$.

Let $\{\epsilon_{n}\}_{n \geq 1} \subset (0,1)$ be a non-increasing sequence, and let $C,\mathbf{C} \geq 1$. Let $\mu$ be an $(s,\mathbf{C})$-regular measure on $\R^{2}$ with $K := \spt \mu$. Let $\nu$ be a finite Borel measure on $S^{1}$, and let $\{L_{\theta}\}_{\theta \in S^{1}}$ be a family of Borel sets in $\R^{2}$ indexed by $S^{1}$. Assume that
\begin{equation}\label{form72} \int_{S^{1}} \mu(B(1) \cap H_{\theta}(K,\delta^{-\sigma},[C\delta,C]) \cap L_{\theta}) \, d\nu(\theta) \geq \epsilon_{1}. \end{equation} 
Write $\mathfrak{gap} := \min\{|p - q| : p,q \in \mathcal{Q}, \, p \neq q\}$. Then, there exists an index $n \leq 2/\zeta_{1}$, a scale $\bar{\delta} \leq \delta^{\mathfrak{gap}^{n}}$, a disc $\mathbf{B} \subset \R^{2}$ of radius $\mathrm{rad}(\mathbf{B}) \in [\delta/\bar{\delta},1]$, and a number $\bar{\sigma} \in [\sigma,s]$ such that
\begin{equation}\label{form46} \int_{S^{1}} \mu^{\mathbf{B}}(B(1) \cap H_{\theta}(T_{\mathbf{B}}(K),\bar{\delta}^{-\bar{\sigma}},[C\bar{\delta},C]) \cap T_{\mathbf{B}}(L_{\theta})) \, d\nu(\theta) \geq \epsilon_{n}, \end{equation} 
but on the other hand
\begin{equation}\label{form47} \int_{S^{1}} \mu^{\mathbf{B}}(B(C) \cap H_{\theta}(T_{\mathbf{B}}(K),(\bar{\delta}^{q - p})^{-(\bar{\sigma} + \zeta_{1})},[C\bar{\delta}^{q},C\bar{\delta}^{p}]) \cap T_{\mathbf{B}}(L_{\theta})) \, d\nu(\theta) \leq \mathbf{C}C^{s}\epsilon_{n + 1} \end{equation}
for all pairs $p,q \in \mathcal{Q}$ with $p < q$.
\end{lemma} 

\begin{proof} We will recursively define a sequence of measures $\{\mu_{n}\}_{n \geq 1}$, where each $\mu_{n + 1}$ has the form $\mu_{n + 1} = \mu_{n}^{B}$ for some disc $B \subset \R^{2}$. We initialise this by setting $\mu_{1} := \mu$. It follows from the "chain rule for renormalisations" \eqref{chainRule} that every element in $\{\mu_{n}\}_{n \geq 0}$ has the form $\mu_{n} = \mu^{\mathbf{B}}$ for some disc $\mathbf{B} \subset \R^{2}$. In particular, the measures $\mu_{n}$ are $(s,\mathbf{C})$-regular.

In addition to the measures $\{\mu_{n}\}$, we will define an increasing sequence of scales $\{\delta_{n}\}_{n \geq 1}$, where $\delta_{1} := \delta$, and every scale $\delta_{n + 1}$ will have the form $\delta_{n + 1} = \delta_{n}^{q_{n} - p_{n}}$ for some $p_{n},q_{n} \in \mathcal{Q}$ with $q_{n} > p_{n}$. In particular, $\delta_{n} \leq \delta_{n + 1} \leq \delta_{n}^{\mathfrak{gap}}$. Finally, we also define $\sigma_{0} := \sigma$, and $\sigma_{n + 1} := \sigma_{n} + \zeta_{1}$.

Assume that $\mu_{n},\delta_{n}$ have already been defined for some $n \geq 1$, and that $\mu_{n} = \mu^{\mathbf{B}_{n}}$ for some disc $\mathbf{B}_{n} \subset \R^{2}$. For $n = 1$, this is true with $\mathbf{B}_{0} = B(1)$. Write $K_{n} := \spt \mu_{n} = T_{\mathbf{B}_{n}}(K)$ and $L_{\theta,n} := T_{\mathbf{B}_{n}}(L_{\theta})$. Assume inductively that 
\begin{equation}\label{form45} \int \mu_{n}(B(1) \cap H_{\theta}(K_{n},\delta_{n}^{-\sigma_{n}},[C\delta_{n},C]) \cap L_{\theta,n}) \, d\nu(\theta) \geq \epsilon_{n}. \end{equation} 
Now, if
\begin{equation}\label{form48} \int \mu_{n}(B(C) \cap H_{\theta}(K_{n},(\delta_{n}^{q - p})^{-\sigma_{n + 1}},[C\delta_{n}^{q},C\delta_{n}^{p}]) \cap L_{\theta,n}) \, d\nu(\theta) \leq \mathbf{C}C^{s}\epsilon_{n + 1},  \end{equation} 
for all $p,q \in \mathcal{Q}$ with $p < q$, the recursive construction terminates. We set $\bar{\sigma} := \sigma_{n}$ and $\bar{\delta} := \delta_{n} \leq \delta^{\mathfrak{gap}^{n}}$. At this point, \eqref{form46}-\eqref{form47} are given by \eqref{form45}-\eqref{form48}. 

Assume then that \eqref{form48} does not hold: there exists a pair $p_{n},q_{n} \in \mathcal{Q}$ with $p_{n} < q_{n}$ and
\begin{equation}\label{form49} \int \mu_{n}(B(C) \cap H_{\theta}(K_{n},(\delta_{n}^{q_{n} - p_{n}})^{-\sigma_{n + 1}},[C\delta_{n}^{q_{n}},C\delta_{n}^{p_{n}}]) \cap L_{\theta,n}) \, d\nu(\theta) > \mathbf{C}C^{s}\epsilon_{n + 1}. \end{equation}
We now define $\delta_{n + 1} := \delta_{n}^{q_{n} - p_{n}}$. To define the measure $\mu_{n + 1}$, cover $K_{n} \cap B(C)$ by a minimal family $\mathcal{B}_{n + 1}$ of discs of radius $\delta_{n}^{p_{n}}$. Using the relation
\begin{align*} \mu_{n}(B \cap & H_{\theta}(K_{n},\delta_{n + 1}^{-\sigma_{n + 1}},[C\delta_{n}^{q_{n}},C\delta_{n}^{p_{n}}]) \cap L_{\theta,n})\\
& = (\delta_{n}^{p_{n}})^{s} \cdot \mu_{n}^{B}(B(1) \cap H_{\theta}(T_{B}(K_{n}),\delta_{n + 1}^{-\sigma_{n + 1}},[C\delta_{n + 1},C]) \cap T_{B}(L_{\theta,n})), \end{align*}
valid for all discs $B \subset \R^{2}$ of radius $\delta_{n}^{p_{n}}$ (see Lemma \ref{lemma7}), we deduce
\begin{align*} \sum_{B \in \mathcal{B}_{n + 1}}  & (\delta_{n}^{p_{n}})^{s} \int \mu_{n}^{B}(B(1) \cap H_{\theta}(T_{B}(K_{n}),\delta_{n + 1}^{-\sigma_{n + 1}},[C\delta_{n + 1},C]) \cap T_{B}(L_{\theta,n})) \, d\nu(\theta)\\
& \geq \sum_{B \in \mathcal{B}_{n + 1}} \int \mu_{n}(B \cap H_{\theta}(K_{n},\delta_{n + 1}^{-\sigma_{n + 1}},[C\delta_{n}^{q_{n}},C\delta_{n}^{p_{n}}]) \cap L_{\theta,n}) \, d\nu(\theta) \geq \mathbf{C}C^{s}\epsilon_{n + 1}. \end{align*} 
Since $|\mathcal{B}_{n + 1}| \leq \mathbf{C}(C/\delta_{n}^{p_{n}})^{s}$ by the $(s,\mathbf{C})$-regularity of $\mu_{n}$, it follows that there exists at least one disc $B_{n + 1} \in \mathcal{B}_{n + 1}$ such that
\begin{displaymath} \int \mu_{n}^{B_{n + 1}}(B(1) \cap H_{\theta}(T_{B_{n + 1}}(K_{n}),\delta_{n + 1}^{-\sigma_{n + 1}},[C\delta_{n + 1},C]) \cap T_{B_{n + 1}}(L_{\theta})) \, d\nu(\theta) \geq \epsilon_{n + 1}. \end{displaymath}
We now define $\mu_{n + 1} := \mu_{n}^{B_{n + 1}}$. Then \eqref{form45} has been verified at the index $n + 1$, and the construction may proceed. We record that if we now expressed $\mu_{n + 1}$ as $\mu^{\mathbf{B}_{n + 1}}$, then 
\begin{equation}\label{form53} \mathrm{rad}(\mathbf{B}_{n + 1}) = \mathrm{rad}(\mathbf{B}_{n})\delta_{n}^{p_{n}}. \end{equation}

To complete the proof, it suffices to check that the construction terminates in $n \leq 2\zeta_{1}^{-1}$ steps. Morally, this is because $\sigma_{n} > 1$ for $n > \zeta_{1}^{-1}$. To be more precise, we claim that the construction terminates latest at the first step "$n$" when $\sigma_{n} \geq s + \zeta_{1}$ is satisfied (clearly $n \leq 2\zeta_{1}^{-1}$). Indeed, if this was not the case, then we could find two numbers $p_{n},q_{n} \in \mathcal{Q}$ with $p_{n} < q_{n}$, such that \eqref{form49} holds. In particular, there exists $\theta \in S^{1}$ such that
\begin{displaymath} H_{\theta}(K_{n},(\delta_{n}^{q_{n} - p_{n}})^{-s - \zeta_{1}},[C\delta_{n}^{q_{n}},C\delta_{n}^{p_{n}}]) \neq \emptyset. \end{displaymath}
If "$x$" lies in the set above, then by definition
\begin{displaymath} |(K_{n})_{C\delta^{q_{n}}} \cap B(x,C\delta^{p_{n}}) \cap \pi_{\theta}^{-1}\{\pi_{\theta}(x)\}|_{C\delta_{n}^{q_{n}}} \geq (\delta_{n}^{q_{n} - p_{n}})^{-s - \zeta_{1}}. \end{displaymath}
However, the left hand side is bounded from above by $\lesssim \mathbf{C}(\delta_{n}^{q_{n} - p_{n}})^{-s}$. These inequalities are incompatible provided that $\delta_{n}^{\zeta_{1}(p_{n} - q_{n})} \geq A\mathbf{C}$ for an absolute constant $A \geq 1$, and this is the case if $\delta > 0$ was initially chosen small enough in terms of $\mathbf{C},\mathcal{Q},\zeta_{1}$.

We claimed in the statement that $\bar{\sigma} \in [\sigma,s]$, but the argument above only seems to yield $\bar{\sigma} \in [\sigma,s + \zeta_{1}]$. This is only a formal problem: if $\bar{\sigma} \in (s,s + \zeta_{1}]$, then we simply redefine $\bar{\sigma} := s$. Now the lower bound \eqref{form72} continues to hold as stated, and the upper bound \eqref{form47} holds with $\bar{\sigma}  + 2\zeta_{1}$ in place of $\bar{\sigma} + \zeta_{1}$.

We finally note that the final disc $\mathbf{B} \subset \R^{2}$ appearing in \eqref{form46}-\eqref{form47} satisfies $\mathrm{rad}(\mathbf{B}) \geq \delta/\bar{\delta}$. This follows from the recursive formula \eqref{form53}, and the fact that $\mathcal{Q} \subset [0,1]$, since
\begin{align*} \bar{\delta} \cdot \mathrm{rad}(\mathbf{B}) & = \delta_{n} \cdot \mathrm{rad}(\mathbf{B}_{n}) = \delta_{n - 1}^{q_{n - 1} - p_{n - 1}} \cdot \mathrm{rad}(\mathbf{B}_{n - 1})\delta_{n - 1}^{p_{n - 1}}\\
& \geq \delta_{n - 1} \cdot \mathrm{rad}(\mathbf{B}_{n - 1}) \geq \ldots \geq \delta_{0} \cdot \mathrm{rad}(\mathbf{B}_{0}) = \delta. \end{align*} 
This completes the proof.  \end{proof} 

\subsection{Proof of Proposition \ref{mainTechnicalProp}}\label{s2} We will now combine Lemmas \ref{lemma3} and \ref{lemma6} to prove Proposition \ref{mainTechnicalProp}, which we restate here despite its length:
\begin{proposition}\label{mainTechnicalPropRestated} Let $A,\mathbf{C},\Delta,\epsilon,\eta,\lambda,s,\sigma,\sigma_{0},\tau,\zeta \in (0,1]$ with $\eta \leq \sigma \leq s$, and $\mathfrak{d} > 1$. Assume that $\delta = \Delta^{N}$ for some $N \in \N$ sufficiently large, depending on all the previously listed constants. Let $\nu$ be a $(\tau,\mathbf{C})$-Frostman measure on $S^{1}$. Let $\mu$ be an $(s,\mathbf{C})$-regular measure with $\spt \mu = K \subset \R^{2}$ satisfying
\begin{equation}\label{form50} \int_{S^{1}} \mu(B(1) \cap H_{\theta}(K,\delta^{-\sigma},[\delta,1]) \cap L_{\theta}) \, d\nu(\theta) \geq \epsilon, \end{equation}
where
\begin{equation}\label{form69} L_{\theta} = \{x \in \R^{2} : \tfrac{1}{N}|\{0 \leq j \leq N - 1 : x \in H_{\theta}(K,\Delta^{-\sigma_{0}},[A\Delta^{j + 1},A\Delta^{j}]\}| \leq \lambda\}. \end{equation} 
Then, there exists an $(s,\mathbf{C})$-regular measure $\bar{\mu}$ with $\spt \bar{\mu} = \bar{K}$ and a $(\tau,\mathbf{C})$-Frostman measure on $\bar{\nu}$ on $S^{1}$ such that 
\begin{equation}\label{form74} \int_{S^{1}} \bar{\mu}(B(1) \cap H_{\theta}(\bar{K},\bar{\delta}^{-\bar{\sigma}},[\bar{\delta},1]) \cap \bar{L}_{\theta}) \, d\bar{\nu}(\theta) = \bar{\epsilon} \gtrsim_{\mathbf{C},\epsilon,\zeta,\eta,\mathfrak{d},\tau} 1, \end{equation} 
where
\begin{displaymath} \bar{\sigma} \in [\sigma - \eta,s] \quad \text{and} \quad \bar{\delta} = \Delta^{\bar{N}} \end{displaymath}
for some $\bar{N} \gtrsim_{\mathfrak{d},\eta,\tau,\zeta} N$, and 
\begin{displaymath} \bar{L}_{\theta} = \{x \in \R^{2} : \tfrac{1}{\bar{N}} |\{0 \leq j \leq \bar{N} - 1 : x \in H_{\theta}(\bar{K},\Delta^{-\sigma_{0}},[\tfrac{A}{50}\Delta^{j + 1},\tfrac{A}{50}\Delta^{j}])\}| \leq C(\mathfrak{d},\tau)\eta^{-1}\lambda\}, \end{displaymath}
where $C(\mathfrak{d},\tau) \geq 1$ is a constant depending only on $\mathfrak{d},\tau$. Moreover,
\begin{equation}\label{form79} \int_{S^{1}} \bar{\mu}(B(1) \cap H_{\theta}(\bar{K},(\bar{\delta}^{q - p})^{-\bar{\sigma} - \zeta},[\bar{\delta}^{q},\bar{\delta}^{p}]) \cap \bar{L}_{\theta}) \, d\bar{\nu}(\theta) \leq \tfrac{\bar{\epsilon}}{10} \end{equation} 
for all pairs $p,q \in \{0,\tfrac{1}{2},1\}$ with $p < q$. Finally, if $I \subset S^{1}$ is a dyadic arc of length $\bar{\delta}^{1/2}$ with $\bar{\nu}(I) > 0$, then $\bar{\nu}(J) \leq \bar{\delta}^{c(\mathfrak{d},\tau)}\bar{\nu}(I)$ for every sub-arc $J \subset I$ with length $|I| \leq \bar{\delta}^{\mathfrak{d}/2}$, where $c(\mathfrak{d},\tau) > 0$ is a constant depending only on $\mathfrak{d} > 1$ and $\tau > 0$. \end{proposition} 
In this section, the constants $\mathbf{C},\Delta,\epsilon,\eta,\lambda,\sigma,\sigma_{0},\tau,\zeta \in (0,1]$ and $\mathfrak{d} > 1$ are the ones given in the statement of Proposition \ref{mainTechnicalPropRestated}. The first step in the proof is to apply Lemma \ref{lemma3} to the measure $\mu$, which satisfies the hypothesis of that lemma thanks to \eqref{form50}-\eqref{form69}. The constants $\mathbf{C},\Delta,\epsilon,\lambda,\eta,\sigma,\sigma_{0}$ are exactly the ones defined just above, but there is additionally the constant $\gamma > 0$, which we now need to define -- and which will depend on the given constants $\mathfrak{d},\tau,\zeta$ in Proposition \ref{mainTechnicalPropRestated}. 

\subsubsection{Defining the constant $\gamma$ and the set $\mathcal{Q} \subset [0,1]$}\label{s:gamma} To define the constant $\gamma > 0$, we will first define a finite set $\mathcal{Q} \subset [0,1]$ -- indeed those to which Lemma \ref{lemma6} will soon be applied. Recall the parameters $\mathfrak{d} > 1$ and $\tau > 0$ in Proposition \ref{mainTechnicalPropRestated}. Let $\mathfrak{n} \lesssim_{\mathfrak{d},\tau} 1$ be the smallest natural number satisfying
 \begin{equation}\label{form70a} \frac{1}{2\mathfrak{d}^{\mathfrak{n} - 1}} \leq \frac{\tau}{4}. \end{equation}
 Then, let $\mathcal{Q} := \mathcal{Q}_{0} \cup \mathcal{Q}_{1} \cup \{0\} \subset [0,1] \cap \Q$, where 
 \begin{equation}\label{rationals} \mathcal{Q}_{0} = \{\tfrac{1}{2}\mathfrak{d}^{-j} : 0 \leq j \leq \mathfrak{n}\} \quad \text{and} \quad \mathcal{Q}_{1} = \{\mathfrak{d}^{-j} : 0 \leq j \leq \mathfrak{n}\}. \end{equation}
In fact, $\mathcal{Q}_{0}$ looks familiar from Lemma \ref{lemma8a}. An important point will be that the numbers $\delta^{p}$ with $p \in \mathcal{Q}$ include $\{\delta,\delta^{1/2},1\}$, but also many other triples of the form $\{\delta^{2p},\delta^{p},1\}$. Write
\begin{equation}\label{gap} \mathfrak{gap} = \min\{|p - q| : p,q \in \mathcal{Q}, \, p \neq q\} \gtrsim_{\mathfrak{d},\tau} 1, \end{equation} 
and 
\begin{equation}\label{zetaOne} \zeta_{1} := \tfrac{\min \mathcal{Q}_{0}}{2} \cdot \min\left\{\zeta,\eta\right\}. \end{equation}
Note that $\min \mathcal{Q}_{0} \gtrsim_{\mathfrak{d},\tau} 1$, thus $\zeta_{1} \gtrsim_{\mathfrak{d},\tau} \min\{\zeta,\eta\}$. Then, define
\begin{equation}\label{defGamma} \gamma := c(\mathfrak{d},\tau) \cdot \mathfrak{gap}^{2\zeta_{1}^{-1}} \gtrsim_{\mathfrak{d},\eta,\tau,\zeta} 1, \end{equation}
where $c(\mathfrak{d},\tau) > 0$ is a constant depending only on $\mathfrak{d},\tau$, to be specified later.

We are then ready to prove Proposition \ref{mainTechnicalProp}.

  
 \begin{proof}[Proof of Proposition \ref{mainTechnicalProp}] As already hinted at above, we begin by applying Lemma \ref{lemma3} to the measures $\mu,\nu$, with parameters $\mathbf{C},\Delta,\epsilon,\lambda,\sigma,\sigma_{0},\tau,\eta/2$ provided by the statement of Proposition \ref{mainTechnicalProp}, and the constant $\gamma \sim_{\mathfrak{d},\eta,\tau,\zeta} 1$ determined at \eqref{defGamma}. The result is an $(s,\mathbf{C})$-regular measure $\mu_{1}$ with $\spt \mu_{1} = K_{1}$ satisfying, for $\sigma_{1} \geq \sigma - \eta/2$,
 \begin{equation}\label{form70} \int \mu_{1}(B(1) \cap H_{\theta}(K_{1},\delta_{1}^{-\sigma_{1}},[50\delta_{1},50]) \cap L_{1,\theta}) \, d\nu(\theta) = \epsilon_{1} \gtrsim_{\mathbf{C},\mathfrak{d},\eta,\epsilon,\tau,\zeta} 1. \end{equation}
 Here $\delta_{1} = \Delta^{N_{1}}$ for some $N_{1} \gtrsim_{\mathfrak{d},\eta,\tau,\zeta} N$, and for all intervals $I \subset \{0,\ldots,N_{1} - 1\}$ of length $|I| \geq \gamma N_{1}$ it holds
 \begin{equation}\label{form56} |\{j \in I : x \in H_{\theta}(K_{1},\Delta^{-\sigma_{0}},[A\Delta^{j + 1},A\Delta^{j}]\}| \leq 50\eta^{-1}\lambda |I|, \qquad x \in L_{1,\theta} \end{equation}
 
 We then prepare to apply Lemma \ref{lemma6} to the measure $\mu_{1}$, with exponent $\sigma_{1}$ (as in \eqref{form70}), and parameter
 \begin{equation}\label{form77a} \zeta_{1} = \tfrac{\min \mathcal{Q}_{0}}{2} \cdot \min\left\{\zeta,\eta\right\} \sim_{\mathfrak{d},\eta,\tau,\zeta} 1 \end{equation}
familiar from \eqref{zetaOne}. By \eqref{form70}, the main hypothesis \eqref{form72} of Lemma \ref{lemma6} is indeed satisfied with constants $\epsilon_{1}$ (as in \eqref{form70}) and $C = 50$, and the sets $L_{1,\theta}$ in place of $L_{\theta}$.
 
 To apply Lemma \ref{lemma6}, we still need to specify the numbers $\mathcal{Q}$, and the sequence $\{\epsilon_{n}\}_{n \geq 1}$. The set $\mathcal{Q}$ is the one defined in Section \ref{s:gamma}, and we have also already defined $\epsilon_{1}$. We then define recursively
\begin{equation}\label{form63} \epsilon_{n + 1} = c(\mathbf{C},\mathfrak{d},\tau)\epsilon_{n}^{2} \end{equation}
for a small constant $c(\mathbf{C},\mathfrak{d},\tau) > 0$, depending only on $\mathbf{C},\mathfrak{d},\tau$.
 
 Now we have defined all the parameters needed to apply Lemma \ref{lemma6}. The conclusion is that there exists an index $n \leq 2\zeta_{1}^{-1}$, a scale 
 \begin{displaymath} \delta_{2} \leq \delta_{1}^{\mathfrak{gap}^{2\zeta_{1}^{-1}}} \stackrel{\eqref{defGamma}}{\leq} \delta_{1}^{\gamma}, \end{displaymath}
 a disc $\mathbf{B} \subset \R^{2}$ of radius $\mathrm{rad}(\mathbf{B}) \in [\delta_{1}/\delta_{2},1]$, and a number 
 \begin{equation}\label{form76a} \sigma_{2} \in [\sigma_{1},s] \subset [\sigma - \eta/2,s] \end{equation}
 such that on the one hand, writing $\mu_{2} := \mu_{1}^{\mathbf{B}}$ and $K_{2} := T_{\mathbf{B}}(K_{1})$,
 \begin{equation}\label{form57} \int_{S^{1}} \mu_{2}(B(1) \cap H_{\theta}(K_{2},\delta_{2}^{-\sigma_{2}},[50\delta_{2},50]) \cap T_{\mathbf{B}}(L_{1,\theta})) \, d\nu(\theta) \geq \epsilon_{n}, \end{equation}
 but on the other hand
 \begin{equation}\label{form58} \int_{S^{1}} \mu_{2}(B(50) \cap H_{\theta}(K_{2},(\delta_{2}^{q - p})^{-\sigma_{2} - \zeta_{1}},[50\delta_{2}^{q},50\delta_{2}^{p}]) \cap T_{\mathbf{B}}(L_{1,\theta})) \, d\nu(\theta) \leq 50^{s}\mathbf{C}\epsilon_{n + 1} \end{equation}
 for all pairs $p,q \in \mathcal{Q}$ with $p < q$. In particular, since $\delta_{2} \leq \delta_{1}^{\gamma}$, and $\delta_{1} = \Delta^{N_{1}}$ for $N_{1} \gtrsim_{\mathfrak{d},\eta,\tau,\zeta}N$ according to Lemma \ref{lemma3}, we may express $\delta_{2}$ as
 \begin{equation}\label{form77} \delta_{2} = \Delta^{N_{2}} \quad \text{with} \quad N_{2} \geq \gamma N_{1} \stackrel{\eqref{defGamma}}{\gtrsim}_{\mathfrak{d},\eta,\tau,\zeta}N. \end{equation} 
 Let us examine the set 
 \begin{displaymath} L_{2,\theta} := T_{\mathbf{B}}(L_{1,\theta}) \end{displaymath}
 for a moment. Write $\mathrm{rad}(\mathbf{B}) = \Delta^{a}$, where $a \in \{0,\ldots,N_{1} - N_{2}\}$. We claim that if $y \in L_{2,\theta}$, then
 \begin{equation}\label{form54} |\{0 \leq j \leq N_{2} - 1 : y \in H_{\theta}(K_{2},\Delta^{-\sigma_{0}},[A\Delta^{j + 1},A\Delta^{j}])\}| \leq 50\eta^{-1} \lambda N_{2}. \end{equation}
 To see this, write $y = T_{\mathbf{B}}(x)$ for some $x \in L_{1,\theta}$, and $I := \{0,\ldots,N_{2} - 1\}$. Then $I + a \subset \{0,\ldots,N_{1} - 1\}$ is an interval of length $|I| \geq \gamma N_{1}$, so 
 \begin{displaymath} |\{j \in I + a : x \in H_{\theta}(K_{1},\Delta^{-\sigma_{0}},[A\Delta^{j + 1},A\Delta^{j}])\}| \leq 50\eta^{-1}\lambda |I| = 50\eta^{-1}\lambda N_{2} \end{displaymath} 
 by \eqref{form56}. This implies \eqref{form54}, because Lemma \ref{lemma7} yields the relation
 \begin{displaymath} H_{\theta}(K_{2},\Delta^{-\sigma_{0}},[A\Delta^{j + 1},A\Delta^{j}]) = T_{\mathbf{B}}(H_{\theta}(K_{1},\Delta^{-\sigma_{0}},[A\Delta^{j + 1 + a},\Delta^{j + a}])). \end{displaymath}

Next, we will put the set $\mathcal{Q}$ into practice. Starting from \eqref{form57}, and noting that $\nu(S^{1}) \leq \mathbf{C}$ and $\mu_{2}(B(1)) \leq \mathbf{C}$, find a subset $S \subset S^{1}$ with $\nu(S) \geq \epsilon_{n}/(2\mathbf{C})$ such that 
 \begin{equation}\label{form62} \mu_{2}(B(1) \cap H_{\theta}(K_{2},\delta_{2}^{-\sigma_{2}},[50\delta_{2},50]) \cap L_{2,\theta}) \geq \epsilon_{n}/(2\mathbf{C}), \qquad \theta \in S. \end{equation}
 We now plan to apply Lemma \ref{lemma8a} to the probability measure $\nu_{S} = \nu(S)^{-1}\nu|_{S}$, and at scale $\delta_{2}$. Since $\nu(S) \geq \epsilon_{n}/(2\mathbf{C})$, the hypothesis of Lemma \ref{lemma8a} is valid for $\delta > 0$ (thus $\delta_{2} > 0$) small enough. We repeat the conclusion of Lemma \ref{lemma8a} in the present notation:  

 \begin{lemma}\label{lemma8} There exists $p_{0} \in \mathcal{Q}_{0} \subset \mathcal{Q} \cap (0,1]$ (see \eqref{rationals}) and a subset $\bar{S} \subset S$ with $\nu(\bar{S}) \gtrsim_{\mathfrak{d},\tau} \nu(S)$ such that if $I \in \mathcal{D}_{\delta_{2}^{p_{0}}}(S^{1})$ has $\nu(I \cap \bar{S}) > 0$, then
\begin{equation}\label{form61} \nu(J \cap \bar{S}) \leq \delta_{2}^{c(\mathfrak{d},\tau)}\nu(I \cap \bar{S}), \qquad J \in \mathcal{D}_{\delta_{2}^{\mathfrak{d} p_{0}}}(I), \end{equation} 
where $c(\mathfrak{d},\tau) = \tau/(20\mathfrak{n}) \gtrsim_{\mathfrak{d},\tau} 1$ (the number $\mathfrak{n} \sim_{\mathfrak{d},\tau} 1$ was defined at \eqref{form70a}).  \end{lemma}

Write $\bar{\delta} := \delta_{2}^{2p_{0}}$ for the number $p_{0} \in \mathcal{Q}_{0}$ provided by Lemma \ref{lemma8}. Equivalently, by \eqref{form77},
\begin{equation}\label{form78} \bar{\delta} = \Delta^{\bar{N}} \quad \text{with} \quad \bar{N} = 2p_{0} \cdot N_{2}. \end{equation}
(We ignore the issue that perhaps $2p_{0} \cdot N_{2} \notin \N$. This could be fixed in several ways, e.g. by considering ceiling functions, or choosing the numbers $p_{0}$ to be rationals.) This is the "final" scale we are looking for in Proposition \ref{mainTechnicalPropRestated}. The restriction $\bar{\nu} := \nu|_{\bar{S}}$ has the property desired in Proposition \ref{mainTechnicalPropRestated} by virtue of Lemma \ref{lemma8}, since $\bar{\delta}^{1/2} = \delta_{2}^{p_{0}}$.

We then proceed to define the objects $\bar{\mu}$ and $\bar{L}_{\theta}$ whose existence is claimed in Proposition \ref{mainTechnicalPropRestated}. To begin with, fix $\theta \in \bar{S}$, and denote (temporarily) $\mu_{\theta} := (\mu_{2})|_{L_{2,\theta}}$ the restriction of $\mu_{2}$ to $L_{2,\theta}$. Then $\mu_{\theta}$ remains $(s,\mathbf{C})$-regular. We apply Lemma \ref{lemma15} at scale $\delta = 50\delta_{2}$ and constant $A = 50$ to the measure $\mu_{\theta}$, the set $K_{2}$, and with parameters
 \begin{displaymath} N := \delta_{2}^{-\sigma_{2}} \quad \text{and} \quad M := (\delta_{2}^{1 - 2p_{0}})^{-\sigma_{2} - \zeta_{1}} \quad \text{and} \quad \Delta := 50\bar{\delta} = 50\delta_{2}^{2p_{0}} \quad \text{and} \quad \kappa := \frac{\epsilon_{n}}{4\mathbf{C}}. \end{displaymath}
 Note that $M \leq N$, because indeed
 \begin{displaymath} \frac{N}{M} = \delta_{2}^{-2p_{0}\sigma_{2} + \zeta_{1} - 2p_{0}\zeta_{1}} = \bar{\delta}^{-\sigma_{2} + \zeta_{1}/(2p_{0}) - \zeta_{1}} \geq \bar{\delta}^{-\sigma_{2} + \zeta_{1}/(2p_{0})}, \end{displaymath}
and here $\sigma_{2} \geq \sigma - \eta/2$ and $\zeta_{1}/(2p_{0}) \leq \eta/4$, recall \eqref{form77a} and \eqref{form76a}.
 
The conclusion of Lemma \ref{lemma15} with these parameters is that
 \begin{align*} \frac{\epsilon_{n}}{2\mathbf{C}} & \stackrel{\eqref{form62}}{\leq} \mu_{\theta}(B(1) \cap H_{\theta}(K_{2},\delta_{2}^{-\sigma_{2}},[50\delta_{2},50]))\\
& \leq 2\mu_{\theta}(B(1) \cap H_{\theta}(K_{2},(\delta_{2}^{1 - 2p_{0}})^{-\sigma_{2} - \zeta_{1}},[150\delta_{2},150\bar{\delta}]) + \kappa\\
&\qquad + \mu_{\theta}(B(1) \cap H_{\theta}(K_{2},\mathbf{c}(\kappa)\bar{\delta}^{-\sigma_{2} + \zeta_{1}/(2p_{0})},[50\bar{\delta},50]), \end{align*}
where $\mathbf{c}(\kappa) = c\kappa^{2}/\mathbf{C}^{3} \in (0,1]$ according to Lemma \ref{lemma15}. When this inequality is $\nu$-integrated over $\theta \in \bar{S}$, and we apply \eqref{form58} (with $q = 1$) to the first term, we find
\begin{align*} \frac{\epsilon_{n}}{2\mathbf{C}}\nu(\bar{S}) & \leq 2\int_{S^{1}}\mu_{\theta}(B(1) \cap H_{\theta}(K_{2},(\delta_{2}^{1 - 2p_{0}})^{-\sigma_{2} - \zeta_{1}},[150\delta_{2},150\bar{\delta}])) \, d\nu(\theta) + \kappa \nu(\bar{S})\\
&\qquad + \int_{\bar{S}} \mu_{\theta}(B(1) \cap H_{\theta}(K_{2},\mathbf{c}(\kappa)\bar{\delta}^{-\sigma_{2} + \zeta_{1}/(2p_{0})},[50\bar{\delta},50])) \, d\nu(\theta)\\
& \leq 2 \cdot 50^{s}\mathbf{C}\epsilon_{n + 1} + \frac{\epsilon_{n}}{4\mathbf{C}}\nu(\bar{S})\\
&\qquad + \int_{\bar{S}} \mu_{\theta}(B(1) \cap H_{\theta}(K_{2},\mathbf{c}(\kappa)\bar{\delta}^{-\sigma_{2} + \zeta_{1}/(2p_{0})},[50\bar{\delta},50])) \, d\nu(\theta).  \end{align*}
Here $\epsilon_{n + 1} \leq c(\mathbf{C},\mathfrak{d},\tau)\epsilon_{n}^{2}$ according to the choice we made at \eqref{form63}. Since on the other hand $\nu(\bar{S}) \gtrsim_{\mathfrak{d},\tau} \nu(S) \gtrsim \epsilon_{n}/\mathbf{C}$ (recall Lemma \ref{lemma8}), we now see that neither of the terms $2 \cdot 50^{s}\mathbf{C}\epsilon_{n + 1}$ nor $\nu(\bar{S})\epsilon_{n}/(4\mathbf{C})$ can dominate the left hand side. Consequently, and recalling now that $\mu_{\theta} = (\mu_{2})|_{L_{2,\theta}}$ and $\bar{\nu} = \nu|_{\bar{S}}$, the third term satisfies
\begin{equation}\label{form75} \int_{S^{1}} \mu_{2}(B(1) \cap H_{\theta}(K_{2},\mathbf{c}(\kappa)\bar{\delta}^{-\sigma_{2} + \zeta_{1}/(2p_{0})},[50\bar{\delta},50]) \cap L_{2,\theta}) \, d\bar{\nu}(\theta) \gtrsim \frac{\epsilon_{n}}{\mathbf{C}}\nu(\bar{S}) \gtrsim_{\mathbf{C},\mathfrak{d},\tau} \epsilon_{n}^{2}. \end{equation}
Recall further that here $n \leq 2\zeta^{-1}$, so in fact $\epsilon_{n} \gtrsim_{\mathbf{C},\mathfrak{d},\tau,\zeta} \epsilon_{1}$. Recalling the definition of $\epsilon_{1}$ from \eqref{form70}, the right hand side of \eqref{form75} is $\gtrsim_{\mathbf{C},\epsilon,\mathfrak{d},\eta,\tau,\zeta} 1$. 

This is nearly what we promised to deliver at \eqref{form74}, copied below for ease of reference:
\begin{equation}\label{form76} \int_{S^{1}} \bar{\mu}(B(1) \cap H_{\theta}(\bar{K},\bar{\delta}^{-\bar{\sigma}},[\bar{\delta},1]) \cap \bar{L}_{\theta}) \, d\bar{\nu}(\theta) =: \bar{\epsilon} \gtrsim_{\mathbf{C},\epsilon,\zeta,\eta,\mathfrak{d},\tau} 1. \end{equation}
We perform one more renormalisation to bring \eqref{form75} even closer to \eqref{form74}. Namely, write
\begin{displaymath} \bar{\mu} := \mu_{2}^{B(50)}, \quad \bar{K} := T_{B(50)}(K_{2}) = \spt \bar{\mu} \quad \text{and} \quad \bar{L}_{\theta} := T_{B(50)}(L_{2,\theta}). \end{displaymath}
Then, using Lemma \ref{lemma7},
\begin{align*} \bar{\mu}(B(1) \cap & H_{\theta}(\bar{K},\mathbf{c}(\kappa)\bar{\delta}^{-\sigma_{2} + \zeta_{1}/(2p_{0})},[\bar{\delta},1] \cap \bar{L}_{\theta})\\
& = 50^{-s} \mu_{2}(B(50) \cap H_{\theta}(K_{2},\mathbf{c}(\kappa)\bar{\delta}^{-\sigma_{2} + \zeta_{1}/(2p_{0})},[50\bar{\delta},50])\\
& \gtrsim \mu_{2}(B(1) \cap H_{\theta}(K_{2},\mathbf{c}(\kappa)\bar{\delta}^{-\sigma_{2} + \zeta_{1}/(2p_{0})},[50\bar{\delta},50]). \end{align*} 
Consequently, \eqref{form75} gives
\begin{equation}\label{form84} \int_{S^{1}} \bar{\mu}(B(1) \cap H_{\theta}(\bar{K},\mathbf{c}(\kappa)\bar{\delta}^{-\sigma_{2} + \zeta_{1}/(2p_{0})},[\bar{\delta},1] \cap \bar{L}_{\theta}) \, d\bar{\nu}(\theta)  \gtrsim_{\mathbf{C},\mathfrak{d},\tau} \epsilon_{n}^{2} \gtrsim_{\mathbf{C},\mathfrak{d},\epsilon,\eta,\tau,\zeta} 1. \end{equation}
To see that this implies \eqref{form76} for some $\bar{\sigma} \geq \sigma - \eta$ (as claimed in Proposition \ref{mainTechnicalPropRestated}), note that
\begin{displaymath} \sigma_{2} - \zeta_{1}/(2p_{0}) \stackrel{\eqref{form76a}}{\geq} \sigma - \zeta_{1}/(2p_{0}) - \eta/2 \stackrel{\eqref{form77a}}{\geq} \sigma - 3\eta/4. \end{displaymath}
In particular,
\begin{displaymath} \bar{\delta}^{-\bar{\sigma}} := \mathbf{c}(\kappa)\bar{\delta}^{-\sigma_{2} + \zeta_{1}/(2p_{0})} \geq \bar{\delta}^{-\sigma + \eta}, \end{displaymath}
provided that $\bar{\delta} > 0$ is small enough in terms of $\eta,\sigma$, and $\kappa = \epsilon_{n}/(4\mathbf{C})$ (which is true if $\delta > 0$ is small enough). If $\delta > 0$ is small enough, the definition above of $\bar{\sigma}$ also shows that
\begin{equation}\label{form82} \bar{\sigma} \geq \sigma_{2} - \frac{\zeta_{1}}{p_{0}} \geq \sigma_{2} - \frac{\zeta_{1}}{\min \mathcal{Q}_{0}}. \end{equation}
Clearly also $\bar{\sigma} \leq \sigma_{2} \leq s$, so $\bar{\sigma} \in [\sigma - \eta,s]$, as claimed in Proposition \ref{mainTechnicalPropRestated}.

We have now formally proved \eqref{form76}, but we still need to check that the sets $\bar{L}_{\theta}  = T_{B(50)}(L_{2,\theta})$ satisfy what we claimed in Proposition \ref{mainTechnicalPropRestated}, namely
\begin{displaymath} |\{0 \leq j \leq \bar{N} - 1 : y \in H_{\theta}(\bar{K},\Delta^{-\sigma_{0}},[\tfrac{A}{50}\Delta^{j + 1},\tfrac{A}{50}\Delta^{j}])\}| \leq C(\mathfrak{d},\tau)\eta^{-1}\lambda \cdot \bar{N} \end{displaymath}
for all $y \in \bar{L}_{\theta}$. This is based on property \eqref{form54} of the set $L_{2,\theta}$. Fix $y = T_{B(50)}(L_{2,\theta}) \in \bar{L}_{\theta}$. Then, recalling that $\bar{K} = T_{B(50)}(K_{2})$, it follows from Lemma \ref{lemma7} that
\begin{displaymath} y \in H_{\theta}(\bar{K},\Delta^{-\sigma_{0}},[\tfrac{A}{50}\Delta^{j + 1},\tfrac{A}{50}\Delta^{j}]) \quad \Longleftrightarrow \quad x \in H_{\theta}(K_{2},\Delta^{-\sigma_{0}},[A\Delta^{j + 1},A\Delta^{j}]). \end{displaymath}
Then \eqref{form54} yields, in combination with \eqref{form78},
\begin{displaymath} |\{0 \leq j \leq \bar{N} - 1 : y \in H_{\theta}(\bar{K},\Delta^{-\sigma_{0}},[\tfrac{A}{50}\Delta^{j + 1},\tfrac{A}{50}\Delta^{j}])\}| \leq 50\eta^{-1}\lambda \cdot N_{2} = \frac{50\eta^{-1}}{2p_{0}}\lambda \cdot \bar{N}. \end{displaymath}
Here $1/p_{0} \leq 1/\min \mathcal{Q}_{0} \lesssim_{\mathfrak{d},\tau} 1$, so the right hand side is $\leq C(\mathfrak{d},\tau)\eta^{-1}\lambda \cdot \bar{N}$, as desired. 

We are almost done with the proof of Proposition \ref{mainTechnicalPropRestated}: we are only missing the estimate \eqref{form79}, copied here for convenience:
\begin{equation}\label{form81} \int_{S^{1}} \bar{\mu}(B(1) \cap H_{\theta}(\bar{K},(\bar{\delta}^{q - p})^{-\bar{\sigma} - \zeta},[\bar{\delta}^{q},\bar{\delta}^{p}]) \cap \bar{L}_{\theta}) \, d\bar{\nu}(\theta) \leq \tfrac{\bar{\epsilon}}{10}, \end{equation}
for all pairs $p,q \in \{0,\tfrac{1}{2},1\}$ with $p < q$. Here $\bar{\epsilon}$ is the value of the integral in \eqref{form84}, so in particular $\bar{\epsilon} \gtrsim_{\mathbf{C},\mathfrak{d},\tau} \epsilon_{n}^{2}$. (Actually \eqref{form81} holds \emph{a fortiori} for $\nu$ in place of $\bar{\nu}$.)

Proving \eqref{form81} is based on the fact that $\epsilon_{n + 1}$ is much smaller than $\epsilon_{n}$, and \eqref{form58}, copied below for ease of reference:
\begin{equation}\label{form80} \int_{S^{1}} \mu_{2}(B(50) \cap H_{\theta}(K_{2},(\delta_{2}^{q - p})^{-\sigma_{2} - \zeta_{1}},[50\delta_{2}^{q},50\delta_{2}^{p}]) \cap L_{2,\theta}) \, d\nu(\theta) \leq 50^{s}\mathbf{C}\epsilon_{n + 1}, \end{equation}
where $p,q \in \mathcal{Q}$ with $p < q$. To see the connection between \eqref{form80} and \eqref{form81}, recall that $\bar{\mu} = \mu_{2}^{B(50)}$, and $\bar{\delta} = \delta_{2}^{2p_{0}}$ for some $p_{0} \in \mathcal{Q}_{0}$. Let us verify \eqref{form81} carefully in the case $(p,q) = (\tfrac{1}{2},1)$, the other two cases $(p,q) = (0,\tfrac{1}{2})$ and $(p,q) = (0,1)$ being similar.

For $(p,q) = (\tfrac{1}{2},1)$, note that 
\begin{displaymath} \bar{\delta}^{q - p} = \bar{\delta}^{1/2} = \delta_{2}^{2p_{0} - p_{0}} \quad \text{and} \quad \bar{\delta}^{q} = \delta_{2}^{2p_{0}} \quad \text{and} \quad \bar{\delta}^{p} = \delta_{2}^{p_{0}}, \end{displaymath}
where $p_{0},2p_{0} \in \mathcal{Q}$ by the definition of $\mathcal{Q}$, recall \eqref{rationals} (and that $p_{0} \in \mathcal{Q}_{0}$). So, \eqref{form81} in the case $(p,q) = (\tfrac{1}{2},1)$ can be written equivalently as
\begin{equation}\label{form83} \int_{S^{1}} \bar{\mu}(B(1) \cap H_{\theta}(\bar{K},(\delta_{2}^{2p_{0} - p_{0}})^{-\bar{\sigma} - \zeta},[\delta_{2}^{2p_{0}},\delta_{2}^{p_{0}}]) \cap \bar{L}_{\theta}) \, d\bar{\nu}(\theta) \leq \tfrac{\bar{\epsilon}}{10}. \end{equation}
Moreover, using \eqref{form82} and then \eqref{zetaOne},
\begin{displaymath} \bar{\sigma} + \zeta \geq \sigma_{2} + \zeta - \zeta_{1}/(\min \mathcal{Q}_{0}) \geq \sigma_{2} + \zeta_{1}. \end{displaymath}
Therefore $H_{\theta}(\bar{K},(\delta_{2}^{2p_{0} - p_{0}})^{-\bar{\sigma} - \zeta},\ldots) \subset H_{\theta}(\bar{K},(\delta_{2}^{2p_{0} - p_{0}})^{-\sigma_{2} - \zeta_{1}},\ldots)$. Combining this with $\bar{\mu} = \mu_{2}^{B(50)}$ and $\bar{L}_{\theta} = T_{B(50)}(L_{2,\theta})$, the integral in \eqref{form83} is bounded from above by
\begin{displaymath} 50^{s} \int_{S^{1}} \mu_{2}(B(50) \cap H_{\theta}(K_{2},(\delta_{2}^{2p_{0} - p_{0}})^{-\sigma_{2} - \zeta_{1}},[50\delta_{2}^{2p_{0}},50\delta_{2}^{p_{0}}]) \cap L_{2,\theta}) \, d\nu(\theta) \stackrel{\eqref{form80}}{\lesssim} \mathbf{C}\epsilon_{n + 1}. \end{displaymath} 
Since $\epsilon_{n + 1} \leq c(\mathbf{C},\mathfrak{d},\tau)\epsilon_{n}^{2}$, the upper bound can be taken to be much smaller than the value of the integral \eqref{form84}, which was our definition of $\bar{\epsilon}$. This concludes the proof of \eqref{form81}, and therefore the proof of Proposition \ref{mainTechnicalPropRestated}. \end{proof}

\section{Proof of the main proposition}\label{s:mainProof}

\subsection{Recap and applying Proposition \ref{mainTechnicalProp}}\label{s:recap} With the proof of Proposition \ref{mainTechnicalProp} finally behind us, we pick up the proof of Proposition \ref{mainProp} where we left off in the beginning of Section \ref{s:statements}. We start by recapping the progress so far, but this time paying proper attention to constants, and then applying Proposition \ref{mainTechnicalProp}.

Recall that by the hypothesis of Proposition \ref{mainProp}, the constants $\sigma,\sigma_{0} > 0$ satisfy $\sigma \leq s$ and $\sigma_{0}(1 - \tfrac{\sigma}{4}) < \sigma$, which can be equivalently written 
\begin{equation}\label{form96a} 2\cdot \tfrac{\sigma \sigma_{0}}{8} + s - \sigma_{0} > s - \sigma. \end{equation}
This form shows that we can choose two constants $\rho,\lambda > 0$ in a manner depending only on $\sigma,\sigma_{0}$, and $\tau$, such that the following constraints are satisfied:
\begin{equation}\label{choiceLambda} \sqrt{\lambda} \leq \sigma, \quad C(\sigma,\tau)\sqrt{\lambda} \leq \tfrac{\rho}{2} \quad \text{and} \quad \sigma \leq \tfrac{3}{2}(\sigma - \sqrt{\lambda}), \end{equation}
where $C(\sigma,\tau) := C(1 + \tfrac{\sigma}{4},\tau) \geq 1$ is the constant from Proposition \ref{mainTechnicalProp} (more precisely see above \eqref{form68}) applied with $\mathfrak{d} := 1 + \tfrac{\sigma}{4}$, and
\begin{equation}\label{def:rho} 2\sigma_{0}(\tfrac{\sigma - \sqrt{\lambda}}{8} + 2\rho) + (s - \sigma_{0} - 100\sqrt{10\rho})(1 - 10\sqrt{\rho}) > s - (\sigma - \sqrt{\lambda}) + 2\rho,\end{equation}
This looks complicated, but \eqref{def:rho} with $\lambda = 0 = \rho$ is simply \eqref{form96a}. By continuity \eqref{def:rho} is also satisfied by $\lambda,\rho$ small enough (the parameter $\rho$ actually just depends on $\sigma,\sigma_{0}$, and not $\tau$). 

\begin{remark}\label{rem2} Note that \eqref{def:rho} remains valid if the term "$\sigma - \sqrt{\lambda}$" is replaced on both sides by $\bar{\sigma} \geq \sigma - \sqrt{\lambda}$ (this term has a positive sign on the left and a negative sign on the right). \end{remark}

In Section \ref{s:statements}, we had already made a counter assumption of the form $\iota(\sigma)[\delta] \geq \epsilon$ for some $\epsilon > 0$, and arbitrarily small $\delta > 0$. We had then defined
\begin{displaymath} \epsilon_{0} := \frac{\lambda \epsilon^{2}}{16\mathbf{C}}, \end{displaymath}
where $\lambda = \lambda(\sigma,\sigma_{0},\tau) > 0$ is the constant satisfying \eqref{choiceLambda}-\eqref{def:rho}. 
We had applied the hypothesis (of Proposition \ref{mainProp}) with constants $\epsilon_{0},\sigma_{0} > 0$. To be accurate, we had applied the hypothesis via Lemma \ref{lemma4}. The main conclusion was \eqref{form23}, repeated here:
\begin{equation}\label{form85} \int_{S^{1}} \mu(B(1) \cap H_{\theta}(K,\delta^{-\sigma},[\delta,1]) \cap L_{\theta}) \, d\nu(\theta) \geq \tfrac{\epsilon}{16}, \end{equation}
where $L_{\theta}$ is the "local low-multiplicity set"
\begin{displaymath} L_{\theta} = \{x \in B(1) \cap K : \tfrac{1}{N}|\{0 \leq j \leq N - 1 : x \in H_{\theta}(K,\Delta^{-\sigma_{0}},[500\Delta^{j + 1},500\Delta^{j}])\}| \leq \lambda\}, \end{displaymath}
and $\Delta \in (0,\tfrac{1}{500}]$ is a scale provided by the hypothesis of Proposition \ref{mainProp}, which we assume to be small in a manner depending on $\mathbf{C},\sigma,\sigma_{0},\tau$ (see below \eqref{form133} for the details).

Note that \eqref{form85} matches the hypothesis \eqref{form66} of Proposition \ref{mainTechnicalProp}, which we plan to apply with parameters $A = 500,\mathbf{C},\Delta,\epsilon,\lambda,\sigma,\sigma_{0},\tau$ already introduced, and additionally
\begin{equation}\label{form133} \eta := \sqrt{\lambda} \quad \text{and} \quad \zeta = \zeta(\Delta,\sigma,\sigma_{0},\tau) > 0 \quad \text{and} \quad \mathfrak{d} = 1 + \tfrac{\sigma}{4}. \end{equation}
The value of the parameter $\zeta > 0$ is determined by the "inverse" Theorem \ref{shmerkin}, applied with parameters $\rho = \rho(\sigma,\sigma_{0},\tau) > 0$ as in \eqref{def:rho}, and the integer $m \in \N$ such that $\Delta = 2^{-m}$. Given the constant $\rho = \rho(\sigma,\sigma_{0},\tau) > 0$ above, the application of Theorem \ref{shmerkin} requires that $m \in \N$ is large in a manner depending on $\rho$, hence $\sigma,\sigma_{0},\tau$. This can be assumed by taking $\Delta > 0$ sufficiently small in terms of $\sigma,\sigma_{0},\tau$. Later, when applying Proposition \ref{prop4}, we will also need $\Delta$ to be small in terms of $\mathbf{C}$. 

With these parameters, Theorem \ref{shmerkin} outputs a constant $\kappa = \kappa(m,\rho) = \kappa(\Delta,\sigma,\sigma_{0},\tau) > 0$. We now set
\begin{equation}\label{choiceZeta} \zeta := \min\{\tfrac{1}{2}\kappa,\tfrac{1}{10}\rho\}. \end{equation}

We apply Proposition \ref{mainTechnicalProp} with the constants listed above. The result is a new $(s,\mathbf{C})$-regular measure $\bar{\mu}$ with $\bar{K} = \spt \bar{\mu}$ a new scale $\bar{\delta} = \Delta^{\bar{N}}$, where $\bar{N} \gtrsim_{\sigma,\sigma_{0},\tau,\Delta} N$, and a new $(\tau,\mathbf{C})$-Frostman measure $\bar{\nu}$ (in fact a restriction of $\nu$) such that
\begin{equation}\label{form86} \int_{S^{1}} \bar{\mu}(B(1) \cap H_{\theta}(\bar{K},\bar{\delta}^{-\bar{\sigma}},[\bar{\delta},1]) \cap \bar{L}_{\theta}) \, d\bar{\nu}(\theta) = \bar{\epsilon} \gtrsim_{\mathbf{C},\Delta,\epsilon,\sigma,\sigma_{0},\tau} 1, \end{equation}
where $\bar{\sigma} \in [\sigma - \eta,s] = [\sigma - \sqrt{\lambda},s]$.

\begin{remark} Note that that we have no upper control for $\bar{\sigma}$, except $\bar{\sigma} \leq s$. However, we will only need that $\bar{\sigma}$ is bigger than $\sigma$, up to the small error $\sqrt{\lambda}$, and in particular "at least as positive as $\sigma$".
\end{remark} 

We recap the findings of Proposition \ref{mainTechnicalProp} with the current notation. Recalling that $\mathfrak{d} = 1 + \tfrac{\sigma}{4}$, and $\eta = \sqrt{\lambda}$, the set $\bar{L}_{\theta}$ is contained in
\begin{equation}\label{form95a} \{x \in \R^{2} : \tfrac{1}{\bar{N}} |\{0 \leq j \leq \bar{N} - 1 : x \in H_{\theta}(\bar{K},\Delta^{-\sigma_{0}},[10\Delta^{j + 1},10\Delta^{j}])\}| \leq C(\sigma,\tau)\sqrt{\lambda}\}. \end{equation}
Additionally, for $(p,q) \in \{(0,1),(0,\tfrac{1}{2}),(\tfrac{1}{2},1)\}$, it holds
\begin{equation}\label{form87} \int_{S^{1}} \bar{\mu}(B(1) \cap H_{\theta}(\bar{K},(\bar{\delta}^{q - p})^{-\bar{\sigma} - \zeta},[\bar{\delta}^{q},\bar{\delta}^{p}]) \cap \bar{L}_{\theta}) \, d\bar{\nu}(\theta) \leq \tfrac{\bar{\epsilon}}{10}, \end{equation}
and the measure $\bar{\nu}$ has "non-trivial branching between scales $\bar{\delta}^{\mathfrak{d}/2}$ and $\bar{\delta}^{1/2}$": if $I \subset S^{1}$ is a dyadic arc of length $\bar{\delta}^{1/2}$ with $\bar{\nu}(I) > 0$, then 
\begin{equation}\label{branching} \bar{\nu}(J) \leq \bar{\delta}^{c(\sigma,\tau)}\bar{\nu}(I), \qquad J \subset I, \, |I| \leq \bar{\delta}^{\mathfrak{d}/2}, \end{equation}
where $c(\sigma,\tau) > 0$ is a constant depending only on $\sigma,\tau$.

The properties listed between \eqref{form86}-\eqref{branching}, and the branching of $\bar{\nu}$, are the only pieces of information we will use in the sequel -- to obtain a contradiction. To simplify notation slightly, we will now "drop the bars" and assume that the original objects $K,\mu,\nu,\sigma,\epsilon$ and $L_{\theta}$ already have these properties. Thus, for the remainder of the paper, we may and will only refer back to the properties of $\bar{\delta},\bar{\epsilon},\bar{\sigma},\bar{\mu},\bar{\nu},\bar{L}_{\theta}$ between \eqref{form86}-\eqref{branching}.

When we write "assuming $\delta > 0$ small enough in terms of XYZ" in the sequel, this actually means "assuming $\bar{\delta}$ small enough in terms of XYZ". Since $\bar{\delta} = \delta^{\bar{N}/N}$ with $\bar{N} \gtrsim_{\sigma,\sigma_{0},\tau,\Delta} N$, such requirements can eventually be satisfied by taking the original $\delta$ small enough in terms of XYZ and additionally $\sigma,\sigma_{0},\tau,\Delta$. Similarly, one has to keep in mind that (e.g.) $A \gtrsim_{\epsilon} B$ with the new notation means the same as $A \gtrsim_{\mathbf{C},\Delta,\epsilon,\sigma,\sigma_{0},\tau} B$.

\subsection{Two good directions} For $\theta \in S^{1}$, let $B_{\theta} \subset \R^{2}$ be the union of the three sets
\begin{equation}\label{form91b} H_{\theta}(K,(\delta^{q - p})^{-\sigma - \zeta},[\delta^{q},\delta^{p}]) \cap L_{\theta}, \qquad p,q \in \{0,\tfrac{1}{2},1\}, \, p < q. \end{equation} 
Then, writing $K_{\theta} := K \cap H_{\theta}(K,\delta^{-\sigma},[\delta,1]) \cap L_{\theta} \, \setminus \, B_{\theta}$, 
\begin{align} \int_{S^{1}} \mu(B(1) \cap K_{\theta}) \, d\nu(\theta) & \geq \int_{S^{1}} \mu(B(1) \cap K \cap H_{\theta}(K,\delta^{-\sigma},[\delta,1]) \cap L_{\theta}) \, d\nu(\theta) \notag\\
&\label{form88} \quad - \int_{S^{1}} \mu(B(1) \cap B_{\theta}) \, d\nu(\theta) \geq \epsilon - \tfrac{3 \epsilon}{10} \geq \tfrac{\epsilon}{2}. \end{align}
Write $S := \{\theta \in S^{1} : \mu(B(1) \cap K_{\theta}) \geq \epsilon/(4\mathbf{C})\}$. It follows from \eqref{form88} that $\nu(S) \geq \epsilon/(4\mathbf{C})$.

Let $\mathcal{D}_{\sqrt{\delta}}$ be a cover of $S^{1}$ by dyadic arcs of length $\delta^{1/2}$. Since
\begin{displaymath} \tfrac{\epsilon}{4\mathbf{C}} \leq \nu(S) = \sum_{I \in \mathcal{D}_{\sqrt{\delta}}} \nu(I)\nu_{I}(S), \end{displaymath}
and $\nu(S^{1}) \leq \mathbf{C}$, there exists an arc $I \in \mathcal{D}_{\delta^{1/2}}$ with $\nu_{I}(S) \geq \epsilon/(4\mathbf{C}^{2})$. Write $S_{I} := I \cap S$. Then, since $\mu(B(1) \cap K_{\theta}) \geq \epsilon/(4\mathbf{C})$ for all $\theta \in S_{I}$, and using Cauchy-Schwarz,
\begin{align*} \frac{\epsilon^{2}}{16\mathbf{C}^{3}} & \leq \int_{S_{I}} \mu(B(1) \cap K_{\theta}) \, d\nu_{I}(\theta) \leq \int_{B(1)} \nu_{I}(\{\theta \in S_{I} : x \in K_{\theta}\}) \, d\mu(x)\\
& \leq \mathbf{C}^{2} \cdot \Big( \int_{B(1)} \nu_{I}(\{(\theta_{0},\theta_{1}) \in S_{I} \times S_{I} : x \in K_{\theta_{0}} \cap K_{\theta_{1}}\}) \, d\mu(x) \Big)^{1/2}\\
& = \mathbf{C}^{2} \cdot \Big( \iint_{S_{I} \times S_{I}} \mu(K_{\theta_{0}} \cap K_{\theta_{1}}) \, d\nu_{I}(\theta_{1}) \, d\nu_{I}(\theta_{0}) \Big)^{1/2}. \end{align*} 
On the other hand, we infer from \eqref{branching} and the upper bound $\mu(K_{\theta_{0}} \cap K_{\theta_{1}}) \leq \mu(B(1)) \leq \mathbf{C}$ that
\begin{displaymath} \int_{S_{I}}\int_{S_{I} \cap B(\theta_{0},\delta^{\mathfrak{d}/2})} \mu(K_{\theta_{0}} \cap K_{\theta_{1}}) \, d\nu_{I}(\theta_{1}) \, d\nu_{I}(\theta_{0}) \leq \mathbf{C} \cdot \delta^{c(\sigma,\tau)}. \end{displaymath}  
Therefore, if $\delta > 0$ is small enough in terms of $\mathbf{C},\epsilon,\sigma,\tau$, there must exist a pair of vectors $\theta_{0},\theta_{1} \in S_{I} \subset S$ such that $\theta_{1} \notin B(\theta_{0},\delta^{\mathfrak{d}/2})$ and such that $\mu(K_{\theta_{0}} \cap K_{\theta_{1}}) \gtrsim_{\mathbf{C},\epsilon} 1$. In particular,
\begin{equation}\label{form128} \delta^{\mathfrak{d}/2} \leq |\theta_{0} - \theta_{1}| \leq \delta^{1/2}. \end{equation}
We now fix such a pair of vectors $\theta_{0},\theta_{1} \in S$ for the rest of the proof, and write
\begin{equation}\label{form89} \bar{K} := K_{\theta_{0}} \cap K_{\theta_{1}}. \end{equation} 

\begin{notation} For the remainder of the proof, we adopt the notation $A \leq O \cdot B$ means the same as $A \lesssim_{\mathbf{C},\Delta,\epsilon,\sigma,\sigma_{0},\tau} B$ and $A \geq \omega \cdot B$ means the same as $A \gtrsim_{\mathbf{C},\Delta,\epsilon,\sigma,\sigma_{0},\tau}B$. The values of the "constants" $O$ and $\omega$ will vary from line to line. It is occasionally convenient to refer to the "constant $\omega$ from line XYZ", and this would be harder to do with the "$\lesssim$" notation. In particular, note that $\mu(\bar{K}) \geq \omega$. \end{notation}  

\subsection{Properties of the set $\bar{K}$} We will record some facts about the $\pi_{\theta_{0}}$ and $\pi_{\theta_{1}}$ projections of the set $\bar{K}$ at scales $\delta$ and $\delta^{1/2}$. These properties are crucially based on the fact that
\begin{equation}\label{form92a} \bar{K} \subset K_{\theta_{j}}, \qquad j \in \{0,1\}, \end{equation}
so on the one hand $\bar{K}$ is contained in the high-multiplicity sets $H_{\theta_{j}}(\bar{K},\delta^{-\sigma},[\delta,1])$ and the "local" low-multiplicity sets $L_{\theta}$, but on the other hand $\bar{K}$ is contained in the complement of the "super" high-multiplicity sets \eqref{form91b}. The first property we claim is the following:
\begin{equation}\label{form91a} |\pi_{\theta_{j}}(\bar{K})|_{\delta} \lesssim \mathbf{C} \delta^{\sigma - s}, \qquad j \in \{0,1\}. \end{equation}
To see \eqref{form91a}, fix $\theta \in \{\theta_{0},\theta_{1}\}$, and let $\mathcal{T}_{\theta}$ be a minimal cover of $\bar{K}$ by tubes of the form $\pi_{\theta}^{-1}\{I\}$, where $I \in \mathcal{D}_{\delta}(\R)$. Then, each $T \in \mathcal{T}_{\theta}$ contains a point $x_{T} \in \bar{K} \subset H_{\theta}(K,\bar{\delta}^{-\sigma},[\delta,1])$, hence $B(x,1) \cap \pi_{\theta}^{-1}\{\pi_{\theta}(x_{T})\} \subset B(2) \cap T$, and
\begin{displaymath} |K_{\delta} \cap B(2) \cap T|_{\delta} \geq \m_{K,e}(x \mid [\delta,1]) \geq \delta^{-\sigma}. \end{displaymath} 
This implies that
\begin{displaymath} |\mathcal{T}_{\theta}| \cdot \delta^{-\sigma } \lesssim |K_{\delta} \cap B(2)|_{\delta} \leq 2\mathbf{C} \cdot \delta^{-s}, \end{displaymath}
and \eqref{form91a} follows by rearranging. 

We next consider the projections of $\bar{K}$ at scale $\delta^{1/2}$, but we first do an initial reduction. Let $\mathcal{K}_{\delta^{1/2}}$ be a minimal cover of $\bar{K}$ by discs of radius $\tfrac{1}{5}\delta^{1/2}$, so in particular $\bar{K} \cap \mathbf{B} \neq \emptyset$ for all $\mathbf{B} \in \mathcal{K}_{\delta^{1/2}}$. (In the sequel, we mostly denote $\delta$-discs and $\delta$-tubes with the font $B,T$ and $\delta^{1/2}$-discs and tubes with bold font $\mathbf{B},\mathbf{T}$.)  Since $\bar{K} \subset K \cap B(1)$, and $\mu(\bar{K}) \geq \omega$, 
\begin{equation}\label{form97} \omega \cdot \delta^{-s/2} \leq |\mathcal{K}_{\delta^{1/2}}| \lesssim |K \cap B(1)|_{\delta^{1/2}} \leq \mathbf{C} \cdot \delta^{-s/2}. \end{equation}
Writing $\boldsymbol{\mu} := \mu(\bar{K}) \geq \omega$, a disc $\mathbf{B} \in \mathcal{K}_{\delta^{1/2}}$ is called \emph{heavy} if $\mu(\mathbf{B} \cap \bar{K}) \geq (\boldsymbol{\mu}/2\mathbf{C}) \cdot \delta^{s/2}$. Then, the total $\mu$ measure of the \emph{light} discs if bounded from above by $|\mathcal{K}_{\delta^{1/2}}| \cdot (\boldsymbol{\mu}/2\mathbf{C}) \cdot \delta^{s/2} \leq \boldsymbol{\mu}/2$. Therefore, if we replace $\bar{K}$ by the intersection
\begin{equation}\label{form111} \bar{K} \cap \bigcup_{\mathbf{B} \in \mathcal{K}_{\delta^{1/2}} \mathrm{\,heavy}} \mathbf{B}, \end{equation}
then $\mu(\bar{K}) \geq \boldsymbol{\mu}/2 \geq \omega$, and the key property \eqref{form92a} of $\bar{K}$ remains valid. We will not need other properties of $\bar{K}$ in the future so, without loss of generality, we may assume that all the discs in $\mathcal{K}_{\delta^{1/2}}$ are heavy.

Now, in order to consider the $\pi_{\theta_{0}}$ and $\pi_{\theta_{1}}$ projections of $\bar{K}$ at scale $\delta^{1/2}$, let $\mathcal{T}_{\delta^{1/2}}$ be a minimal cover of $\bar{K}$ by tubes $\mathbf{T} = \pi_{\theta_{0}}^{-1}(I)$ of width $|I| = \delta^{1/2}$. We first claim the following upper bound:
\begin{equation}\label{form99} |\pi_{\theta_{j}}(\bar{K})|_{\delta^{1/2}} \sim |\mathcal{T}_{\delta^{1/2}}| \leq O \cdot (\delta^{1/2})^{\sigma - s - \zeta}, \qquad j \in \{0,1\}. \end{equation} 
The idea goes as follows. Every tube $\mathbf{T} \in \mathcal{T}_{\delta^{1/2}}$ meets at least one (heavy!) ball $\mathbf{B}_{\mathbf{T}} \in \mathcal{K}_{\delta^{1/2}}$. We will shortly see that
\begin{equation}\label{form98} |\pi_{\theta_{j}}(\bar{K} \cap \mathbf{B})|_{\delta} \geq \omega \cdot (\delta^{1/2})^{\sigma - s + \zeta}, \qquad j \in \{0,1\}, \, \mathbf{B} \in \mathcal{K}_{\delta^{1/2}}. \end{equation}
Since the sets $\pi_{\theta_{0}}(\bar{K} \cap \mathbf{B}_{\mathbf{T}})$ (or $\pi_{\theta_{1}}(\bar{K} \cap \mathbf{B}_{\mathbf{T}})$) have bounded overlap as $\mathbf{T} \in \mathcal{T}_{\delta^{1/2}}$ varies, it follows from \eqref{form91a} and \eqref{form98} that
\begin{displaymath} \mathbf{C} \delta^{\sigma - s} \gtrsim |\pi_{\theta_{0}}(\bar{K})|_{\delta} \geq |\mathcal{T}_{\delta^{1/2}}| \cdot \omega \cdot (\delta^{1/2})^{\sigma - s + \zeta}, \end{displaymath}
and then \eqref{form99} follows by rearranging terms. Let us then prove \eqref{form98}. Fix $\mathbf{B} \in \mathcal{K}_{\delta^{1/2}}$, and keep in mind that $\mu(\bar{K} \cap \mathbf{B}) \geq \omega$, and $\mathbf{B}$ is a disc of radius $\tfrac{1}{5}\delta^{1/2}$. Instead of proving \eqref{form98} directly, we show that if $T_{\delta/4} = \pi_{\theta_{j}}^{-1}(I)$ is an arbitrary tube of width $|I| = \delta/4$, then
\begin{equation}\label{form100} |\mathbf{B} \cap \bar{K} \cap T_{\delta/4}|_{\delta/4} \lesssim (\delta^{1/2})^{-\sigma - \zeta}. \end{equation}
On the other hand, since $\mu(\mathbf{B} \cap \bar{K}) \geq \omega \cdot \delta^{s/2}$ for all $\mathbf{B} \in \mathcal{K}_{\delta^{1/2}}$, we have $|\mathbf{B} \cap \bar{K}|_{\delta} \geq \omega \cdot (\delta^{1/2})^{-s}$ by the $(s,\mathbf{C})$-regularity of $\mu$. Combining this lower bound with \eqref{form100} implies that it takes $\geq \omega \cdot (\delta^{1/2})^{\sigma - s + \zeta}$ tubes of the form $T_{\delta/4} = \pi_{\theta_{j}}^{-1}(I)$, with $|I| = \delta/4$, to cover $\bar{K} \cap \mathbf{B}$. This is equivalent to \eqref{form98}.

The proof of \eqref{form100} is based on $\bar{K} \subset \R^{2} \, \setminus \, H_{\theta_{j}}(K,(\delta^{1/2})^{-\sigma - \zeta},[\delta,\delta^{1/2}])$, or equivalently
\begin{equation}\label{form93a} |K_{\delta} \cap B(x,\delta^{1/2}) \cap \pi_{\theta_{j}}^{-1}\{\pi_{\theta}(x)\}|_{\delta} \leq (\delta^{1/2})^{-\sigma - \zeta}, \qquad x \in \bar{K}. \end{equation}
It is straightforward to check that if \eqref{form100} failed for some $\tfrac{1}{5}\delta^{1/2}$-disc $\mathbf{B}$ and some $\tfrac{1}{4}\delta$-tube $T$, then one could find a point $x \in \mathbf{B} \cap \bar{K}$ violating \eqref{form93a}. This proves \eqref{form99}.

We have now shown that $\bar{K}$ can be covered by $\leq O \cdot (\delta^{1/2})^{\sigma - s - \zeta}$ tubes (in the family $\mathcal{T}_{\delta^{1/2}}$) of width $\delta^{1/2}$ perpendicular to $\theta_{0}$ (or equivalently $\theta_{1}$; since $|\theta_{0} - \theta_{1}| \leq \delta^{1/2}$, this covering property at scale $\delta^{1/2}$ sees no difference between $\theta_{0}$ and $\theta_{1}$). This implies that the average tube in $\mathcal{T}_{\delta^{1/2}}$ meets $\geq \omega \cdot |\mathcal{K}_{\delta^{1/2}}| \cdot (\delta^{1/2})^{s - \sigma +  \zeta}$ discs in $\mathcal{K}_{\delta^{1/2}}$. It would be convenient if the statement above held for every tube in $\mathcal{T}_{\delta^{1/2}}$. This can be arranged by a standard pruning argument, as follows. 

We say that a tube $\mathbf{T} \in \mathcal{T}_{\delta^{1/2}}$ is \emph{heavy} if
\begin{equation}\label{form104} |\{\mathbf{B} \in \mathcal{K}_{\delta^{1/2}} : \mathbf{B} \cap \mathbf{T} \neq \emptyset\}| \geq \omega \cdot (\delta^{1/2})^{-\sigma + \zeta} \end{equation}
 for a suitable constant $\omega > 0$. Since $|\mathcal{K}_{\delta^{1/2}}| \geq \omega \cdot (\delta^{1/2})^{-s}$ by \eqref{form97}, and $|\mathcal{T}_{\delta^{1/2}}| \leq O \cdot (\delta^{1/2})^{\sigma - s - \zeta}$ as we saw in \eqref{form99}, the union of the non-heavy tubes in $\mathcal{T}_{\delta^{1/2}}$ intersects at most half of the discs in $\mathcal{K}_{\delta^{1/2}}$, assuming that the "$\omega$" constant in the definition of heaviness above is chosen appropriately. Therefore, by discarding at most half of the discs in $\mathcal{K}_{\delta^{1/2}}$, and restricting $\bar{K}$ further to the union of the remaining discs, we may assume that every tube in $\mathcal{T}_{\delta^{1/2}}$ satisfies \eqref{form104}.
 
 We write
\begin{equation}\label{form110} \bar{K}_{\mathbf{T}} := \mathop{\bigcup_{\mathbf{B} \in \mathcal{K}_{\delta^{1/2}}}}_{\mathbf{B} \cap \mathbf{T} \neq \emptyset} \bar{K} \cap \mathbf{B}, \qquad \mathbf{T} \in \mathcal{T}_{\delta^{1/2}}. \end{equation}
We next claim that for most tubes $\mathbf{T} \in \mathcal{T}_{\delta^{1/2}}$, it holds $|\pi_{\theta_{j}}(\bar{K}_{\mathbf{T}})|_{\delta} \leq (\delta^{1/2})^{\sigma - s - 2\zeta}$ for both $j \in \{0,1\}$. The only clue available is \eqref{form91a}, which controls the $\pi_{\theta_{j}}$-projections of the whole. For $j \in \{0,1\}$ fixed, this can be used to produce an upper bound on the number of tubes $\mathbf{T} \in \mathcal{T}_{\delta^{1/2}}$ such that
\begin{equation}\label{form107} |\pi_{\theta_{j}}(\bar{K}_{\mathbf{T}})|_{\delta} \geq (\delta^{1/2})^{\sigma - s - 2\zeta}. \end{equation} 
Namely, if $\mathcal{T}_{\delta^{1/2},j} \subset \mathcal{T}_{\delta^{1/2}}$ is the sub-family of those tubes for which \eqref{form107} holds, then it follows from \eqref{form91a}, and the bounded overlap of the projections $\pi_{\theta_{j}}(K_{\mathbf{T}})$, $\mathbf{T} \in \mathcal{T}_{\delta^{1/2}}$, that
\begin{equation}\label{form77a} |\mathcal{T}_{\delta^{1/2},j}| \cdot (\delta^{1/2})^{\sigma - s - 2\zeta} \lesssim |\pi_{\theta_{j}}(\bar{K})|_{\delta} \lesssim \mathbf{C} \delta^{\sigma - s}, \end{equation}
and hence
\begin{equation}\label{form108} |\mathcal{T}_{\delta^{1/2},j}| \leq (\delta^{1/2})^{\sigma - s + 2\zeta}, \end{equation}
assuming $\delta > 0$ is so small that the constants $\lesssim \mathbf{C}$ in \eqref{form77a} are bounded from above by $(\delta^{1/2})^{-\zeta}$. Recall that we want to show that $\pi_{\theta_{j}}(\bar{K}_{\mathbf{T}})$ is small -- more precisely satisfies the reverse of \eqref{form107}) for most tubes $\mathbf{T} \in \mathcal{T}_{\delta^{1/2}}$. This would follow from \eqref{form108} if we knew that 
\begin{equation}\label{form91} |\mathcal{T}_{\delta^{1/2}}| \geq \omega \cdot (\delta^{1/2})^{\sigma - s + \zeta}, \end{equation} 
and we claim that this is the case. The proof is similar to the proof of the lower bound in \eqref{form98}, but is this time based on $\bar{K} \subset \R^{2} \, \setminus \, H_{\theta_{0}}(K,(\delta^{1/2})^{-\sigma - \zeta},[\delta^{1/2},1])$, equivalently
\begin{displaymath} |K_{\delta^{1/2}} \cap B(x,1) \cap \pi_{\theta_{0}}^{-1}\{\pi_{\theta_{0}}(x)\}|_{\delta^{1/2}} \leq (\delta^{1/2})^{-\sigma - \zeta}, \qquad x \in \bar{K}. \end{displaymath}
In fact, repeating the argument for \eqref{form100}, this inequality shows that if $\mathbf{T}_{\delta^{1/2}/4} = \pi_{\theta_{0}}^{-1}(J)$ is any tube of width $|J| = \tfrac{1}{4}\delta^{1/2}$ perpendicular to $\theta_{0}$, then
\begin{displaymath} |B(y,\tfrac{1}{4}) \cap \bar{K} \cap \mathbf{T}_{\delta^{1/2}/4}|_{\delta^{1/2}/4} \lesssim (\delta^{1/2})^{-\sigma - \zeta}, \qquad y \in \R^{2}. \end{displaymath}
Since on the other hand $\mathcal{T}_{\delta^{1/2}}$ is a cover of $\bar{K}$ by $\delta^{1/2}$-tubes, and $|B(y,\tfrac{1}{4}) \cap \bar{K}|_{\delta^{1/2}} \geq \omega \cdot \delta^{-s/2}$ for some $y \in B(1)$, one can deduce \eqref{form91}.

Combining \eqref{form108}-\eqref{form91}, we see that if $\delta > 0$ is small enough in terms of $\omega,\zeta$, then only a small fraction of the tubes in $\mathcal{T}_{\delta^{1/2}}$ lies in $\mathcal{T}_{\delta^{1/2},0} \cup \mathcal{T}_{\delta^{1/2},1}$. In particular, we may find a tube $\mathbf{T}_{0} \in \mathcal{T}_{\delta^{1/2}} \, \setminus \, [\mathcal{T}_{\delta^{1/2},0} \cup \mathcal{T}_{\delta^{1/2},1}]$. We finally gather the relevant properties of $\mathbf{T}_{0}$ and 
\begin{equation}\label{defK0} K_{0} := \bar{K}_{\mathbf{T}_{0}} \end{equation}
for future reference:
\begin{itemize}
\item[(K1) \phantomsection \label{K1}] $|\{\mathbf{B} \in \mathcal{K}_{\delta^{1/2}} : \mathbf{B} \cap \mathbf{T}_{0} \neq \emptyset\}| \geq (\delta^{1/2})^{-\sigma + 2\zeta}$ by \eqref{form104}. All the discs $\mathbf{B} \in \mathcal{K}_{\delta^{1/2}}$ here are heavy, that is, $\mu(\mathbf{B} \cap \bar{K}) \geq \omega \cdot \delta^{s/2}$.
\item[(K2) \phantomsection \label{K2}] The set $K_{0}$ is a subset of $\bar{K} = K_{\theta_{0}} \cap K_{\theta_{1}}$, and has small $\pi_{\theta_{j}}$-projections at scale $\delta$ in the sense that $|\pi_{\theta_{j}}(K_{0})|_{\delta} \leq (\delta^{1/2})^{\sigma - s - 2\zeta}$ for both $j \in \{0,1\}$. 
\end{itemize}

\subsection{Statement the inverse theorem}

We pause the main line of the proof for a moment to introduce the second author's inverse theorem \cite[Theorem 2.1]{Sh}, and related notation. 

\begin{definition}[$\delta$-measures and $L^{2}$-norms]\label{def:deltaMeasures} Let $\delta \in 2^{-\N}$ be a dyadic rational. Then, any probability measure supported on the discrete set $\delta \cdot \Z \cap [0,1)$ is called a \emph{$\delta$-measure}. The $L^{2}$-norm of a $\delta$-measure $\mu$ is defined by
\begin{displaymath} \|\mu\|_{2} := \left(\sum_{z \in \delta \cdot \Z} \mu(\{z\})^{2} \right)^{1/2}. \end{displaymath}
\end{definition}

\begin{thm}[Inverse theorem]\label{shmerkin} For every $\rho > 0$ there exists $m_{0} = m_{0}(\rho) > 0$ such that the following holds for all $m \geq m_{0}$. There exists $\kappa = \kappa(m,\rho) > 0$ and $N_{0} = N_{0}(m,\rho) \in \N$ such that the following holds for all $N \geq N_{0}$. Let $\delta = (2^{-m})^{N}$, and let $\eta_{1},\eta_{2}$ be $\delta$-measures such that
\begin{equation}\label{inverseHyp} \|\eta_{1} \ast \eta_{2}\|_{L^{2}} \geq \delta^{\kappa}\|\eta_{1}\|_{L^{2}}. \end{equation}
Then, there exist sets $A \subset \spt \eta_{1}$ and $B \subset \spt \eta_{2}$ such that
\begin{equation}\label{form131} \|\eta_{1}|_{A}\|_{2} \geq \delta^{\rho}\|\mu\|_{2} \quad \text{and} \quad \eta_{2}(B) \geq \delta^{\rho}, \end{equation} 
and both $A,B$ are $\{2^{-jm}\}_{j = 0}^{N - 1}$-uniform. This means the following:
\begin{itemize}
\item[\textup{(A)}] there is a sequence $(R_{j}^{1})_{j = 0}^{N - 1} \subset \{1,\ldots,2^{m}\}^{N}$, such that for all $0 \leq j \leq N - 1$,
\begin{displaymath} |A \cap I|_{2^{-(j + 1)m}} = R_{j}^{1}, \qquad I \in \mathcal{D}_{2^{-jm}}(A). \end{displaymath}
\item[\textup{(B)}] there is a sequence $(R_{j}^{2})_{j = 0}^{N - 1} \subset \{1,\ldots,2^{m}\}^{N}$, such that for all $0 \leq j \leq N - 1$,
\begin{displaymath} |B \cap I|_{2^{-(j + 1)m}} = R_{j}^{2}, \qquad I \in \mathcal{D}_{2^{-jm}}(B). \end{displaymath}
\end{itemize}
Moreover, $\{j : R_{j}^{2} > 1\} \subset \{j : R_{j}^{1} \geq 2^{(1 - \rho)m}\} =: \mathcal{S}$, and
\begin{equation}\label{form116} m|\calS| \geq \log \|\eta_{2}\|_{L^{2}}^{-2} + \rho \log_{2} \delta. \end{equation}
\end{thm}

\subsection{Finding a product structure inside $K_{0}$} We then return to the main line of the argument. Based on the properties \nref{K1}-\nref{K2} of the set $K_{0} = \bar{K}_{\mathbf{T}_{0}}$ defined at \eqref{defK0}, we will construct two $\delta$-measures to which we can apply Theorem \ref{shmerkin}. This will allow us to find a contradiction, and complete the proof of Proposition \ref{mainProp}.

 We assume with no loss of generality that $\theta_{0} = (1,0)$, so $\pi_{\theta_{0}} =: \pi_{1}$ is the projection to the $x$-axis. Later on, it will also simplify notation to assume that the tube $\mathbf{T}_{0}$ (which is perpendicular to $\theta_{0}$) satisfies
\begin{displaymath} 2\mathbf{T}_{0} = \pi_{1}^{-1}(I_{0}) = I_{0} \times \R, \quad \text{where} \quad I_{0} := [0,2\delta^{1/2}]. \end{displaymath}
We begin to set the scene for applying Theorem \ref{shmerkin}. Let $\mathcal{T}_{\delta}$ be a minimal cover of $K_{0}$ with $(\delta/4)$-tubes of the form $\pi_{1}^{-1}(I)$, where $I \in \mathcal{D}_{\delta/4}(\R)$. Then the tubes in $\mathcal{T}_{\delta}$ are all contained in $2\mathbf{T}_{0} = I_{0} \times \R$, recalling from \eqref{form110} that $K_{0}$ is the union of certain $\tfrac{1}{5}\delta^{1/2}$-discs intersecting the $\delta^{1/2}$-tube $\mathbf{T}_{0} \in \mathcal{T}_{\delta^{1/2}}$. Moreover,
\begin{equation}\label{form109} |\mathcal{T}_{\delta}| \leq (\delta^{1/2})^{\sigma - s - 2\zeta} \end{equation}
according to property \nref{K2}. We also let $\mathcal{K}_{\delta^{1/2}}(\mathbf{T}_{0}) := \{\mathbf{B} \in \mathcal{K}_{\delta^{1/2}} : \mathbf{B} \cap \mathbf{T}_{0} \neq 0\}$, and we write
\begin{equation}\label{form118} \M := |\mathcal{K}_{\delta^{1/2}}(\mathbf{T}_{0})| \stackrel{\mathrm{\nref{K1}}}{\geq} (\delta^{1/2})^{-\sigma + 2\zeta}. \end{equation}
We now want to remove from $\mathcal{T}_{\delta}$ a few tubes which have a relatively sparse intersection with $K_{0}$. To make this precise, and motivate the numerology, we introduce some notation. For $\mathbf{B} \in \mathcal{K}_{\delta^{1/2}}(\mathbf{T}_{0})$, let 
\begin{equation}\label{form94a} K_{\mathbf{B}} := K_{0} \cap \mathbf{B} = \bar{K} \cap \mathbf{B}, \end{equation}
and let $\mathcal{K}_{\delta}(\mathbf{B})$ be a minimal cover of $K_{\mathbf{B}}$ by discs of radius $\tfrac{1}{4}\delta$. Recall that each disc $\mathbf{B} \in \mathcal{K}_{\delta^{1/2}}$ satisfies $\mu(\mathbf{B} \cap \bar{K}) \geq \omega \cdot \delta^{s/2}$ (this was discussed at \eqref{form111}), so the $(s,\mathbf{C})$-regularity of $\mu$ implies
\begin{equation}\label{form112} \omega \cdot (\delta^{1/2})^{-s} \leq |\mathcal{K}_{\delta}(\mathbf{B})| \leq \mathbf{C} \cdot (\delta^{1/2})^{-s}, \qquad \mathbf{B} \in \mathcal{K}_{\delta^{1/2}}(\mathbf{T}_{0}). \end{equation}
By discarding at most $\tfrac{1}{2}$ of the $\mu$ measure of $K_{\mathbf{B}}$, we may assume that all the discs in $\mathcal{K}_{\delta}(\mathbf{B})$ are \emph{heavy}, meaning this time that $\mu(B_{\delta}) \geq \omega \cdot \delta^{s}$ for all $B \in \mathcal{K}_{\delta}(\mathbf{B})$.

We let $\mathcal{K}_{\delta}$ to be the union of all the families $\mathcal{K}_{\delta}(\mathbf{B})$, with $\mathbf{B} \in \mathcal{K}_{\delta^{1/2}}(\mathbf{T}_{0})$. Thus 
\begin{equation}\label{form114} \omega \cdot \M \cdot (\delta^{1/2})^{-s} \leq |\mathcal{K}_{\delta}| \leq \mathbf{C} \cdot \M \cdot (\delta^{1/2})^{-s}. \end{equation}
We also recall from \eqref{form100} that if $T = T_{\delta/4}$ is an arbitrary $(\delta/4)$-tube perpendicular to $\theta_{0}$, in particular if $T \in \mathcal{T}_{\delta}$, then
\begin{equation}\label{form113} |K_{\mathbf{B}} \cap T|_{\delta/4} \leq |\mathbf{B} \cap \bar{K} \cap T|_{\delta/4} \lesssim (\delta^{1/2})^{-\sigma - \zeta}. \end{equation}
Therefore, it is reasonable to define that $T \in \mathcal{T}_{\delta}$ is $\mathbf{B}$-\emph{dense} if 
\begin{equation}\label{form119} |K_{\mathbf{B}} \cap T|_{\delta} \geq (\delta^{1/2})^{-\sigma + 3\zeta}. \end{equation}
We say that $T \in \mathcal{T}_{\delta}$ is \emph{$K_{0}$-sparse} if
\begin{equation}\label{form141} |\{\mathbf{B} \in \mathcal{K}_{\delta^{1/2}}(\mathbf{T}_{0}) : T \text{ is $\mathbf{B}$-dense}\}| \leq \M \cdot \delta^{2\zeta}. \end{equation}
Otherwise $T$ is \emph{$K_{0}$-dense}. This numerology is also sensible, because each tube $T \in \mathcal{T}_{\delta}$ can meet at most $\M$ discs in $\mathcal{K}_{\delta^{1/2}}(\mathbf{T}_{0})$ (namely all of them). Let us count the number of discs $B \in \mathcal{K}_{\delta}$ which intersect some $K_{0}$-sparse tube. Using the uniform upper bound \eqref{form113}, every fixed $K_{0}$-sparse tube $T \in \mathcal{T}_{\delta}$ satisfies
\begin{displaymath} |\{B \in \mathcal{K}_{\delta} : T \cap B \neq \emptyset\}| \leq [\M \cdot \delta^{2\zeta}] \cdot (\delta^{1/2})^{-\sigma - \zeta} + \M \cdot (\delta^{1/2})^{-\sigma + 3\zeta} \leq 2 \M \cdot (\delta^{1/2})^{-\sigma + 3\zeta}. \end{displaymath}
Since the number of sparse tubes is also bounded from above by $|\mathcal{T}_{\delta}| \leq (\delta^{1/2})^{\sigma - s - 2\zeta}$ by \eqref{form109}, we conclude that the number of discs in $\mathcal{K}_{\delta}$ which meet some sparse tube is bounded from above by $\lesssim \M \cdot (\delta^{1/2})^{-s + \zeta}$. 

Recalling from \eqref{form114} that $|\mathcal{K}_{\delta}| \geq \omega \cdot \M \cdot (\delta^{1/2})^{-s}$, we may finally infer that if $\delta > 0$ is small enough, there exist $\geq \tfrac{1}{2} \cdot |\mathcal{K}_{\delta}|$ discs in $\mathcal{K}_{\delta}$ which intersect some $K_{0}$-dense tube in $\mathcal{T}_{\delta}$. After this observation, we discard all $K_{0}$-sparse tubes from $\mathcal{T}_{\delta}$ without changing notation; in other words, we assume in the sequel that all the tubes in $\mathcal{T}_{\delta}$ are $K_{0}$-dense.

We record the following converse to \eqref{form109}:
\begin{equation}\label{form115} |\mathcal{T}_{\delta}| \geq (\delta^{1/2})^{\sigma - s + 2\zeta}, \end{equation}
assuming that $\delta > 0$ is sufficiently small. Indeed, according to \eqref{form113}, every individual tube $T \in \mathcal{T}_{\delta}$ intersects $\lesssim \M \cdot (\delta^{1/2})^{-\sigma - \zeta}$ discs in $\mathcal{K}_{\delta}$. But since the tubes in $\mathcal{T}_{\delta}$ in total intersect $\geq \tfrac{1}{2} \cdot |\mathcal{K}_{\delta}| \geq \omega \cdot \M \cdot \delta^{-s/2}$ discs in $\mathcal{K}_{\delta}$, the lower bound \eqref{form115} follows.

Recall that the tubes in $\mathcal{T}_{\delta}$ have the form $T = \pi_{1}^{-1}(I_{T})$, where $I_{T} \in \mathcal{D}_{\delta}(\R)$ is a dyadic interval of side-length $\delta/4$. Therefore, the left end-points of the intervals $\{I_{T} : T \in \mathcal{T}_{\delta}\}$, form a certain finite subset $A_{1} \subset \tfrac{1}{4}\delta \cdot \Z$, and we record from \eqref{form109} that
\begin{equation}\label{form136} (\delta^{1/2})^{\sigma - s + 2\zeta} \stackrel{\eqref{form115}}{\leq} |A_{1}| = |\mathcal{T}_{\delta}| \leq (\delta^{1/2})^{\sigma - s - 2\zeta} \stackrel{\eqref{choiceZeta}}{\leq} (\delta^{1/2})^{\sigma - s - \rho}. \end{equation}
Ket $\Pi_{1}$ be the uniformly distributed probability measure on $A_{1}$. Since $T \subset 2\mathbf{T}_{0} = \pi_{1}^{-1}([0,2\delta^{1/2}])$ for all $T \in \mathcal{T}_{\delta}$, it holds
\begin{equation}\label{form125} \spt \Pi_{1} = A_{1} \subset [0,2\delta^{1/2}] \cap \tfrac{1}{4}\delta \cdot \Z. \end{equation}
This measure is a $(\delta/4)$-measure in the sense of Definition \ref{def:deltaMeasures}. We note that
\begin{equation}\label{form116} \|\Pi_{1}\|_{2} := \left(\sum_{a \in A_{1}} \Pi_{1}(\{a\})^{2} \right)^{1/2} = |A_{1}|^{-1/2} \stackrel{\eqref{form136}}{\leq} (\delta^{1/2})^{(s - \sigma - 2\zeta)/2}. \end{equation} 
We next define another discrete measure, which is associated with the $y$-coordinates of the discs $\mathbf{B} \in \mathcal{K}_{\delta^{1/2}}$. In fact, for every $\mathbf{B} \in \mathcal{K}_{\delta^{1/2}}(\mathbf{T}_{0})$, let 
\begin{displaymath} y_{\mathbf{B}} \in \tfrac{1}{4}\delta^{1/2} \cdot \Z \cap \pi_{2}(\mathbf{B}), \end{displaymath}
where $\pi_{2}(x,y) = y$ is the projection to the $y$-coordinate. This point exists, since $\pi_{2}(\mathbf{B})$ is an interval of length $\geq \tfrac{1}{3}\delta^{1/2}$. We also note that $\pi_{2}(\mathbf{B}) \subset [-2,2]$, since $\mathbf{B}$ intersects $\bar{K} \subset B(1)$. Let $A_{2} := \{y_{\mathbf{B}} : \mathbf{B} \in \mathcal{K}_{\delta^{1/2}}\}$, and let $\Pi_{2}$ be the uniformly distributed probability measure on $A_{2}$. Then 
\begin{equation}\label{form126} \spt \Pi_{2} = A_{2} \subset [-2,2] \cap \tfrac{1}{4}\delta^{1/2} \cdot \Z. \end{equation}
Thus $\Pi_{2}$ is a $\tfrac{1}{4}\delta^{1/2}$-measure in the sense of Definition \ref{def:deltaMeasures}, and
\begin{equation}\label{form117} \|\Pi_{2}\|_{2} = \left(\sum_{a \in A_{2}} \Pi_{2}(\{a\})^{2} \right)^{1/2} = |A_{2}|^{-1/2} \sim \M^{-1/2} \stackrel{\eqref{form118}}{\leq} (\delta^{1/2})^{(\sigma - 2\zeta)/2}. \end{equation} 
The product measure $\Pi := \Pi_{1} \times \Pi_{2}$ is a discrete probability measure supported on $A_{1} \times A_{2}$. We will next investigate a certain subset of $A_{1} \times A_{2}$ of substantial $\Pi$-measure. To define this subset, recall that that every point $x \in A_{1}$ was contained in the $\pi_{1}$-projection of a certain tube $T_{x} \in \mathcal{T}_{\delta}$. Similarly, recall that every point $y \in A_{2}$ was contained in the $\pi_{2}$-projection of a certain disc $\mathbf{B}_{y} \in \mathcal{K}_{\delta^{1/2}}$. With this notation, we define
\begin{equation}\label{form122} G := \{(x,y) \in A_{1} \times A_{2} : T_{x} \cap K_{\mathbf{B}_{y}} \neq \emptyset\}, \end{equation} 
recall from \eqref{form94a} that $K_{\mathbf{B}} = K_{0} \cap \mathbf{B}$. Morally, $G$ consists of those tube-disc pairs $(T,\mathbf{B}) \in \mathcal{T}_{\delta} \times \mathcal{K}_{\delta^{1/2}}(\mathbf{T}_{0})$, where $T$ intersects $\bar{K}$ inside $\mathbf{B}$. Note that if $x \in A_{1}$, then $T_{x}$ is a $K_{0}$-dense tube (see \eqref{form141} for the definition), so $T_{x}$ has non-empty intersection with $K_{\mathbf{B}}$ for $\geq \M \cdot \delta^{2\zeta}$ distinct discs $\mathbf{B} \in \mathcal{K}_{\delta^{1/2}}(\mathbf{T}_{0})$. In other words,
\begin{displaymath} |\pi_{1}^{-1}\{x\} \cap G| \geq \M \cdot \delta^{2\zeta}, \qquad x \in A_{1}, \end{displaymath}
and consequently
\begin{equation}\label{form120} \Pi(G) = \frac{|G|}{|A_{1}||A_{2}|} \gtrsim \delta^{2\zeta}. \end{equation}

\subsection{The $\pi_{\theta_{1}}$-projection} While constructing the sets $A_{1},A_{2},G$ above, all the arguments were based on the structure of $K_{0} = \bar{K}_{\mathbf{T}_{0}}$ relative to tubes perpendicular to $\theta_{0}$. Now, the plan is to exploit the information available for the $\pi_{\theta_{1}}$-projection in \nref{K2}, namely that
\begin{equation}\label{form123} |\pi_{\theta_{1}}(K_{0})|_{\delta} \leq (\delta^{1/2})^{\sigma - s - 2\zeta}. \end{equation}
To use this, we claim that there exists an absolute constant $C > 0$ such that
\begin{equation}\label{form121} \pi_{\theta_{1}}(G) \subset [\pi_{\theta_{1}}(K_{0})]_{C\delta}, \end{equation}
where $G \subset A_{1} \times A_{2}$ is the set defined in \eqref{form122}, and $[\cdots]_{r}$ is the $r$-neighbourhood. To prove \eqref{form121}, fix $(x,y) \in G$, so $T_{x} \cap K_{\mathbf{B}_{y}} \neq \emptyset$ by definition. Here $T_{x} = \pi_{1}^{-1}(I_{x}) = \pi_{e_{0}}^{-1}(I_{x})$ for some interval $I_{x} \in \mathcal{D}_{\delta/4}(\R)$. Then we know the following:
\begin{enumerate}
\item $x \in I_{x} = \pi_{1}(T_{x})$ and $y \in \pi_{2}(\mathbf{B}_{y})$ by definitions of $T_{x}$ and $\mathbf{B}_{y}$.
\item There exists a point $(x_{0},y_{0}) \in T_{x} \cap K_{\mathbf{B}_{y}} = T_{x} \cap K_{0} \cap \mathbf{B}_{y}$.
\end{enumerate}
It follows that the points $(x,y)$ and $(x_{0},y_{0})$ lie in the same (vertical) tube $T_{x}$, and $|y - y_{0}| \leq 2\delta^{1/2}$. Now, since $|\theta_{0} - \theta_{1}| \leq \delta^{1/2}$, one may check by elementary geometry that
\begin{displaymath} |\pi_{\theta_{1}}(x,y) - \pi_{\theta_{1}}(x_{0},y_{0})| \lesssim \delta. \end{displaymath}
But here $\pi_{\theta_{1}}(x_{0},y_{0}) \in \pi_{\theta_{1}}(K_{0})$, and hence $\pi_{\theta_{1}}(G)$ lies in the $C\delta$-neighbourhood of $\pi_{\theta_{1}}(K_{0})$. This proves \eqref{form121}. Combining \eqref{form123}-\eqref{form121}, we obtain
\begin{equation}\label{form124} |\pi_{e_{1}}(G)|_{\delta} \lesssim (\delta^{1/2})^{\sigma - s - 2\zeta}. \end{equation}
Combined with the lower bound \eqref{form120} for the $\Pi$-measure of $G$, we will shortly infer from \eqref{form124} a lower bound for the $L^{2}$-norm of the projection $\pi_{\theta_{1}}\Pi$. To make this perfectly precise, we will need an additional piece of notation. Start by writing
\begin{displaymath} \pi_{\theta_{1}}(x,y) = x \cos \varphi + y \sin \varphi = \cos \varphi \left( x + \tfrac{\sin \varphi}{\cos \varphi} y \right) =: c \cdot (x + \theta y) \end{displaymath}
for the parameter $\varphi \in [-\pi,\pi)$ satisfying $\theta_{1} = (\cos \varphi,\sin \varphi)$. In fact, since we assume that $e_{0} = (1,0)$, and $\delta^{\mathfrak{d}/2} \leq |\theta_{0} - \theta_{1}| \leq \delta^{1/2}$ by \eqref{form128}, we have $c = \cos \varphi \sim 1$ and 
\begin{equation}\label{form127} \delta^{\mathfrak{d}/2} \lesssim |\theta| \sim |\sin \varphi| \lesssim \delta^{1/2}. \end{equation}
For reasons to become apparent momentarily, we would prefer $\pi_{\theta_{1}}$ to map $\R^{2}$ inside the discrete set $\delta \cdot \Z$. So, let us, in place of $\pi_{\theta_{1}}$, consider the map $\bar{\pi}_{\theta} \colon \R^{2} \to \tfrac{1}{4}\delta \cdot \Z$,
\begin{displaymath} \bar{\pi}_{\theta} := [x] + [\theta y], \end{displaymath} 
where $[r] \in \delta \cdot \Z$ is determined by $r = [r] + c$, $0 \leq c < \tfrac{1}{4}\delta$. We claim, based on \eqref{form124}, that
\begin{equation}\label{form129} |\bar{\pi}_{\theta}(G)| \lesssim (\delta^{1/2})^{\sigma - s - 2\zeta}, \end{equation}
where the left hand side refers to cardinality. The reason is simply that 
\begin{displaymath} |\pi_{\theta_{1}}(x,y) - c \cdot \bar{\pi}_{\theta}(x,y)| = |c| \cdot (|x - [x]| + |\theta y - [\theta y]|) \lesssim \delta, \qquad (x,y) \in \R^{2}, \end{displaymath} 
so any $\delta$-cover for $\pi_{\theta_{1}}(G)$ can be used to construct a $\delta$-cover for $c \cdot \bar{\pi}_{\theta}(G)$ of comparable cardinality. Then the $\delta$-separation of the points in $\bar{\pi}_{\theta}(G)$ yields \eqref{form129}.

The benefit of considering the projection $\bar{\pi}_{\theta}$ is that the $\bar{\pi}_{\theta}$-projection of $\Pi$ can be expressed as the following convolution:
\begin{displaymath} \bar{\pi}_{\theta}\Pi = [\Pi_{1}] \ast [\theta\Pi_{2}] = \Pi_{1} \ast [\theta \Pi_{2}] \end{displaymath}
where $[\theta \Pi_{2}]$ refers to the push-forward of $\Pi_{2}$ under the map $y \mapsto [\theta y]$ (recall from \eqref{form125} that $\Pi_{1}$ is already supported on $\tfrac{1}{4}\delta \cdot \Z$, so $[\Pi_{1}] = \Pi$). Note that both $\Pi_{1}$ and $[\theta \Pi_{2}]$ are discrete measures supported on $\tfrac{1}{4}\delta \cdot \Z$, so the same is true for their convolution. From the facts $(\Pi_{1} \times \Pi_{2})(G) = \Pi(G) \gtrsim \delta^{2\zeta}$ (recall \eqref{form120}) and \eqref{form129}, we can deduce the following lower bound for the (discrete) $L^{2}$-norm of this convolution:
\begin{align*} \delta^{2\zeta} \lesssim \sum_{z \in \bar{\pi}_{\theta}(G)} \bar{\pi}_{\theta}(\Pi|_{G})(z) & \leq |\bar{\pi}_{\theta}(G)|^{1/2} \left(\sum_{z \in (\delta/4) \cdot \Z} \bar{\pi}_{\theta}(\Pi)(z)^{2} \right)^{1/2}\\
& \lesssim (\delta^{1/2})^{(\sigma- s - 2\zeta)/2} \cdot \|\Pi_{1} \ast [\theta \Pi_{2}]\|_{2}, \end{align*}
or in other words
\begin{equation}\label{form130} \|\Pi_{1} \ast [\theta \Pi_{2}]\|_{L^{2}} \gtrsim (\delta^{1/2})^{(s - \sigma + 10\zeta)/2} \stackrel{\eqref{form116}}{\geq} (\delta^{1/2})^{4\zeta} \cdot \|\Pi_{1}\|_{L^{2}}. \end{equation}
The lower bound \eqref{form130} will eventually place us in a position to apply the inverse Theorem \ref{shmerkin}. 
To prepare the ground for this, we record an upper bound for the $L^{2}$-norm of $[\theta \Pi_{2}]$. This is based on the following elementary observation: since $|\theta| \gtrsim \delta^{\mathfrak{d}/2}$ by \eqref{form127},
\begin{equation}\label{form132} |\{y \in \tfrac{1}{4}\delta^{1/2} \cdot \Z :  [\theta y] = z\}| \lesssim \delta^{(1 - \mathfrak{d})/2}, \qquad z \in \delta \cdot \Z. \end{equation}
Indeed, if $[\theta y_{1}] = [\theta y_{2}]$, then certainly $|\theta| \cdot |y_{1} - y_{2}| \leq \delta$, hence $|y_{1} - y_{2}| \lesssim \delta^{1 - \mathfrak{d}/2}$. But any fixed interval of radius $\sim \delta^{1 - \mathfrak{d}/2}$ contains $\lesssim \delta^{(1 - \mathfrak{d})/2}$ points from $\tfrac{1}{4}\delta^{1/2} \cdot \mathbb{Z}$, and this implies \eqref{form132}. Now, recall from \eqref{form126} that $\Pi_{2}$ was defined to be the uniform probability measure on the set $A_{2} \subset [-2,2] \cap \tfrac{1}{4}\delta^{1/2} \cdot \Z$, and $|A_{2}| \sim \M$. Hence, it follows from \eqref{form132}, and $|[\theta A_{2}]| \leq |A_{2}| \sim \M$, that
\begin{align} \|[\theta \Pi_{2}]\|_{2} & = \left( \sum_{z \in [\theta A_{2}]} [\theta \Pi_{2}](\{z\})^{2} \right)^{1/2} \lesssim \left( \M \cdot \M^{-2} \cdot \delta^{1 - \mathfrak{d}} \right)^{1/2} \notag\\
&\label{form151} = \delta^{(1 - \mathfrak{d})/2} \cdot \M^{-1/2} \stackrel{\eqref{form118}}{\lesssim} (\delta^{1/2})^{(\sigma - 2\zeta)/2 + 1 - \mathfrak{d}} = (\delta^{1/2})^{\sigma/8 - \zeta}. \end{align}  
In the final equality, we used the choice $\mathfrak{d} := 1 + \tfrac{\sigma}{4}$ at \eqref{form133}. One has to be careful, however, since the "$\sigma$" in \eqref{form151} is actually the $\bar{\sigma} \in [\sigma - \sqrt{\lambda},s]$ from Section \ref{s:recap}. However, for \eqref{form151} we need that $\bar{\sigma}/2 + 1 - \mathfrak{d} \geq \bar{\sigma}/8$, or equivalently $\mathfrak{d} \leq 1 + 3\bar{\sigma}/8$, and this is true because $\sigma \leq 3\bar{\sigma}/2$ by the choice of $\lambda$ at \eqref{choiceLambda}.

\subsection{The branching numbers of $\Pi_{1}$} We arrive at the final step of the proof. In \eqref{form130}, we have seen that $\eta_{1} = \Pi_{1}$ and $\eta_{2} = [\theta \Pi_{2}]$ are $(\delta/4)$-measures with the property 
\begin{equation}\label{form135} \|\eta_{1} \ast \eta_{2}\| \geq (\delta^{1/2})^{4\zeta}\|\eta_{1}\|_{L^{2}}. \end{equation}
This places us in a position to apply Theorem \ref{shmerkin}. The overall plan is to apply Theorem \ref{shmerkin} to derive such structural information about $\eta_{1}$ and its support $A_{1}$ that we can eventually contradict the upper bound for $|A_{1}|$ obtained in \eqref{form136}.

To make this precise, recall the constant $\rho = \rho(\sigma,\sigma_{0},\tau) > 0$ from \eqref{def:rho}. Let $m \in \N$ be such that $\Delta = 2^{-m}$. We recap from \eqref{choiceZeta} that
\begin{equation}\label{form147} \zeta = \min\{\tfrac{1}{2}\kappa,\tfrac{1}{10}\rho\}, \end{equation}
where $\kappa = \kappa(m,\rho) > 0$ is the constant given by Theorem \ref{shmerkin}. Now, by Theorem \ref{shmerkin}, if $\delta = (2^{-m})^{N}$ with $N$ large enough, and $\eta_{1},\eta_{2}$ are $(\delta/4)$-measures satisfying \eqref{form135}, then the following holds. There exist sets $U \subset \spt \eta_{1} = A_{1}$ and $V \subset \spt \eta_{2} = A_{2}$ such that
\begin{equation}\label{form219} \eta_{1}(U) \geq \delta^{2\rho}\ \quad \text{and} \quad \eta_{2}(V) \geq \delta^{\rho}, \end{equation} 
(the first condition is equivalent to "$\|\eta_{1}|_{U}\|_{L^{2}} \geq \delta^{\rho}\|\eta_{1}\|_{L^{2}}$" since $\eta_{1}$ is the uniformly distributed measure on $A_{1}$) and the following properties hold:
\begin{itemize}
\item[\textup{(A)}] there is a sequence $(R_{n}^{1})_{n = 0}^{N - 1} \subset \{1,\ldots,2^{m}\}^{N - 1}$, such that 
\begin{equation}\label{form149} |U \cap I|_{2^{-(n + 1)m}} = R_{n}^{1}, \qquad I \in \mathcal{D}_{2^{-mn}}(U), \end{equation}
\item[\textup{(B)}] there is a sequence $(R_{n}^{2})_{n = 0}^{N - 1} \subset \{1,\ldots,2^{m}\}^{N - 1}$, such that 
\begin{displaymath} |V \cap I|_{2^{-(n + 1)m}} = R_{n}^{2}, \qquad I \in \mathcal{D}_{2^{-mn}}(V). \end{displaymath}
\end{itemize}
Moreover, $\{n : R_{n}^{2} > 1\} \subset \{n : R_{n}^{1} \geq 2^{(1 - \rho)m}\} =: \mathcal{S}$, and
\begin{equation}\label{form216} m|\calS| \geq \log \|\eta_{2}\|_{L^{2}}^{-2} - \rho \log (1/\delta). \end{equation}

Based on the information above, we compute a preliminary lower bound for the cardinality of the set $U$, and hence $A_{1} \supset U$. The cardinality of $U$ equals the product of the branching numbers $R_{n}^{1}$, $n \in \{0,\ldots,N - 1\}$, and hence
\begin{equation}\label{form150} |A_{1}| \geq |U| \geq \prod_{n = 0}^{N - 1} R_{n}^{1} \geq \prod_{n \in \mathcal{S}} 2^{(1 - \rho)m} \times \prod_{n \notin \mathcal{S}} R_{n}^{1} = 2^{(1 - \rho)m|\mathcal{S}|} \times \prod_{n \notin \mathcal{S}} R_{n}^{1}.  \end{equation}
It is worth keeping in mind here that $A_{1} \subset [0,2\delta^{1/2}]$ (see \eqref{form125}), so the branching numbers $R_{n}^{1} \sim 1$ for $n \in \{0,\ldots,N/2 - 1\}$. Bounding the second factor in \eqref{form150} from below takes some work, and a few reminders. Before we launch into the technicalities, let us mention that "morally" $R_{n}^{1} \geq \Delta^{\sigma_{0} - s} = 2^{(s - \sigma_{0})m}$ for all $n \in \{0,\ldots,N - 1\}$. A slightly weaker version of this will literally be true for "most" $s \notin \mathcal{S}$, see \eqref{form4a}. This part of the proof corresponds to the final computation in the sketch shown in Section \ref{s:outline}. 

The main tool for lower bounding the numbers $R_{n}^{1}$ is the following proposition:

\begin{proposition}\label{prop4Restated} For every $\mathbf{C},\epsilon > 0$, there exists $\Delta_{0} = \Delta_{0}(\mathbf{C},\epsilon) \in 2^{-\N}$ such that the following holds for all $\Delta \in 2^{-\N} \cap (0,\Delta_{0}]$. Let $\delta = \Delta^{N}$, and let $K \subset B(1)$ be an upper $(s,\mathbf{C})$-regular set with $|K|_{\delta} \geq \delta^{-s + \epsilon^{2}/2}$. Let $\delta = \Delta^{N}$. For $\theta \in S^{1}$ fixed, assume that
\begin{displaymath} |\{0 \leq j \leq N - 1 : \m_{K,\theta}(x \mid [10\Delta^{j + 1},10\Delta^{j}]) > \Delta^{-\sigma_{0}}\}| \leq \epsilon N, \qquad x \in K. \end{displaymath}
Then, there exists a family $\mathcal{G} \subset \{0,\ldots,N - 1\}$ with $|\mathcal{G}| \geq (1 - 10\epsilon)N$ such that 
\begin{displaymath} |\pi_{\theta}(K \cap Q)|_{\Delta^{j + 1}} \geq \Delta^{\sigma_{0} - s + 100\epsilon} \end{displaymath}
for all $j \in \mathcal{G}$, and for at least one square $Q \in \mathcal{D}_{\Delta^{j}}(K)$.
\end{proposition}

We postpone the proof to Section \ref{s:branching}. To start heading towards applying Proposition \ref{prop4Restated}, recall from above \eqref{form136} that $A_{1}$ consists of the left end-points of $\{\pi_{1}(T) : T \in \mathcal{T}_{\delta}\}$. The tubes in $\mathcal{T}_{\delta}$ are $K_{0}$-dense in the sense defined in \eqref{form141}, and recalled here:
\begin{displaymath} |\{\mathbf{B} \in \mathcal{K}_{\delta^{1/2}}(\mathbf{T}_{0}) : T \text{ is $\mathbf{B}$-dense}\}| > \M \cdot \delta^{2\zeta}, \end{displaymath}
where $\M = |\mathcal{K}_{\delta^{1/2}}(\mathbf{T}_{0})| \geq (\delta^{1/2})^{-\sigma + 2\zeta}$ by \eqref{form118}, and $\mathbf{B}$-density meant that $|K_{\mathbf{B}} \cap T|_{\delta} \geq (\delta^{1/2})^{-\sigma + 3\zeta}$. Recall also from \eqref{form219} that $\eta_{1}(U) \geq \delta^{2\rho}$, or in other words 
\begin{equation}\label{sizeU} |U| \geq \delta^{2\rho}|A_{1}| \stackrel{\eqref{form136}}{\geq} (\delta^{1/2})^{\sigma - s + 2\zeta + 4\rho}. \end{equation}
We want to relate the "largeness" of $U$ to the "largeness" of some points in $K_{0}$ projecting to $U$; the $K_{0}$-density is crucial for that purpose. We use the slightly sloppy notation $T \in U$ if the left end-point of the interval $\pi_{1}(T)$ lies in $U$. With this notation,
\begin{displaymath} \sum_{T \in U} \sum_{\mathbf{B} \in \mathcal{K}_{\delta^{1/2}}(\mathbf{T}_{0})} |K_{\mathbf{B}} \cap T|_{\delta} \geq |U| \cdot (\M \cdot \delta^{2\zeta}) \cdot (\delta^{1/2})^{-\sigma + 3\zeta} = |U| \cdot \M \cdot (\delta^{1/2})^{-\sigma + 7\zeta}. \end{displaymath} 
Consequently, there exists a fixed disc $\mathbf{B} \in \mathcal{K}_{\delta^{1/2}}(\mathbf{T}_{0})$ with
\begin{displaymath} \sum_{T \in U} |K_{\mathbf{B}} \cap T|_{\delta} \geq |U| \cdot (\delta^{1/2})^{-\sigma + 7\zeta} \stackrel{\eqref{sizeU}}{\geq} (\delta^{1/2})^{-s +9\zeta + 4\rho}.\end{displaymath}
For $\mathbf{B} \in \mathcal{K}_{\delta}(\mathbf{T}_{0})$ as above, we write (again slightly abusing notation)
\begin{displaymath} K_{\mathbf{B}} \cap \pi_{1}^{-1}(U) := \{x \in K_{\mathbf{B}} : x \in T \text{ for some } T \in U\}. \end{displaymath} 
As we just demonstrated,
\begin{equation}\label{form143} |K_{\mathbf{B}} \cap \pi_{1}^{-1}(U)|_{\delta} \geq (\delta^{1/2})^{-s + 9\zeta + 4\rho} \stackrel{\eqref{form147}}{\geq} (\delta^{1/2})^{-s + 5\rho}. \end{equation}
We also record that 
\begin{equation}\label{form148} \pi_{1}(K_{\mathbf{B}} \cap \pi_{1}^{-1}(U)) \subset [U]_{\delta}. \end{equation}

An important feature of the sets $K_{\mathbf{B}}$, $\mathbf{B} \in \mathcal{K}_{1/2}(\mathbf{T}_{0})$, is that they are all subsets of $\bar{K}$ (see \eqref{form94a}), and therefore subsets of $L_{\theta_{0}}$ (see \eqref{form89} and \eqref{form91b}). Recall from \eqref{form95a} that
\begin{displaymath} \tfrac{1}{N} |\{0 \leq j \leq N - 1 : x \in H_{\theta_{0}}(K,\Delta^{-\sigma_{0}},[10\Delta^{j + 1},10\Delta^{j}])\}| \leq C(\sigma,\tau)\sqrt{\lambda}, \quad x \in L_{\theta}. \end{displaymath}
Unwrapping the definition further, and recalling that $C(\sigma,\tau)\sqrt{\lambda} \leq \tfrac{\rho}{2}$ by \eqref{choiceLambda}, this implies
\begin{displaymath} \tfrac{1}{N} |\{0 \leq j \leq N - 1 : \m_{K,\theta_{0}}(x \mid [10\Delta^{j + 1},10\Delta^{j}]) > \Delta^{-\sigma_{0}}\}| \leq \rho/2, \qquad x \in K_{\mathbf{B}} \cap U. \end{displaymath}
This information is preserved if we rescale $K_{\mathbf{B}} \cap U$ "to the unit scale" by $\delta^{-1/2} = \Delta^{-N/2}$. More precisely, let $T_{\mathbf{B}} \colon \R^{2} \to \R^{2}$ be the rescaling map which sends $\mathbf{B}$ to $B(1)$, and let 
\begin{displaymath} \mathbf{K} := T_{\mathbf{B}}(K_{\mathbf{B}} \cap \pi_{1}^{-1}(U)) \subset B(1) \cap T_{\mathbf{B}}(K). \end{displaymath}
Then $\mathbf{K}$ is an upper $(s,\mathbf{C})$-regular subset of $B(1)$ satisfying (using Lemma \ref{lemma7} in (b)),
\begin{itemize}
\item[(a)] $|\mathbf{K}|_{\delta^{1/2}} \geq (\delta^{1/2})^{-s + 5\rho}$ by \eqref{form143}, and
\item[(b)] $(N/2)^{-1}|\{0 \leq j \leq N/2 - 1 : \m_{\mathbf{K},\theta_{0}}(x \mid [10\Delta^{j + 1},10\Delta^{j}]) > \Delta^{-\sigma_{0}}\}| \leq \rho$ for $x \in \mathbf{K}$.
\end{itemize}
The information (a)-(b) places us in a position to apply Proposition \ref{prop4Restated} with $\epsilon := \sqrt{10\rho}$. The conclusion is that there exists a family of indices
\begin{equation}\label{form152} \mathcal{G} \subset \{0,\ldots,\tfrac{N}{2} - 1\} \quad \text{with} \quad |\mathcal{G}| \geq (1 - 10\sqrt{\rho})\tfrac{N}{2} \end{equation}
such that for each $j \in \mathcal{G}$ there exists at least one square $Q \in \mathcal{D}_{\Delta^{j}}(\mathbf{K})$ with the property
\begin{displaymath} |\pi_{1}(\mathbf{K} \cap Q)|_{\Delta^{j + 1}} \geq \Delta^{\sigma_{0} - s + 100\sqrt{10\rho}}. \end{displaymath}
Recalling that $\mathbf{K} = T_{\mathbf{B}}(K_{\mathbf{B}} \cap \pi_{1}^{-1}(U))$, and writing $I := \pi_{1}(T_{\mathbf{B}}(Q))$, the previous inequality yields (after rescaling by $\Delta^{N/2} = \delta^{1/2}$)
\begin{displaymath} |U \cap I|_{\Delta^{j + 1 + N/2}} \geq |\pi_{1}(K_{\mathbf{B}} \cap \pi_{1}^{-1}(U) \cap T_{\mathbf{B}}(Q))|_{\Delta^{j + 1 + N/2}} \geq \Delta^{\sigma_{0} - s + 100\sqrt{10\rho}}.  \end{displaymath}
Note that $I$ is an interval of length $\Delta^{j + N/2}$. Recalling from \eqref{form149} the definition of the "branching numbers" $R_{n}^{1}$ associated to $U$, we have just proven that
\begin{equation}\label{form4a} R_{n}^{1} \geq \Delta^{\sigma_{0} - s + 100\sqrt{10\rho}} = 2^{m(s - \sigma_{0} - 100\sqrt{10\rho})}, \qquad n \in \mathcal{G} + \tfrac{N}{2} =: \overline{\mathcal{G}}. \end{equation}

We may now continue the lower estimate \eqref{form150} for the cardinality of $A_{1}$:
\begin{align} |A_{1}| \geq 2^{(1 - \rho)m|\mathcal{S}|} \times \prod_{n \in \overline{\mathcal{G}} \, \setminus \, \mathcal{S}} R_{n}^{1} & \geq 2^{(1 - \rho)m|\mathcal{S}|} \cdot 2^{m(s - \sigma_{0} - 100\sqrt{10\rho})(|\overline{\mathcal{G}}| - |\mathcal{S}|)} \notag\\
&\label{form5a} \geq 2^{m|\mathcal{S}|\sigma_{0}} \cdot 2^{m(s - \sigma_{0} - 100\sqrt{10\rho})|\overline{\mathcal{G}}|}.  \end{align} 
From \eqref{form216} we further infer that
\begin{displaymath} m|\calS| \geq \log \|\eta_{2}\|_{L^{2}}^{-2} - \rho \log (1/\delta) \stackrel{\eqref{form151}}{\geq} (\tfrac{\sigma}{8} - \zeta - \rho) \log (1/\delta) \stackrel{\eqref{form147}}{\geq} (\tfrac{\sigma}{8} - 2\rho)\log (1/\delta). \end{displaymath}
Recalling finally that $|\overline{\mathcal{G}}| \geq (1 - 10\sqrt{\rho})\tfrac{N}{2}$ by \eqref{form152}, and $2^{mN} = \delta^{-1}$, we deduce from \eqref{form5a}
\begin{align*} |A_{1}| & \geq \delta^{\sigma_{0}(-\tfrac{\sigma}{8} + 2\rho)} \cdot (\delta^{1/2})^{(\sigma_{0} + C\sqrt{\rho} - s)(1 - 10\sqrt{\rho})}\\
& = (\delta^{1/2})^{-(2\sigma_{0}(\tfrac{\sigma}{8} + 2\rho) + (s - \sigma_{0} - 100\sqrt{10\rho})(1 - 10\sqrt{\rho}))}. \end{align*}
On the other hand $|A_{1}| \leq (\delta^{1/2})^{\sigma - s - \rho}$ by \eqref{form136}. These inequalities are incompatible by the choice of $\rho$ at \eqref{def:rho} for $\delta > 0$ small enough (recall also Remark \ref{rem2}, and that the "$\sigma$" above is actually the parameter $\bar{\sigma} \in [\sigma - \sqrt{\lambda},s]$ introduced at \eqref{form86}). The required smallness of "$\delta$" here superficially seems to depend only on the size of $\rho = \rho(\sigma,\sigma_{0},\tau) > 0$, but one should keep in mind that the "$\delta$" here is actually the parameter $\bar{\delta} = \Delta^{\bar{N}}$ introduced at \eqref{form86}, where $\bar{N} \gtrsim_{\sigma,\sigma_{0},\tau,\Delta} N$. Therefore, the requirement that $\bar{\delta}$ be small enough in terms of $\sigma,\sigma_{0},\tau$ entails that the "original" $\delta > 0$ in the counter assumption \eqref{form22} needs to be chosen small in a manner depending on $\Delta$. 

This contradiction completes the proof of Proposition \ref{mainProp}.

\section{Lower bounding branching numbers}\label{s:branching}

The purpose of this section is to prove Proposition \ref{prop4Restated}, restated below:

\begin{proposition}\label{prop4} For every $\mathbf{C},\epsilon > 0$, there exists $\Delta_{0} = \Delta_{0}(\mathbf{C},\epsilon) \in 2^{-\N}$ such that the following holds for all $\Delta \in 2^{-\N} \cap (0,\Delta_{0}]$. Let $\delta = \Delta^{N}$, and let $K \subset B(1)$ be an upper $(s,\mathbf{C})$-regular set with $|K|_{\delta} \geq \delta^{-s + \epsilon^{2}/2}$. Let $\delta = \Delta^{N}$. For $\theta \in S^{1}$ in fixed, assume that
\begin{equation}\label{form2a} |\{0 \leq j \leq N - 1 : \m_{K,\theta}(x \mid [10\Delta^{j + 1},10\Delta^{j}]) > \Delta^{-\sigma}\}| \leq \epsilon N, \qquad x \in K. \end{equation}
Then, there exists a family $\mathcal{G} \subset \{0,\ldots,N - 1\}$ with $|\mathcal{G}| \geq (1 - 10\epsilon)N$ such that 
\begin{displaymath} |\pi_{\theta}(K \cap Q)|_{\Delta^{j + 1}} \geq \Delta^{\sigma - s + 100\epsilon} \end{displaymath}
for all $j \in \mathcal{G}$, and for at least one square $Q \in \mathcal{D}_{\Delta^{j}}(K)$.
\end{proposition}

We start with two lemmas.

\begin{lemma}\label{lemma2a} For every $\mathbf{C},\epsilon,s > 0$ there exists $\Delta_{0} = \Delta_{0}(\mathbf{C},\epsilon) \in 2^{-\N}$ such that the following holds for all $\Delta \in 2^{-\N} \cap (0,\Delta_{0}]$. Let $\delta = \Delta^{N}$ for some $N \in \N$. $K \subset [0,1)^{2}$ be upper $(s,\mathbf{C})$-regular with $|K|_{\delta} \geq \delta^{-s + \epsilon}$. Let $\bar{\mu}$ be the uniformly distributed probability on $\mathcal{D}_{\delta}(K)$ (thus $\bar{\mu}(Q) = |K|_{\delta}^{-1}$ for all $Q \in \mathcal{D}_{\delta}(K)$). Then, there exist $\geq (1 - 2\sqrt{\epsilon})N$ indices $0 \leq j \leq N - 1$ such that
\begin{displaymath} H(\bar{\mu},\mathcal{D}_{\Delta^{j + 1}} \mid \mathcal{D}_{\Delta^{j}}) \geq (s - \sqrt{\epsilon}) \log \tfrac{1}{\Delta}. \end{displaymath}
 \end{lemma}

\begin{proof} We begin with the following simple observation:
\begin{displaymath} H(\bar{\mu},\mathcal{D}_{\Delta^{j + 1}} \mid \mathcal{D}_{\Delta^{j}}) \leq s\log \tfrac{1}{\Delta} + \log \mathbf{C}, \end{displaymath}
Indeed,
\begin{displaymath} H(\bar{\mu},\mathcal{D}_{\Delta^{j + 1}} \mid \mathcal{D}_{\Delta^{j}}) = \sum_{Q \in \mathcal{D}_{\Delta^{j}}} \bar{\mu}(Q)H(\bar{\mu}^{Q},\mathcal{D}_{\Delta}) \end{displaymath}
by definition, and each factor $H(\bar{\mu}^{Q},\mathcal{D}_{\Delta})$ is bounded from above by the logarithm of the number of (non-zero) elements in this partition, which is $\leq \mathbf{C} \Delta^{-s}$.

A scale $j$ is called \emph{bad} if $H(\bar{\mu},\mathcal{D}_{\Delta^{j + 1}} \mid \mathcal{D}_{\Delta^{j}}) \leq (s - \sqrt{\epsilon}) \log \tfrac{1}{\Delta}$, and \emph{good} otherwise. Using the inequality $|K|_{\delta} \geq \delta^{-s + \epsilon}$, we observe that
\begin{align*}  (s - \epsilon)\log \tfrac{1}{\delta} \leq \log |K|_{\delta} = H(\bar{\mu},\mathcal{D}_{\delta}) & = \sum_{j = 0}^{N - 1} H(\mu,\mathcal{D}_{\Delta^{j + 1}} \mid \mathcal{D}_{\Delta^{j}}) = \sum_{j \emph{ bad}} \ldots + \sum_{j \emph{ good}} \ldots. \end{align*}
Let $\lambda N$ be the number of bad scales, so the number of good scales is $(1 - \lambda)N$. With this notation, the right hand side is bounded from above by
\begin{align*} \lambda N\cdot (s - \sqrt{\epsilon})\log \tfrac{1}{\Delta} & + (1 - \lambda)N \cdot (s\log \tfrac{1}{\Delta} + \log \mathbf{C})\\
& \leq [\lambda \cdot (s - \sqrt{\epsilon}) + (s - s\lambda)] \log \tfrac{1}{\delta} + N \log \mathbf{C} \\
& = (s - \lambda \sqrt{\epsilon}) \log \tfrac{1}{\delta} + N \log \mathbf{C}. \end{align*}
Since this number must exceed $(s - \epsilon)\log \tfrac{1}{\delta}$, and $\log \tfrac{1}{\delta} = N \log \tfrac{1}{\Delta}$, we infer that
\begin{displaymath} \lambda \sqrt{\epsilon} \log \tfrac{1}{\delta} \leq \epsilon \log \tfrac{1}{\delta} + \log \mathbf{C}^{N} \quad \Longrightarrow \quad \lambda \leq \sqrt{\epsilon} + \tfrac{\mathbf{C}}{\sqrt{\epsilon} \log (1/\Delta)}. \end{displaymath}
Provided $\Delta > 0$ is so small that $\log (1/\Delta) \geq \mathbf{C}/\epsilon$, we infer $\lambda \leq 2\sqrt{\epsilon}$. In other words, there are $\geq (1 - 2\sqrt{\epsilon})N$ good scales, as claimed.\end{proof}

\begin{lemma}\label{lemma1} For every $\mathbf{C},\epsilon,s > 0$ and $H \geq 1$ there exists $\Delta_{0} = \Delta_{0}(\mathbf{C},\epsilon,s) > 0$ such that the following holds for all $\Delta \in 2^{-\N} \cap (0,\Delta_{0}]$. Let $\mu$ be a probability measure with $K := \spt \mu \subset [0,1)^{2}$, let $\Delta \in 2^{-\N}$, and assume that $|K|_{\Delta} \leq \mathbf{C}/\Delta^{s}$. Assume also that
\begin{equation}\label{form1a} H(\mu,\mathcal{D}_{\Delta}) \geq (s - \epsilon) \log \tfrac{1}{\Delta}. \end{equation}
Then, there exists a constant $\mathbf{c} \sim (\mathbf{C}s)^{-1}\epsilon$ such that the collection
\begin{displaymath} \mathcal{D}_{\mathrm{good}} := \{Q \in \mathcal{D}_{\Delta} : \mathbf{c}\Delta^{s} \leq \mu(Q) \leq \Delta^{s - H\epsilon}\} \end{displaymath}
satisfies $\mu(\{\cup \mathcal{D}_{\mathrm{good}}\}) \geq 1 - \tfrac{3}{H}$.
\end{lemma}

\begin{proof} Let $\mathcal{D}_{\mathrm{light}} := \{Q \in \mathcal{D}_{\Delta} : 0 < \mu(Q) \leq c \Delta^{s}\}$, where $\mathbf{c} = c (\mathbf{C}s)^{-1}\epsilon$ for a suitable small absolute constant $c > 0$. Then
\begin{align*} \sum_{Q \in \mathcal{D}_{\mathrm{light}}} \mu(Q) \log \tfrac{1}{\mu(Q)} & = \sum_{2^{j} \leq \mathbf{c} \Delta^{s}} \sum_{2^{j - 1} \leq \mu(Q) \leq 2^{j}} \mu(Q) \log\tfrac{1}{\mu(Q)}\\
& \leq |K|_{\Delta} \cdot \sum_{2^{j} \leq \mathbf{c} \Delta^{s}} 2^{j} \cdot \log 2^{1 - j} \sim \mathbf{c}\mathbf{C} \cdot \log \tfrac{2}{\mathbf{c}\Delta^{s}} \lesssim \mathbf{c}\mathbf{C}s \log \tfrac{1}{\Delta},  \end{align*} 
provided that $\log \tfrac{1}{\mathbf{c}} \leq s\log \tfrac{1}{\Delta}$, as we may assume by choosing $\Delta = \Delta(\mathbf{C},\epsilon,s) > 0$ sufficiently small. In particular, choosing the absolute constant $c > 0$ in $\mathbf{c} = c(\mathbf{C}s)^{-1}\epsilon$ sufficiently small, we find
\begin{displaymath} \sum_{Q \in \mathcal{D}_{\mathrm{light}}} \mu(Q) \log \tfrac{1}{\mu(Q)} \leq \epsilon \log \tfrac{1}{\Delta}. \end{displaymath} 
With these choices, and by \eqref{form1a},
\begin{displaymath} \sum_{Q \in \mathcal{D}_{\Delta} \, \setminus \, \mathcal{D}_{\mathrm{light}}} \mu(Q) \log \tfrac{1}{\mu(Q)} \geq (s - 2\epsilon) \log \tfrac{1}{\Delta}. \end{displaymath}
Next, fix $H \geq 1$, and let $\mathcal{D}_{\mathrm{heavy}} := \{Q \in \mathcal{D}_{\Delta} : \mu(Q) \geq \Delta^{s - H\epsilon}\}$. Recall also that $\mathcal{D}_{\mathrm{good}} = \mathcal{D}_{\Delta} \, \setminus  \, [\mathcal{D}_{\mathrm{light}} \cup \mathcal{D}_{\mathrm{heavy}}]$. Write
\begin{displaymath} 1 - \lambda := \mu(\cup \mathcal{D}_{\mathrm{good}}) \in [0,1], \end{displaymath}
so that $\mu(\cup \mathcal{D}_{\mathrm{heavy}}) \leq \lambda$. With this notation, we derive the following inequalities regarding $\mathcal{D}_{\mathrm{heavy}}$ and $\mathcal{D}_{\mathrm{good}}$:
\begin{displaymath} \sum_{Q \in \mathcal{D}_{\mathrm{heavy}}} \mu(Q) \log \tfrac{1}{\mu(Q)} \leq (s - H\epsilon)\lambda \log \tfrac{1}{\Delta} \end{displaymath} 
and
\begin{displaymath} \sum_{Q \in \mathcal{D}_{\mathrm{good}}} \mu(Q) \log \tfrac{1}{\mu(Q)} \leq \log \tfrac{1}{\mathbf{c}\Delta^{s}} \cdot (1 - \lambda) = (1 - \lambda)[s \log \tfrac{1}{\Delta} + \log \tfrac{1}{\mathbf{c}}]. \end{displaymath}
However, the sum of the two right hand sides must satisfy
\begin{displaymath} (s - H\epsilon)\lambda \log \tfrac{1}{\Delta} +  (1 - \lambda) [s \log \tfrac{1}{\Delta} + \log \tfrac{1}{\mathbf{c}}] \geq (s - 2\epsilon) \log \tfrac{1}{\Delta}. \end{displaymath}
Dividing by $\log \tfrac{1}{\Delta}$, and writing $\kappa := (\log \tfrac{1}{\mathbf{c}})/(\log \tfrac{1}{\Delta})$, we find
\begin{displaymath} \lambda(s - H\epsilon) + (1 - \lambda)(s + \kappa) \geq s - 2\epsilon. \end{displaymath}
Assuming finally that $\Delta > 0$ is so small that $\kappa \leq \epsilon$, this leads to
\begin{displaymath} 1- \mu(\cup \mathcal{D}_{\mathrm{good}}) = \lambda \leq \frac{2\epsilon + \kappa}{H\epsilon + \kappa} \leq \frac{3\epsilon}{H\epsilon} \leq \frac{3}{H}, \end{displaymath}
which completes the proof. \end{proof}

We are the prepared to prove Proposition \ref{prop4}.

\begin{proof}[Proof of Proposition \ref{prop4}] Apply Lemma \ref{lemma2a} with parameters $\mathbf{C},\epsilon^{2}/2$ to find a set of scales $\mathcal{G}_{1} \subset \{0,\ldots,N - 1\}$ of cardinality $|\mathcal{G}_{1}| \geq (1 - \epsilon)N$ with the property 
\begin{displaymath} H(\bar{\mu},\mathcal{D}_{\Delta^{j + 1}} \mid \mathcal{D}_{\Delta^{j}}) \geq (s - \epsilon) \log \tfrac{1}{\Delta}, \qquad j \in \mathcal{G}_{1}. \end{displaymath}
Here $\bar{\mu}$ is the uniform probability measure on the union of the squares $\mathcal{D}_{\delta}(K)$. In particular $\spt \bar{\mu} \subset [K]_{\delta}$ (the $\delta$-neighbourhood of $K$). It follows from a few applications of the triangle inequality, that if $x \in [K]_{\delta}$ and $x' \in K$ is some point with $|x - x'| \leq \delta$, then
\begin{displaymath} \mathfrak{m}_{K,\theta}(x \mid [5\Delta^{j + 1},5\Delta^{j}]) \leq 10 \cdot \mathfrak{m}_{K,\theta}(x' \mid [10\Delta^{j + 1},10\Delta^{j}]), \qquad 0 \leq j \leq N - 1. \end{displaymath}
In particular, by \eqref{form2a},
\begin{equation}\label{form3a} |\{0 \leq j \leq N - 1 : \mathfrak{m}_{K,\theta}(x \mid [5\Delta^{j + 1},5\Delta^{j}]) > 10\Delta^{-\sigma}\}| \leq \epsilon N, \qquad x \in \spt \bar{\mu}. \end{equation}
 
We then also consider another set of good scales $\mathcal{G}_{2}$:
\begin{displaymath} \mathcal{G}_{2} := \{0 \leq j \leq N - 1 : \bar{\mu}(\mathbf{Bad}(j)) \leq \tfrac{1}{9}\}, \end{displaymath} 
where 
\begin{displaymath} \mathbf{Bad}(j) := \{x \in \R^{2} : \m_{K,\theta}(x \mid [5\Delta^{j + 1},5\Delta^{j}]) > 10\Delta^{-\sigma}\}, \qquad 0 \leq j \leq N - 1. \end{displaymath}
We claim that $|\mathcal{G}_{2}| \geq (1 - 9\epsilon)N$. Writing $|\mathcal{G}_{2}| =: (1 - \lambda)N$, then $|\{0,\ldots,N - 1\} \, \setminus \, \mathcal{G}_{2}| = \lambda N$, and
\begin{displaymath} \tfrac{\lambda N}{9} \leq \sum_{j \notin \mathcal{G}_{2}} \bar{\mu}(\mathbf{Bad}(j)) \leq \int \sum_{j = 0}^{N - 1} \mathbf{1}_{\mathbf{Bad}(j)}(x) \, d\bar{\mu}(x) \leq \epsilon N, \end{displaymath} 
using \eqref{form3a}. This yields $\lambda \leq 9\epsilon$.

Now we define $\mathcal{G} := \mathcal{G}_{1} \cap \mathcal{G}_{2}$, and we note that $|\mathcal{G}| \geq (1 - \epsilon - 9\epsilon)N \geq (1 - 10\epsilon)N$. These are the good scales whose existence is stated in the proposition. So, it remains to find the good cubes $Q \in \mathcal{D}_{\Delta^{j}}(K)$.

Fix $j \in \mathcal{G}$, and write $G_{2} := [K]_{\delta} \, \setminus \, \mathbf{Bad}(j)$, so that $\bar{\mu}(G_{2}) \geq \tfrac{8}{9}$, because $j \in \mathcal{G}_{2}$. Next, because $j \in \mathcal{G}_{1}$, we have
\begin{displaymath} (s - \epsilon) \log \tfrac{1}{\Delta} \leq \sum_{Q \in \mathcal{D}_{\Delta^{j}}} \bar{\mu}(Q)H(\bar{\mu}^{Q},\mathcal{D}_{\Delta}). \end{displaymath}
Here uniformly $H(\bar{\mu}^{Q},\mathcal{D}_{\Delta}) \leq s\log \tfrac{1}{\Delta} + \mathbf{C}$ by the $(s,\mathbf{C})$-regularity of $\mu$. Provided that $\Delta > 0$ is sufficiently small in terms of $\mathbf{C},\epsilon$, we claim that there exists a sub-collection $\mathcal{D} \subset \mathcal{D}_{\Delta^{j}}$ of total measure $\bar{\mu}(\cup \mathcal{D}) \geq \tfrac{7}{8}$ such that $H(\bar{\mu}^{Q},\mathcal{D}_{\Delta}) \geq (s - 16\epsilon)\log \tfrac{1}{\Delta}$ for all $Q \in \mathcal{D}$. Indeed, if $\bar{\mu}(\cup \mathcal{D}) =: 1 - \lambda$, then
\begin{align*} (s - \epsilon)\log \tfrac{1}{\Delta} & \leq (s - 16\epsilon) \log \tfrac{1}{\Delta} \sum_{Q \in \mathcal{D}_{\Delta^{j}} \, \setminus \, \mathcal{D}} \bar{\mu}(Q) + (s\log \tfrac{1}{\Delta} + \mathbf{C})\sum_{Q \in \mathcal{D}} \bar{\mu}(Q)\\
& = [\lambda(s - 16\epsilon) + s(1 - \lambda)]\log \tfrac{1}{\Delta} + \mathbf{C}.  \end{align*} 
This yields $\lambda \leq \tfrac{1}{16} + \tfrac{\mathbf{C}}{16\epsilon}(\log \tfrac{1}{\Delta})^{-1} \leq \tfrac{1}{8}$ after rearranging, provided $\Delta > 0$ is small enough in terms of $\mathbf{C},\epsilon$, as claimed.

Using $\bar{\mu}(G_{2}) \geq \tfrac{8}{9}$, and $\bar{\mu}(\cup \mathcal{D}) \geq \tfrac{7}{8}$, we claim the existence of a square $Q \in \mathcal{D}$ with the property 
\begin{equation}\label{form139} \bar{\mu}(Q \cap G_{2}) \geq \tfrac{3}{4}\bar{\mu}(Q). \end{equation} 
If this were not the case, then $G_{2}$ would be covered by the (a) the squares in $\mathcal{D}_{\Delta^{j}} \, \setminus \, \mathcal{D}$ and (b) the squares $Q \in \mathcal{D}$ with $\bar{\mu}(Q \cap G_{2}) < \tfrac{3}{4}\bar{\mu}(Q)$, which would yield
\begin{displaymath} \tfrac{8}{9} \leq \bar{\mu}(G_{2}) \leq \bar{\mu}(\cup [\mathcal{D}_{\Delta^{j}} \, \setminus \, \mathcal{D}]) + \tfrac{3}{4} < \tfrac{1}{8} + \tfrac{3}{4} = \tfrac{7}{8}. \end{displaymath}
This is a contradiction. Therefore the existence of the square $Q \in \mathcal{D}$ satisfying \eqref{form139} is guaranteed. We now fix this square $Q \in \mathcal{D}$ for the remainder of the argument.

Recall that $H(\bar{\mu}^{Q},\mathcal{D}_{\Delta}) \geq (s - 16\epsilon)\log \tfrac{1}{\Delta}$ by the definition of $\mathcal{D}$, and also $|\spt \bar{\mu}^{Q}|_{\Delta} \leq \mathbf{C}/\Delta^{s}$ by the upper $(s,\mathbf{C})$-regularity of $K$. It then follows from Lemma \ref{lemma1} with "$16\epsilon$" in place of "$\epsilon$", and $H := 6$, that the family
\begin{displaymath} \mathcal{D}_{\mathrm{good}} := \{Q' \in \mathcal{D}_{\Delta} : \bar{\mu}^{Q}(Q') \leq \Delta^{s - 99\epsilon}\} \end{displaymath}
satisfies $\bar{\mu}^{Q}(\cup \mathcal{D}_{\mathrm{good}}) \geq \tfrac{1}{2}$. Let $G_{2}^{Q} := T_{Q}(G_{2})$. Then \eqref{form139} implies $\bar{\mu}^{Q}(G_{2}^{Q}) \geq \tfrac{3}{4}$. Combining these two estimates,
\begin{displaymath} \bar{\mu}^{Q}(\{Q' \in \mathcal{D}_{\mathrm{good}} : Q' \cap G_{2}^{Q} \neq \emptyset\}) \geq \tfrac{1}{4}. \end{displaymath}
Since $\bar{\mu}^{Q}(Q') \leq \Delta^{s - 100\epsilon}$ for all $Q' \in \mathcal{D}_{\mathrm{good}}$, we infer
\begin{equation}\label{form140} |G_{2}^{Q}|_{\Delta} \geq \tfrac{1}{4}\Delta^{99\epsilon - s}. \end{equation}

By definition $G_{2} = [K]_{\delta} \, \setminus \, \mathbf{Bad}(j)$, which means that 
\begin{displaymath} \m_{K,\theta}(y \mid [5\Delta^{j + 1},5\Delta^{j}]) \leq 10\Delta^{-\sigma}, \qquad y \in G_{2}. \end{displaymath}
After rescaling by $\Delta^{-j}$, this has the following consequence for $G_{2}^{Q}$:
\begin{displaymath} |K_{5\Delta}^{Q} \cap \pi_{e}^{-1}\{\pi(y)\}|_{5\Delta} \leq \m_{K^{Q},e}(y \mid [5\Delta,5]) \leq 10\Delta^{-\sigma}, \qquad y \in G_{2}^{Q}. \end{displaymath}
Since $G_{2} \subset [K]_{\delta}$, the set $K^{Q}$ on the left hand side can be replaced by $G_{2}^{Q}$ without altering the conclusion. This inequality combined with \eqref{form140} implies that
\begin{displaymath} |\pi_{\theta}(G_{2}^{Q})|_{\Delta} \gtrsim \Delta^{\sigma - s + 99\epsilon}, \end{displaymath}
and consequently $|\pi_{\theta}(K \cap Q)|_{\Delta^{j + 1}} \geq \Delta^{\sigma - s + 100\epsilon}$ for $\Delta = \Delta(\epsilon) > 0$ small enough. \end{proof}

\section{Self-improvements}\label{s:selfImprovement}

The purpose of this section is to deduce Theorem \ref{mainThm2} from Theorem \ref{mainThm}. This is accomplished via Lemma \ref{lemma9}, or more precisely its simpler-to-state Corollary \ref{cor1}. The details can be found in Section \ref{s3}. 
\begin{lemma}\label{lemma9} For all $\mathbf{C} \geq 1$ and $\mathfrak{d} > 0$, there exist constants $\mathbf{c} = \mathbf{c}(\mathbf{C}) > 0$ and $\Delta_{0}(\mathbf{C},\mathfrak{d}) \in (0,\tfrac{1}{2}]$ such that the following holds.

Let $s \in (0,2]$, let $0 \leq \sigma < \Sigma \leq 1$ with $\Sigma - \sigma \geq \mathfrak{d}$, and let $\epsilon > 0$. Let $\nu$ be a finite Borel measure on $S^{1}$. Assume that there exists a scale $\Delta \in (0,\Delta_{0}]$ such that
\begin{equation}\label{form14a} \sup_{\mu} \nu\left(\left\{\theta : \mu(B(1) \cap H_{\theta}(K,\Delta^{-\sigma},[\Delta,1])) \geq \mathbf{c} \cdot \mathfrak{d}2^{-32/\mathfrak{d}} \right\} \right) < \tfrac{\mathfrak{d}^{2}}{400}\epsilon, \end{equation}
where the $\sup$ runs over all Ahlfors $(s,\mathbf{C})$-regular measures $\mu$ with $K = \spt \mu$ and $\diam K \geq 1$.

Let $\eta > 0$ and $N \geq (200^{2}/\mathfrak{d}^{4}) \log \eta^{-1}$, and $\delta := \Delta^{N}$. Then,
\begin{equation}\label{form15a} \sup_{\mu} \nu(\{\theta : \mu(B(1) \cap H_{\theta}(K,\delta^{-\Sigma},[\delta,1])) \geq \eta\}) \leq \epsilon, \end{equation}
where the $\sup$ runs over Ahlfors $(s,\mathbf{C})$-regular measures $\mu$ with $K = \spt \mu$ and $\diam K \geq 1$. \end{lemma} 

Only the following corollary of Lemma \ref{lemma9} is needed to prove Theorem \ref{mainThm2}:

\begin{cor}\label{cor1} For every $\mathbf{C},\epsilon > 0$, $s \in (0,2]$, and $0 < \sigma < \Sigma \leq 1$, there exists $\epsilon_{0} = \epsilon_{0}(\mathbf{C},\Sigma - \sigma,\epsilon) > 0$ and $\Delta_{0} = \Delta_{0}(\mathbf{C},\Sigma - \sigma) > 0$ such that the following holds. Assume that $\nu$ is a fixed Borel probability measure on $S^{1}$, and there exists some $\Delta \in 2^{-\N} \cap (0,\Delta_{0}]$ such that
\begin{equation}\label{form14b} \sup_{\mu} \int \mu(B(1) \cap H_{\theta}(K,\Delta^{-\sigma},[\Delta,1])) \, d\nu(\theta) \leq \epsilon_{0}, \end{equation} 
where the $\sup$ runs over all Ahlfors $(s,\mathbf{C})$-regular measures with $K = \spt \mu$. Then, there exists $\kappa = \kappa(\Delta,\Sigma - \sigma,s) > 0$ and $\delta_{0} = \delta_{0}(\Delta,\Sigma - \sigma,s) > 0$ such that
\begin{equation}\label{form15b} \sup_{\mu} \nu(\{\theta : \mu(B(1) \cap H_{\theta}(K,\delta^{-\Sigma},[\delta,1])) \geq \delta^{\kappa}\}) \leq \epsilon, \qquad \delta \in 2^{-\N} \cap (0,\delta_{0}], \end{equation}
where the $\sup$ again runs over all Ahlfors $(s,\mathbf{C})$-regular measures with $K = \spt \mu$. \end{cor} 

\begin{proof} We start with a reduction: Corollary \ref{cor1} follows from a formally weaker version, where the $\sup$ in \eqref{form15b} only runs over measures $\mu$ with $\diam \spt \mu \geq 1$. Let us assume for now that this version has already been established. We claim that the general version follows with the following choice of constants: 
\begin{itemize}
\item[(a)] The constant "$\kappa$" (provided by the special case) has to be reduced to 
\begin{displaymath} \bar{\kappa} := \min\{\kappa s/2,(1 - \kappa)\kappa\}. \end{displaymath}
\item[(b)] The scale $\delta > 0$ needs to be so small that
\begin{displaymath} 0 < \delta \leq \delta_{0}^{1/(1 - \kappa)} \quad \text{and} \quad \delta^{\kappa s/2} < \mathbf{C}^{-1}, \end{displaymath}
where $\delta_{0} > 0$ is provided by the special case.
\end{itemize}
Let us prove this. Let $\mu$ be any Ahlfors $(s,\mathbf{C})$-regular measure with $\spt \mu = K$. If $\diam K \geq 1$, then \eqref{form15b} follows from the special case. Assume next that $\diam K \leq 1$, and fix a ball $B \subset \R^{2}$ with radius $r = \diam K \leq 1$ containing $K$. Fix also $\delta > 0$ as in (b).

Consider first the case where $r \leq \delta^{\kappa}$. Then, by our choices on $\delta,\bar{\kappa}$,
\begin{displaymath} \mu(B(1) \cap H_{\theta}(K,\delta^{-\Sigma},[\delta,1])) \leq \mu(B) \leq \mathbf{C}\delta^{\kappa s} < \delta^{\bar{\kappa}}, \qquad \theta \in S^{1}, \end{displaymath}
so certainly $\nu(\{\theta : \mu(B(1) \cap H_{\theta}(K,\delta^{-\Sigma},[\delta,1])) \geq \delta^{\bar{\kappa}}\}) = 0$. Assume next that $r \geq \delta^{\kappa}$. Now $K^{B} = \spt \mu^{B}$ has $\diam K^{B} = 1$, and $\delta/r \leq \delta_{0}$, so the special case implies $\nu(E) \leq \epsilon$, where 
\begin{displaymath} E := \{\theta : \mu^{B}(B(1) \cap H_{\theta}(K^{B},(\delta/r)^{-\Sigma},[\delta/r,1])) \geq (\delta/r)^{\kappa}\}. \end{displaymath}
Now, using Lemma \ref{lemma7}, and noting that $\delta^{-\Sigma} \geq (\delta/r)^{-\Sigma}$ and $\delta/r \leq \delta_{0}$,
\begin{displaymath} \mu(B \cap H_{\theta}(K,\delta^{-\Sigma},[\delta,r])) \leq \mu^{B}(B(1) \cap H_{\theta}(K^{B},\delta^{-\Sigma},[(\delta/r),1])) < (\delta/r)^{\kappa} \leq \delta^{\bar{\kappa}} \end{displaymath}
for all $\theta \in S^{1} \, \setminus \, E$. Thus, $\nu(\{\theta :  \mu(B \cap H_{\theta}(K,\delta^{-\Sigma},[\delta,r])) \geq \delta^{\bar{\epsilon}}\}) \leq \epsilon$. This completes the proof of the general case, because
\begin{displaymath}  \mu(B \cap H_{\theta}(K,\delta^{-\Sigma},[\delta,r])) = \mu(B(1) \cap H_{\theta}(K,\delta^{-\Sigma},[\delta,1])), \end{displaymath}
using $\spt \mu \subset K$ and $\diam K \leq r$.

Let us then establish the special case, where the $\sup$ in \eqref{form15b} only runs over measures $\mu$ with $\diam \spt \mu \geq 1$. In this case it will be legitimate to apply Lemma \ref{lemma9}.

Fix $\epsilon > 0$ and $0 < \sigma < \Sigma \leq 1$. Write $\bar{\Sigma} := \tfrac{1}{2}(\sigma + \Sigma) \in (\sigma,\Sigma)$. Write $\mathfrak{d} := \bar{\Sigma} - \sigma$. Let $\mathbf{c} = \mathbf{c}(\mathbf{C}) > 0$ and $\Delta_{0} = \Delta_{0}(\mathbf{C},\mathfrak{d}) > 0$ be the constants from Lemma \ref{lemma9}. By the hypothesis \eqref{form14b} applied with a suitable $\epsilon_{0} \sim_{\mathbf{C},\mathfrak{d},\epsilon} 1$, there exists a scale $\Delta \in (0,\Delta_{0}]$ such that
\begin{displaymath} \sup_{\mu} \nu\left(\left\{\theta : \mu(B(1) \cap H_{\theta}(K,\Delta^{-\sigma},[\Delta,1])) \geq \mathbf{c} \cdot \mathfrak{d}2^{-32/\mathfrak{d}} \right\} \right) < \tfrac{\mathfrak{d}^{2}}{400}\epsilon. \end{displaymath}
Define
\begin{displaymath} \kappa := \tfrac{\mathfrak{d}^{4}}{2 \cdot 200^{2}\log(1/\Delta)} > 0. \end{displaymath}
We claim that \eqref{form15b} holds for all $\delta > 0$ satisfying
\begin{equation}\label{form64a} \delta \leq \Delta \quad \text{and} \quad \Delta^{-1}\delta^{(\Sigma - \sigma)/2} \leq 1. \end{equation}
Fix $\delta$ as above, and let $N \geq 1$ satisfy $\Delta^{N + 1} \leq \delta \leq \Delta^{N} =: \bar{\delta}$. Write $\eta := \delta^{\kappa}$, and note that 
\begin{displaymath} (200^{2}/\mathfrak{d}^{4})\log \eta^{-1} \leq (N + 1) \kappa \cdot (200^{2}/\mathfrak{d}^{4}) \log \Delta^{-1} \leq N \end{displaymath}
by the choice of $\kappa$. Therefore, Lemma \ref{lemma9} with this "$\eta$" implies
\begin{equation}\label{form79a} \sup_{\mu} \nu(\{\theta : \mu(B(1) \cap H_{\theta}(K,\bar{\delta}^{-\bar{\Sigma}},[\bar{\delta},1])) \geq \delta^{\kappa}\}) \leq \epsilon, \end{equation}
where the $\sup$ runs over $(s,\mathbf{C})$-regular measures with $\diam K = \diam \spt \mu \geq 1$. Using first \eqref{form64a}, and then Lemma \ref{lemma13}(iii) (observe that $1 \leq \bar{\delta}/\delta \leq \Delta^{-1}$) leads to
\begin{displaymath} H_{\theta}(K,\delta^{-\Sigma},[\delta,1]) \subset H_{\theta}(K,\Delta^{-1}\bar{\delta}^{-\bar{\Sigma}},[\delta,1]) \subset H_{\theta}(K,\bar{\delta}^{-\bar{\Sigma}},[\bar{\delta},1]). \end{displaymath}
Now \eqref{form79a} gives
\begin{displaymath} \sup_{\mu} \nu(\{\theta : \mu(B(1) \cap H_{\theta}(K,\delta^{-\Sigma},[\delta,1])) \geq \delta^{\kappa}\}) \leq \epsilon, \end{displaymath}
as desired. \end{proof}

\subsection{Proof of Theorem \ref{mainThm2}}\label{s3} Fix $C > 0$ and $s,\tau \in (0,1]$ and $\underline{s} \in (0,s)$, as in the statement of Theorem \ref{mainThm2}. Fix an Ahlfors $(s,C)$-regular measure $\mu$ on $\R^{2}$, and a Borel probability measure $\nu$ on $S^{1}$ satisfying $\nu(B(x,r)) \leq Cr^{\tau}$ for all $x \in S^{1}$ and $r > 0$. We claim the existence of a constant $\kappa = \kappa(C,s,\underline{s},\tau) > 0$, and a vector $\theta \in \spt \nu$ such that
\begin{equation}\label{form38a} |\pi_{\theta}(F)|_{\delta} \geq \delta^{-\underline{s}} \end{equation}
for all Borel sets $F \subset K \cap B(1)$ with $\mu(F) \geq 2\delta^{\kappa}$, provided $\delta > 0$ is sufficiently small. 

To this end, we first apply Corollary \ref{cor1} with constants $\mathbf{C} = C$, and
\begin{displaymath} \sigma := \tfrac{s - \underline{s}}{4} \quad \text{and} \quad \Sigma := \sigma + \tfrac{s - \underline{s}}{4} = \tfrac{s - \underline{s}}{2} \quad \text{and} \quad \epsilon := \tfrac{1}{2}. \end{displaymath}
This outputs constants $\Delta_{0} = \Delta_{0}(C,s,\underline{s},\tfrac{1}{2}) > 0$ and $\epsilon_{0} = \epsilon_{0}(C,s,\underline{s},\tfrac{1}{2}) > 0$. Now, Corollary \ref{cor1} promises that if the hypothesis \eqref{form14b} holds for some $\Delta \leq \Delta_{0}$ (and we will check this in a moment), then we are guaranteed a $\kappa = \kappa(\Delta,s,\underline{s}) > 0$ (which may be assumed to be $< (s - \underline{s})/2$) such that
\begin{displaymath} \nu(\{\theta \in S^{1} : \mu(B(1) \cap H_{\theta}(K,\delta^{-\Sigma},[\delta,1])) \geq \delta^{\kappa}\}) \leq \tfrac{1}{2}. \end{displaymath} 
In particular, recalling that $\Sigma = (s - \underline{s})/2$, we may fix $\theta \in \spt \nu$ such that 
\begin{displaymath} \mu(B(1) \cap H_{\theta}(K,\delta^{(\underline{s} - s)/2},[\delta,1])) < \delta^{\kappa}. \end{displaymath}
Now, if $F \subset K \cap B(1)$ is any Borel set with $\mu(F) \geq 2\delta^{\kappa}$, we have $\mu(G) \geq \delta^{\kappa}$, where
\begin{displaymath} G := F \, \setminus \, H_{\theta}(K,\delta^{(\underline{s} - s)/2},[\delta,1]). \end{displaymath}
Since $\mu$ is Ahlfors $(s,C)$-regular, $|G|_{\delta} \gtrsim C^{-1}\delta^{\kappa - s}$. Since $G \subset \R^{2} \, \setminus \, H_{\theta}(K,\delta^{(\underline{s} - s)/2},[\delta,1])$, it is easy to check that
\begin{displaymath} |\pi_{\theta}(G)|_{\delta} \gtrsim \delta^{(s - \underline{s})/2}|G|_{\delta} \gtrsim C^{-1}\delta^{-\underline{s} + (\underline{s} - s)/2 + \kappa}. \end{displaymath}
Since $\kappa < (s - \underline{s})/2$, In particular $|\pi_{\theta}(F)|_{\delta} \geq |\pi_{\theta}(G)|_{\delta} > \delta^{-\underline{s}}$ for $\delta > 0$ small enough. This completes the proof of \eqref{form38a}, up to verifying the hypothesis \eqref{form14b} of Corollary \ref{cor1}.

Recall that this hypothesis requires there to exist $\Delta \in (0,\Delta_{0}]$ such that 
\begin{displaymath} \sup_{\bar{\mu}} \int \bar{\mu}(B(1) \cap H_{\theta}(\bar{K},\Delta^{-\sigma},[\Delta,1])) \, d\nu(\theta) \leq \epsilon_{0}, \end{displaymath}
where the $\sup$ runs over (Ahlfors) $(s,C)$-regular measures $\bar{\mu}$. But this follows from the conclusion of Theorem \ref{mainThm} applied with constants $C$ and $\epsilon_{0}$: the value of $\Delta$ only depends on $C,\epsilon_{0} = \epsilon_{0}(C,s,\underline{s},\tfrac{1}{2}),s,\underline{s}$, and $\tau$. Therefore, finally, the value of $\kappa = \kappa(\Delta,s,\underline{s}) > 0$ also only depends on $C,s,\underline{s},\tau$, as claimed. This completes the proof of Theorem \ref{mainThm2}.

\subsection{Preliminaries for the proof of Lemma \ref{lemma9}} For ease of reference, we record here the special case of Lemma \ref{lemma5} with $\{0,\ldots,a_{n}\} = \{0,\ldots,N\}$:

\begin{lemma}\label{lemma5a} Let $0 < \Delta \leq \tfrac{1}{50}$, $N \geq 1$, $\eta \in (0,1]$, and $\theta \in S^{1}$. Let $F \subset B(1)$ be a set with the property
\begin{displaymath} \tfrac{1}{N} |\{0 \leq j \leq N - 1 : x \in H_{\theta}(F,\Delta^{-\sigma},[50\Delta^{j + 1},50\Delta^{j}])\}| \leq \eta, \qquad x \in F. \end{displaymath}
Write $\delta := \Delta^{N}$. Then, for an absolute constant $C \geq 1$,
\begin{equation}\label{form47a} |F_{5\delta} \cap \pi_{\theta}^{-1}\{t\}|_{5\delta} \leq C^{N}\delta^{-\sigma - \eta}, \qquad t \in \R. \end{equation} 
\end{lemma}

We also record the following elementary lemma:

\begin{lemma}\label{lemma12} Let $\mathcal{D} = \bigcup_{j = 0}^{N - 1} \mathcal{D}_{j}$ a sequence of measurable partitions of a probability space $(\Omega,\mu)$. Assume that for every $Q \in \mathcal{D}$ there is an associated $\mu$ measurable set $H(Q) \subset Q$. Assume that
\begin{displaymath} |\{0 \leq j \leq N - 1 : x \in H(Q_{j}(x))\}| \geq \mathfrak{d} N \end{displaymath}
for $\mu$ almost every $x \in \Omega$, where $Q_{j}(x)$ is the unique element of $\mathcal{D}_{j}$ containing $x$, and $\mathfrak{d} > 0$. Then, there exists a subset $F \subset \Omega$ with $\mu(F) \geq \mathfrak{d}^{2}/16$ such that
\begin{equation}\label{form50a} |\{0 \leq j \leq N - 1 : \mu(H(Q_{j}(x))) \geq \tfrac{\mathfrak{d}}{4}\mu(Q)\}| \geq \tfrac{\mathfrak{d}^{2}}{16}N, \qquad x \in F. \end{equation}
\end{lemma}

\begin{proof} By hypothesis,
\begin{displaymath} \sum_{j = 0}^{N - 1} \sum_{Q \in \mathcal{D}_{j}} \mu(H(Q)) = \int_{\Omega} |\{0 \leq j \leq N - 1 : x \in H(Q_{j}(x))\}| \, d\mu(x) \geq \mathfrak{d} N. \end{displaymath}
Since $\mu$ is a probability measure, $\sum_{Q \in \mathcal{D}_{j}} \mu(H(Q)) \leq 1$ uniformly, so there exists a subset $\mathcal{G} \subset \{0,\ldots,N - 1\}$ with $|\mathcal{G}| \geq \mathfrak{d} N/2$ such that
\begin{displaymath} \sum_{Q \in \mathcal{D}_{j}} \mu(H(Q)) \geq \tfrac{\mathfrak{d}}{2} \qquad j \in \mathcal{G}. \end{displaymath}
For $j \in \mathcal{G}$ fixed, the collection
\begin{displaymath} \mathcal{D}_{j,\mathrm{heavy}} := \{Q \in \mathcal{D}_{j} : \mu(H(Q)) \geq \tfrac{\mathfrak{d}}{4} \mu(Q)\} \end{displaymath}
has to satisfy $\mu(\cup \mathcal{D}_{j,\mathrm{heavy}}) \geq \mathfrak{d}/4$, since otherwise
\begin{displaymath} \sum_{Q \in \mathcal{D}_{j}} \mu(H(Q)) \leq \sum_{Q \in \mathcal{D}_{j,\mathrm{heavy}}} \mu(Q) + \tfrac{\mathfrak{d}}{4} \sum_{Q \in \mathcal{D}_{j} \, \setminus \, \mathcal{D}_{j,\mathrm{heavy}}} \mu(Q) < \tfrac{\mathfrak{d}}{2}, \end{displaymath}
contradicting the definition of $j \in \mathcal{G}$. Therefore, 
\begin{displaymath} \int_{\Omega} |\{0 \leq j \leq N - 1 : Q_{j}(x) \in \mathcal{D}_{j,\mathrm{heavy}}\}| \, d\mu(x) \geq \sum_{j \in \mathcal{G}} \sum_{Q \in \mathcal{D}_{j,\mathrm{heavy}}} \mu(Q) \geq \tfrac{\mathfrak{d}^{2}}{8}N. \end{displaymath} 
It follows that there exists a subset $F \subset \Omega$ with $\mu(F) \geq \mathfrak{d}^{2}/16$ such that
\begin{displaymath} |\{0 \leq j \leq N - 1 : Q_{j}(x) \in \mathcal{D}_{j,\mathrm{heavy}}\}| \geq \tfrac{\mathfrak{d}^{2}}{16}N, \qquad x \in F. \end{displaymath}
This is equivalent to \eqref{form50a}, so the proof is complete.   \end{proof} 

We will also need the notion of \emph{Kullback-Leibler divergence}, recalled below:
\begin{definition}[Kullback-Leibler divergence]\label{def:divergence} Let $\mu,\nu$ be probability measures on a space $\Omega$, and let $\mathcal{E}$ be a ($\mu$ and $\nu$ measurable) partition of $\Omega$. The \emph{Kullback-Leibler divergence of $\nu$ relative to $\mu$ and the partition $\mathcal{E}$} is defined as
\begin{displaymath} \Div_{\mu}(\nu,\mathcal{E}) := \sum_{E \in \mathcal{E}} \nu(E) \log \frac{\nu(E)}{\mu(E)}. \end{displaymath} 
If $\mathcal{F}$ is another ($\mu$ and $\nu$ measurable) partition of $\Omega$, we also define the following relative version of this quantity:
\begin{displaymath} \Div_{\mu}(\nu,\mathcal{E} \mid \mathcal{F}) := \sum_{F \in \mathcal{F}} \nu(F) \Div_{\mu_{F}}(\nu_{F},\mathcal{E}), \end{displaymath}
where $\mu_{F} = \mu(F)^{-1}\mu|_{F}$ and $\nu_{F} = \nu(F)^{-1}\nu|_{F}$. \end{definition}

Relative Kullback-Leibler divergences enjoy a useful telescoping property:

\begin{lemma}[Telescoping property of relative Kullback-Leibler divergences]\label{lemma:telescope} Assume that $\mathcal{E}$ and $\mathcal{F}$ are partitions, as in Definition \ref{def:divergence}, such that $\mathcal{E}$ refines $\mathcal{F}$. Then,
\begin{equation}\label{form19a} \Div_{\mu}(\nu,\mathcal{E} \mid \mathcal{F}) = \Div_{\mu}(\nu,\mathcal{E}) - \Div_{\mu}(\nu,\mathcal{F}). \end{equation}
In particular: if $\delta = \Delta^{N}$ for some $\Delta \in 2^{-\N}$ and $N \in \N$, and $\mathcal{D}_{j} = \mathcal{D}_{\Delta^{j}}$ is the partition to dyadic cubes of side-length $\Delta^{j}$, then 
\begin{displaymath} \Div_{\mu}(\nu,\mathcal{D}_{N}) - \Div_{\mu}(\nu,\mathcal{D}_{0}) = \sum_{j = 0}^{N - 1} \Div_{\mu}(\nu,\mathcal{D}_{j + 1} \mid \mathcal{D}_{j}). \end{displaymath} 
 \end{lemma}

\begin{proof} Write $\mathcal{E}(F) := \{E \in \mathcal{F} : E \subset F\}$, $F \in \mathcal{F}$. By hypothesis, $\mathcal{E}(F)$ is a partition of $F$. Now, \eqref{form19a} follows by chasing definitions:
\begin{align*} \Div_{\mu}(\nu,\mathcal{E} \mid \mathcal{F}) & = \sum_{F \in \mathcal{F}} \nu(F) \sum_{E \in \mathcal{E}(F)} \nu_{F}(E) \log \frac{\nu_{F}(E)}{\mu_{F}(E)}\\
& = \sum_{F \in \mathcal{F}} \sum_{E \in \mathcal{E}(F)} \nu(E) \log \frac{\nu(E)/\nu(F)}{\mu(E)/\mu(F)}\\
& = \sum_{F \in \mathcal{F}} \sum_{E \in \mathcal{E}(F)} \nu(E) \log \frac{\nu(E)}{\mu(E)} - \sum_{F \in \mathcal{F}} \sum_{E \in \mathcal{E}(F)} \nu(E) \log \frac{\nu(F)}{\mu(F)}\\
& = \sum_{E \in \mathcal{E}} \nu(E) \log \frac{\nu(E)}{\mu(E)} - \sum_{F \in \mathcal{F}} \nu(F) \log \frac{\nu(F)}{\mu(F)} = \Div_{\mu}(\nu,\mathcal{E}) - \Div_{\mu}(\nu,\mathcal{F}), \end{align*} 
as desired. \end{proof} 

The following lemma gives a universal lower bound for "partial sums" of the type appearing in the definition of $\Div_{\mu}(\nu,\mathcal{E})$.

\begin{lemma}\label{lemma10} Let $\mu,\nu$ be two probability measures on $\Omega$, and let $\mathcal{G}$ be a collection of disjoint $\mu$ and $\nu$ measurable subsets of $\Omega$ whose union $G$ has $\nu(G) > 0$. Then, 
\begin{displaymath} \sum_{E \in \mathcal{G}} \nu(E) \log \frac{\nu(E)}{\mu(E)} \geq \nu(G) \log \frac{\nu(G)}{\mu(G)} \geq -1. \end{displaymath}
In particular, $\mathfrak{D}_{\mu}(\nu,\mathcal{E}) \geq 0$ whenever $\mathcal{E}$ is a $\mu$ and $\nu$ measurable partition of $\Omega$ (this special case is Gibbs' inequality). \end{lemma}

\begin{proof} Apply Jensen's inequality in the space $(\mathcal{G},\nu_{G})$ to the variable $E \mapsto \mu(E)/\nu(E)$:
\begin{align*} \sum_{E \in \mathcal{G}} \nu(E) \log \frac{\nu(E)}{\mu(E)} & = - \nu(G) \sum_{E \in \mathcal{G}} \frac{\nu(E)}{\nu(G)} \log \frac{\mu(E)}{\nu(E)}\\
& \geq -\nu(G) \log \sum_{E \in \mathcal{G}} \frac{\mu(E)}{\nu(G)} = \nu(G) \log \frac{\nu(G)}{\mu(G)}. \end{align*} 
The inequality $\mathfrak{D}_{\mu}(\nu,\mathcal{E}) \geq 0$ follows by applying the above to $\mathcal{G} = \mathcal{E}$ and noting that $\nu(\cup \mathcal{E}) = 1 = \mu(\cup \mathcal{E})$. \end{proof}

We also need the following definition of \emph{David cubes} associated to a Ahlfors $(s,\mathbf{C})$-regular measure (see \cite{MR1009120} or \cite[p. 86]{MR1123480} for a proof of existence).

\begin{definition}[David cubes]\label{def:DavidCubes} Let $\mu$ be a Ahlfors $(s,\mathbf{C})$-regular measure on $\R^{d}$. Then, there exists a family of partitions $\mathcal{D}_{\mu}^{i}$, $i \in \Z$, of $\R^{d}$ into "cubes" with the following properties (the constant $A = A(\mathbf{C},d) \geq 1$ depends only on $\mathbf{C}$ and $d$):
\begin{itemize}
\item[(Q1) \phantomsection \label{Q1}] If $i \leq j$, $Q \in \mathcal{D}_{\mu}^{j}$ and $Q' \in \mathcal{D}_{\mu}^{i}$, then either $Q \cap Q' = \emptyset$, or $Q \subset Q'$.
\item[(Q2) \phantomsection \label{Q2}] If $Q \in \mathcal{D}^{i}_{\mu}$, then $A^{-1}2^{-j} \leq \diam Q \leq A2^{-j}$.
\item[(Q3) \phantomsection \label{Q3}] If $Q \in \mathcal{D}^{i}_{\mu}$, and $2^{-i} \leq \diam(\spt \mu)$, then $A^{-1}2^{-is} \leq \mu(Q) \leq A2^{-is}$
\end{itemize}
\end{definition}

\begin{remark} The existence of David cubes requires $\mu$ to be Ahlfors $s$-regular, since the lower bound $\mu(Q) \geq A^{-1}2^{-is}$ in \nref{Q3} would be hopeless for (weakly) $s$-regular measures. This property is certainly useful in the proof of Lemma \ref{lemma9}, below, but it seems plausible that the result also holds for (weakly) $s$-regular measures. Accomplishing this would also allow one to generalise Theorem \ref{mainThm2} to (weakly) $s$-regular sets and measures. \end{remark}

\subsection{Proof of Lemma \ref{lemma9}} We briefly describe the constants $\mathbf{c}(\mathbf{C}) > 0$ and $\Delta_{0}(\mathbf{C},\mathfrak{d})$ from the statement. Recall that $\mathfrak{d} = \Sigma - \sigma > 0$. Let $\Delta_{0} = \Delta_{0}(\mathbf{C},\mathfrak{d}) \in (0,\tfrac{1}{2}]$ be so small that 
\begin{equation}\label{form17a} (\Delta_{0}(\mathbf{C},\mathfrak{d}))^{-\mathfrak{d}} \geq C \end{equation}
for a suitable constant $C \geq 1$, and in particular is than the (absolute) constant from Lemma \ref{lemma5a}. A further $\mathbf{C}$-dependent requirement will appear below \eqref{form28a}. The constant $\mathbf{c} \gtrsim_{\mathbf{C}} 1$ will be determined right above \eqref{form56a}.

Let $\Delta \in (0,\Delta_{0}]$ be the scale satisfying \eqref{form14a}, and let $\delta = \Delta^{N}$, where 
\begin{equation}\label{form52a}N > \tfrac{200^{2}}{\mathfrak{d}^{4}} \log \tfrac{1}{\eta}. \end{equation}
We make a counter assumption to \eqref{form15a}: there exists an Ahlfors $(s,\mathbf{C})$-regular measure $\mu$ with $\spt \mu = K$ and $\diam(K) \geq 1$, and a set $E \subset S^{1}$ with $\nu(E) \geq \epsilon$ such that
\begin{displaymath} K_{\theta} := B(1) \cap K \cap H_{\theta}(K,\delta^{-\Sigma},[\delta,1]) \end{displaymath}
satisfies $\mu(K_{\theta}) \geq \eta$ for each $\theta \in E$. We note that \eqref{form52a} (and $\Delta \leq \Delta_{0} \leq \tfrac{1}{2}$) implies
\begin{equation}\label{form54a} \eta \geq 2^{-(\mathfrak{d}^{4}/200^{2}) N} \geq (\Delta^{N})^{\mathfrak{d}^{4}/200^{2}} \geq \delta^{\mathfrak{d}^{4}/200^{2}} \geq \delta^{\mathfrak{d}/8}. \end{equation}

Let $\mathcal{D} = \bigcup_{i \in \Z} \mathcal{D}_{\mu}^{i}$ be a family of David cubes as in Definition \ref{def:DavidCubes}. Let $Q_{0} \in \mathcal{D}_{\mu}^{0}$ be a "unit cube" containing $B(1) \cap \spt \mu$, and in particular all the sets $K_{\theta}$. It would be useful if we knew that $\mu$ is a probability measure on $Q_{0}$. To arrange this, we re-define
\begin{equation}\label{form37a} \mu := \mu_{Q_{0}}, \end{equation}
where $\mu_{Q_{0}} = \mu(Q_{0})^{-1}\mu|_{Q_{0}}$. (Nothing happening outside $Q_{0}$ will concern us.)

Fix $\theta \in E$. We claim that there exists a subset $\bar{K}_{\theta} \subset K_{\theta}$ with $\mu(\bar{K}_{\theta}) \geq \tfrac{1}{2}\mu(K_{\theta})$ such that 
\begin{equation}\label{form7a} |\{0 \leq j \leq N - 1 : x \in H_{\theta}(K,\Delta^{-\bar{\sigma}},[50\Delta^{j + 1},50\Delta^{j}]) \geq \tfrac{\mathfrak{d}}{2}N, \qquad x \in \bar{K}_{\theta}, \end{equation}
where
\begin{equation}\label{form23a} \bar{\sigma} := \sigma + \frac{\mathfrak{d}}{4}. \end{equation}
Suppose \eqref{form7a} fails. Then, there exists a subset $B_{\theta} \subset K_{\theta}$ with $\mu(B_{\theta}) \geq \tfrac{1}{2}\mu(K_{\theta}) \geq \eta/2$ such that the inequality opposite to \eqref{form7a} holds for all $x \in B_{\theta}$. According to Lemma \ref{lemma5a} applied to $F := B_{\theta} \subset K$, this yields
\begin{equation}\label{form55a} |(B_{\theta})_{4\delta} \cap \pi_{\theta}^{-1}\{t\}|_{4\delta} \leq C^{N}\delta^{-\bar{\sigma} - \mathfrak{d}/2} \stackrel{\eqref{form17a}}{\leq} \delta^{-\sigma - 3\mathfrak{d}/4}, \qquad t \in \R. \end{equation}
On the other hand, by Lemma \ref{l:hereditary} ("hereditary property of high-multiplicity sets"), since $B_{\theta} \subset B(1) \cap H_{\theta}(K,\delta^{-\Sigma},[\delta,\infty])$ with 
\begin{displaymath} \mu(B_{\theta}) =: \kappa \geq \tfrac{\eta}{2} \stackrel{\eqref{form54a}}{\geq} \tfrac{1}{2}\delta^{\mathfrak{d}/8}, \end{displaymath}
there exists a non-empty subset $G_{\theta} \subset B_{\theta}$ such that $\mathfrak{m}_{B_{\theta},\theta}(x \mid [4\delta,\infty]) \geq c\mathbf{C}^{-2}\kappa \delta^{-\Sigma}$ for $x \in G_{\theta}$. In other words
\begin{displaymath} |(B_{\theta})_{4\delta} \cap \pi_{\theta}^{-1}\{\pi_{\theta}(x)\}|_{4\delta} \geq |(G_{\theta})_{4\delta} \cap \pi_{\theta}^{-1}\{\pi_{\theta}(x)\}|_{4\delta} \gtrsim \mathbf{C}^{-2}\delta^{-\Sigma + \mathfrak{d}/8}, \qquad x \in G_{\theta}. \end{displaymath}
This contradicts \eqref{form55a}, since $\delta^{-\mathfrak{d}/8} \geq \Delta_{0}^{-\mathfrak{d}/8} \gg \mathbf{C}^{2}$ by \eqref{form17a}. Therefore, the existence of $\bar{K}_{\theta} \subset K_{\theta}$ with $\mu(\bar{K}_{\theta}) \geq \tfrac{1}{2}\mu(K_{\theta}) \geq \eta/2$, as in \eqref{form7a}, is guaranteed. Since there will be no difference between $\bar{K}_{\theta}$ and $K_{\theta}$ in the sequel (the difference between "$\eta$" and "$\eta/2$" is irrelevant) we assume with no loss of generality that \eqref{form7a} holds for all $x \in K_{\theta}$.

As a final twist of the previous argument, we note that if $\theta \in E$ and $x \in [K_{\theta}]_{A\delta}$ (with $A \geq 1$), then \eqref{form7a} holds in the following slightly weaker form:
\begin{equation}\label{form7b} |\{0 \leq j \leq N - 1 : x \in H_{\theta}(K,a\Delta^{-\bar{\sigma}},[A'\Delta^{j + 1},A'\Delta^{j}])\}| \geq \tfrac{\mathfrak{d}}{2}N, \qquad x \in [K_{\theta}]_{A\delta}, \end{equation} 
where $A' \lesssim A$, and $a \sim A^{-1}$. This is because to every $x \in [K_{\theta}]_{A\delta}$ corresponds $x' \in K_{\theta}$ with $|x - x'| \leq A\delta$ satisfying \eqref{form7a}, and one can check using the triangle inequality that
\begin{displaymath} x' \in H_{\theta}(K,\Delta^{-\bar{\sigma}},[50\Delta^{j + 1},50\Delta^{j}]) \quad \Longrightarrow \quad x \in H_{\theta}(K,a\Delta^{-\bar{\sigma}},[\bar{A}\Delta^{j + 1},\bar{A}\Delta^{j}]). \end{displaymath} 

Recall that $\delta = \Delta^{N}$ with $\Delta \leq \Delta_{1}$. For all $0 \leq j \leq N$, there exists a unique integer $i(j) \in \N$ such that
\begin{displaymath} \Delta^{j} \in [2^{-i(j) - 1},2^{-i(j)}]. \end{displaymath}
We now abbreviate $\mathcal{D}_{j} := \{Q \in \mathcal{D}_{\mu}^{i(j)} : Q \subset Q_{0}\}$ for $0 \leq j \leq N$. Thus, the cubes $Q \in \mathcal{D}_{j}$ satisfy
\begin{equation}\label{form26a} A^{-1}\Delta^{j} \leq A^{-1}2^{-i(j)} \leq \diam(Q) \leq A2^{-i(j)} \leq 2A\Delta^{j}, \qquad Q \in \mathcal{D}_{j}, \end{equation}
where $A = A(\mathbf{C}) \geq 1$ is the constant from \nref{Q2}.

For every $\theta \in E$, write $(K_{\theta})_{N} := \cup \{Q \in \mathcal{D}_{N} : Q \cap K_{\theta} \neq \emptyset\}$, and define the measure
\begin{displaymath} \mu_{\theta} := \mu((K_{\theta})_{N})^{-1}\mu|_{(K_{\theta})_{N}}, \qquad \theta \in E. \end{displaymath}
Recall that \eqref{form7b} holds for all $x \in \spt \mu_{\theta}$, where $\mu_{\theta}$ is a probability measure. For $\theta \in E$ fixed, applying Lemma \ref{lemma12} with parameter $\mathfrak{d}/2$ to the sets
\begin{displaymath} H(Q) := Q \cap H_{\theta}(K,a\Delta^{-\bar{\sigma}},[A'\Delta^{j + 1},A'\Delta^{j}]) \subset Q \end{displaymath}
yields the following: there exists a subset $F_{\theta} \subset \spt \mu_{\theta}$ with $\mu_{\theta}(F_{\theta}) \geq \mathfrak{d}^{2}/100$ such that
\begin{equation}\label{form51a} |\{0 \leq j \leq N - 1 : \mu_{\theta,Q_{j}(x)}(H_{\theta}(K,a\Delta^{-\bar{\sigma}},[A'\Delta^{j + 1},A'\Delta^{j}])) \geq \tfrac{\mathfrak{d}}{8}\}| \geq \tfrac{\mathfrak{d}^{2}}{100}N, \quad x \in F_{\theta}. \end{equation}
Here $Q_{j}(x)$ is the unique element of $\mathcal{D}_{j}$ containing $x$, and $\mu_{\theta,Q} = \mu_{\theta}(Q)^{-1}(\mu_{\theta})|_{Q}$. 

Note that if $\mu_{\theta}(Q) \neq 0$ for some $Q \in \mathcal{D}_{N}$, then in fact $\mu_{\theta}(Q) = \mu((K_{\theta})_{N})^{-1}\mu(Q)$. Thus,
\begin{align*} \Div_{\mu}(\mu_{\theta},\mathcal{D}_{N}) & =\sum_{Q \in \mathcal{D}_{N}} \mu_{\theta}(Q) \log \frac{\mu_{\theta}(Q)}{\mu(Q)}\\
& = \sum_{Q \in \mathcal{D}_{N}(F_{\theta})} \frac{\mu(Q)}{\mu(F_{\theta})} \log \frac{\mu(Q)/\mu(F_{\theta})}{\mu(Q)} \leq \log \frac{1}{\mu(K_{\theta})} \leq \log \tfrac{1}{\eta}. \end{align*}
Let us note that if $Q \in \mathcal{D}$ is a cube with $\mu(Q) = 0$, then also $\mu_{\theta}(Q) = 0$. Therefore, expressions of the form $\mu_{\theta}(Q) \log (\mu_{\theta}(Q)/\mu(Q))$ appearing below (and above for $Q \in \mathcal{D}_{N}$) are well-defined for all $Q \in \mathcal{D}$.

For $x \in \R^{2}$ and $0 \leq j \leq N$, let $Q_{j}(x) \in \mathcal{D}_{j}$ be the unique cube containing $x$. Fix $\theta \in E$. Using Lemma \ref{lemma:telescope}, and noting that $\mathfrak{D}_{\mu}(\mu_{\theta},\mathcal{D}_{0}) = 0$,
\begin{align} \log \tfrac{1}{\eta} \geq \Div_{\mu}(\mu_{\theta},\mathcal{D}_{\delta}) & = \sum_{j = 0}^{N - 1} \Div_{\mu}(\mu_{\theta},\mathcal{D}_{j + 1} \mid \mathcal{D}_{j}) \notag\\
& = \sum_{j = 0}^{N - 1} \sum_{Q \in \mathcal{D}_{j}} \mu_{\theta}(Q)\Div_{\mu_{Q}}((\mu_{\theta})_{Q},\mathcal{D}_{j + 1}) \notag \\
&\label{form9a} = \int \sum_{j = 0}^{N - 1} \Div_{\mu_{Q_{j}(x)}}((\mu_{\theta})_{Q_{j}(x)}, \mathcal{D}_{j + 1}) \, d\mu_{\theta}(x). \end{align}
Here $\mu_{\theta}^{Q} = \mu_{\theta}(Q)^{-1}\mu_{\theta}|_{Q}$ is the renormalised restriction to $Q$, as in Definition \ref{def:divergence}. Since $\mathfrak{D}_{\mu}(\nu,\mathcal{E}) \geq 0$ for all probabilities $\mu,\nu$ and partitions $\mathcal{E}$, \eqref{form9a} implies by Chebychev's inequality that there exists a subset $G_{\theta} \subset F_{\theta}$ with $\mu_{\theta}(G_{\theta}) \geq \tfrac{1}{2}\mu_{\theta}(F_{\theta}) \geq \tfrac{\mathfrak{d}^{2}}{200}$ such that
\begin{equation}\label{form8a} \sum_{j = 0}^{N - 1} \Div_{\mu_{Q_{j}(x)}}((\mu_{\theta})_{Q_{j}(x)}, \mathcal{D}_{j + 1}) \leq \tfrac{200}{\mathfrak{d}^{2}}\log \tfrac{1}{\eta}, \qquad x \in G_{\theta}. \end{equation}
Since $N \geq (200^{2}/\mathfrak{d}^{4})\log \tfrac{1}{\eta}$ by \eqref{form52a}, this further implies
\begin{displaymath} |\{0 \leq j \leq N - 1 : \Div_{\mu_{Q_{j}(x)}}((\mu_{\theta})_{Q_{j}(x)}, \mathcal{D}_{j + 1}) \leq 1\}| \geq (1 - \tfrac{\mathfrak{d}^{2}}{200}) N, \qquad x \in G_{\theta}, \end{displaymath}
Combining this information with \eqref{form51a} (which holds in particular for $x \in G_{\theta} \subset F_{\theta}$),
\begin{align} & |\{0 \leq j \leq N - 1 :  \mu_{\theta,Q_{j}(x)}(H_{\theta}(K,a\Delta^{-\bar{\sigma}},[A'\Delta^{j + 1},A'\Delta^{j}])) \geq \tfrac{\mathfrak{d}}{8} \notag\\
&\label{form11a} \qquad \qquad \text{ and } \Div_{\mu_{Q_{j}(x)}}((\mu_{\theta})_{Q_{j}(x)}, \mathcal{D}_{j + 1}) \leq 1\}| \geq \tfrac{\mathfrak{d}^{2}}{200} N, \qquad x \in G_{\theta}. \end{align}
for all $x \in G_{\theta}$, where $a \sim A^{-1} \gtrsim_{\mathbf{C}} 1$ and $\bar{A} \lesssim A \lesssim_{\mathbf{C}} 1$.

For $0 \leq j \leq N - 1$ and $\theta \in E$, let $K(j,\theta)$ be the points $x \in \R^{2}$ such that
\begin{equation}\label{form21a} \mu_{\theta,Q_{j}(x)}(H_{\theta}(K,a\Delta^{-\bar{\sigma}},[A'\Delta^{j + 1},A'\Delta^{j}])) \geq \tfrac{\mathfrak{d}}{8} \quad \text{and} \quad  \Div_{\mu_{Q_{j}(x)}}((\mu_{\theta})_{Q_{j}(x)}, \mathcal{D}_{j + 1}) \leq 1. \end{equation}
Using Fubini's theorem, recalling that $\mu_{\theta}(G_{\theta}) \geq \tfrac{1}{2}$ for $\theta \in E$, and that $\nu(E) \geq \epsilon$,
\begin{align*} \frac{1}{N}\sum_{j = 0}^{N - 1} \int_{E} \mu_{\theta}(K(j,\theta)) \, d\nu(\theta) & \geq \int_{E} \int_{G_{\theta}} \tfrac{1}{N}|\{0 \leq j \leq N - 1 : \text{as in } \eqref{form11a} \}| \, d\mu_{\theta}(x) \, d\nu(\theta)\\
& \geq \int_{E} \int_{G_{\theta}} \tfrac{\mathfrak{d}^{2}}{200} \, d\mu_{\theta}(x) \, d\nu(\theta) \geq \tfrac{\mathfrak{d}^{2}}{400}\nu(E) \geq \tfrac{\mathfrak{d}^{2}\epsilon}{400}. \end{align*}
Consequently, there exists a fixed index $0 \leq j \leq N - 1$ such that
\begin{displaymath} \sum_{Q \in \mathcal{D}_{j}} \int_{E} \mu_{\theta}(Q \cap K(j,\theta)) \, d\nu(\theta) \geq \tfrac{\mathfrak{d}^{2}\epsilon}{400}. \end{displaymath}
Recalling that each $\mu_{\theta}$, $\theta \in E$, is a probability measure, there exists a fixed cube $Q \in \mathcal{D}_{j}$ such that
\begin{displaymath} \int_{E} (\mu_{\theta})_{Q}(K(j,\theta)) \, d\nu(\theta) \geq \tfrac{\mathfrak{d}^{2}\epsilon}{400}. \end{displaymath}
Thus, fixing this cube $Q \in \mathcal{D}_{j}$ for the remainder of the argument, there is a set $S \subset E$ with $\nu(S) \geq \mathfrak{d}^{2}\epsilon/400$ such that $Q \cap K(j,\theta) \neq \emptyset$ for $\theta \in S$. Thus, if $\theta \in S$, then $Q$ contains a point $x_{0} \in K(j,\theta)$. Therefore by the definition \eqref{form21a},
\begin{displaymath} \mu_{\theta,Q_{j}(x_{0})}(H_{\theta}(K,a\Delta^{-\bar{\sigma}},[A'\Delta^{j + 1},A'\Delta^{j}])) \geq \tfrac{\mathfrak{d}}{8} \quad \text{and} \quad \Div_{\mu_{Q_{j}(x_{0})}}((\mu_{\theta})_{Q_{j}(x)}, \mathcal{D}_{j + 1}) \leq 1. \end{displaymath}
But now $Q_{j}(x_{0}) = Q$, so
\begin{equation}\label{form20a} \mu_{\theta,Q}(H_{\theta}(K,a\Delta^{-\bar{\sigma}},[A'\Delta^{j + 1},A'\Delta^{j}])) \geq \tfrac{\mathfrak{d}}{8} \quad \text{and} \quad \Div_{\mu_{Q}}((\mu_{\theta})_{Q}, \mathcal{D}_{j + 1})\leq 1, \quad \theta \in S. \end{equation}
For the remainder of the argument, we abbreviate $\bar{\mu}_{\theta} := \mu_{\theta,Q}$ and $\bar{\mu} := \mu_{Q}$. With this notation, consider the collections
\begin{equation}\label{form13a} \mathcal{B}_{\theta} := \{p \in \mathcal{D}_{j + 1} : \bar{\mu}_{\theta}(p) \geq C(\mathfrak{d})\mu(p)\} \quad \text{and} \quad \mathcal{G}_{\theta} := \mathcal{D}_{j + 1} \, \setminus \, \mathcal{B}, \end{equation}
where $C(\mathfrak{d}) := 2^{32/\mathfrak{d}}$, thus $2 / \log C(\mathfrak{d}) \leq \mathfrak{d}/16$. Applying Lemma \ref{lemma10} to $\mathcal{G}$,
\begin{displaymath} 1 \geq \sum_{p \in \mathcal{B}_{\theta}} \bar{\mu}_{\theta}(p) \log \frac{\bar{\mu}_{\theta}(p)}{\bar{\mu}(p)} + \sum_{p \in \mathcal{G}_{\theta}} \bar{\mu}_{\theta}(p) \log \frac{\bar{\mu}_{\theta}(p)}{\bar{\mu}(p)} \geq \log C(\mathfrak{d}) \bar{\mu}_{\theta}(B_{\theta}) - 1, \end{displaymath} 
and therefore $\bar{\mu}_{\theta}(B_{\theta}) \leq \frac{2}{\log C(\mathfrak{d})} \leq \frac{\mathfrak{d}}{16}$. Recalling \eqref{form20a}, this implies 
\begin{equation}\label{form23a} \bar{\mu}_{\theta}((\cup \mathcal{G}_{\theta}) \cap H_{\theta}(K,a\Delta^{-\bar{\sigma}},[A'\Delta^{j + 1},A'\Delta^{j}])) \geq \tfrac{\mathfrak{d}}{16}, \qquad \theta \in S. \end{equation} 
Finally, let $\mathcal{H}_{\theta} \subset \mathcal{G}_{\theta} \subset \mathcal{D}_{j + 1}$ be the subset of cubes with 
\begin{displaymath} p \cap H_{\theta}(K,a\Delta^{-\bar{\sigma}},[A'\Delta^{j + 1},A'\Delta^{j}]) \neq \emptyset. \end{displaymath}
Since the cubes in $\mathcal{H}_{\theta}$ have diameter $\leq 2A\Delta^{j + 1}$ by \eqref{form26a}, it is easy to check (compare with \eqref{form7b}) that
\begin{equation}\label{form28a} p \subset H_{\theta}(K,\bar{a}\Delta^{-\bar{\sigma}},[\bar{A}\Delta^{j + 1},\bar{A}\Delta^{j}]), \qquad p \in \mathcal{H}_{\theta}, \end{equation}
where $\bar{A} \lesssim A'$, and $\bar{a} \sim a$. In particular, this holds for some $\bar{A} \geq 2A$. Moreover $\bar{a}\Delta^{-\bar{\sigma}} = \bar{a}\Delta^{-\sigma - \mathfrak{d}/4} \geq \Delta^{-\sigma}$, since $\Delta \leq \Delta_{1} \leq \mathbf{c}(\mathbf{C},\mathfrak{d})$, and \eqref{form17a} holds. Consequently,
\begin{align} \bar{\mu}(H_{\theta}(K,\Delta^{-\sigma},[\bar{A}\Delta^{j + 1},\bar{A}\Delta^{j}])) & \geq \sum_{p \in \mathcal{H}_{\theta}} \bar{\mu}(p) \geq C(\eta)^{-1} \sum_{p \in \mathcal{H}_{\theta}} \bar{\mu}_{\theta}(p) \notag\\
& \geq C(\mathfrak{d})^{-1}\bar{\mu}_{\theta}((\cup \mathcal{G}_{\theta}) \cap H_{\theta}(K,\bar{a}\Delta^{-\bar{\sigma}},[A'\Delta^{j + 1},A'\Delta^{j}])) \notag\\
&\label{form29a} \stackrel{\eqref{form23a}}{\geq} C(\mathfrak{d})^{-1}\tfrac{\mathfrak{d}}{16}, \qquad \theta \in S. \end{align}
Recall that $\nu(S) \geq \mathfrak{d}^{2}\epsilon/400$, so we are now quite close to contradicting our main hypothesis \eqref{form14a}. To make this rigorous, recall that $\bar{\mu} = \mu_{Q} = \mu(Q)^{-1}\mu|_{Q}$, where $Q \in \mathcal{D}_{j} = \mathcal{D}_{\mu}^{i(j)}$. Therefore,
\begin{equation}\label{form27a} \mu(Q) \geq A^{-1}2^{-i(j)s} \geq A^{-1}\Delta^{js} \end{equation} 
according to \nref{Q3}. Let $B \supset Q$ be a disc of radius $\bar{A}\Delta$; this exists by \eqref{form26a}, and since $\bar{A} \geq 2A$. Then, with our usual notation $\mu^{B} = \mathrm{rad}(B)^{-s}T_{B}\mu$, and using Lemma \ref{lemma7},
\begin{align*} \mu^{B}(B(1) \cap H_{\theta}(K^{B},\Delta^{-\sigma},[\Delta,1])) & = (\bar{A}\Delta^{j})^{-s}\mu(B \cap H_{\theta}(K,\Delta^{-\sigma},[\bar{A}\Delta^{j + 1},\bar{A}\Delta^{j}]))\\
& \stackrel{\eqref{form27a}}{\geq} (A\bar{A})^{-1}\mu(Q)^{-1}\mu(Q \cap H_{\theta}(K,\Delta^{-\sigma},[\bar{A}\Delta^{j + 1},\bar{A}\Delta^{j}]))\\
& = (A\bar{A})^{-1}\bar{\mu}(H_{\theta}(K,\Delta^{-\sigma},[\bar{A}\Delta^{j + 1},\bar{A}\Delta^{j}])). \end{align*} 
Writing $\mathbf{c} := \tfrac{1}{16}(A\bar{A})^{-1} \gtrsim_{\mathbf{C}} 1$, and combining this lower bound with \eqref{form29a} yields
\begin{equation}\label{form56a} \nu\left(\left\{\theta \in S^{1} : \mu^{B}(B(1) \cap H_{\theta}(K^{B},\Delta^{-\sigma},[\Delta,1])) \geq \tfrac{\mathbf{c}\mathfrak{d}}{C(\mathfrak{d})}\right\}\right) \geq \nu(S) \geq \tfrac{\mathfrak{d}^{2}}{400}\epsilon. \end{equation} 
Since $\mu^{B}$ is Ahlfors $(s,\mathbf{C})$-regular, this contradicts the hypothesis \eqref{form14a}, and completes the proof of Lemma \ref{lemma9}. 

A minor technical point is that we had at \eqref{form37a} (re-)defined $\mu$ to mean $\mu_{Q_{0}}$. One may question if $\mu_{Q_{0}}^{B}$ is indeed Ahlfors $(s,\mathbf{C})$-regular, and applying \eqref{form14a} is legitimate. But since $\mu_{Q_{0}}$ is a renormalised restriction of $\mu$, the lower bound \eqref{form56a} holds also for the "original" $\mu$, at the expense of multiplying $\mathbf{c}_{1}$ by $\mu(Q_{0}) \gtrsim_{\mathbf{C}} 1$ (recall that $\diam K \geq 1$).

\bibliographystyle{plain}
\bibliography{references}

\def\cprime{$'$}
\begin{thebibliography}{10}

\bibitem{MR3955707}
Bal\'{a}zs B\'{a}r\'{a}ny, Michael Hochman, and Ariel Rapaport.
\newblock Hausdorff dimension of planar self-affine sets and measures.
\newblock {\em Invent. Math.}, 216(3):601--659, 2019.

\bibitem{Bo2}
Jean Bourgain.
\newblock The discretized sum-product and projection theorems.
\newblock {\em J. Anal. Math.}, 112:193--236, 2010.

\bibitem{MR3903263}
Catherine Bruce and Xiong Jin.
\newblock Projections of {G}ibbs measures on self-conformal sets.
\newblock {\em Nonlinearity}, 32(2):603--621, 2019.

\bibitem{MR4745881}
Damian D{\k a}browski, Tuomas Orponen, and Hong Wang.
\newblock How much can heavy lines cover?
\newblock {\em J. Lond. Math. Soc. (2)}, 109(5):Paper No. e12910, 33, 2024.

\bibitem{MR1009120}
Guy David.
\newblock Morceaux de graphes lipschitziens et int\'{e}grales singuli\`eres sur
  une surface.
\newblock {\em Rev. Mat. Iberoamericana}, 4(1):73--114, 1988.

\bibitem{MR1123480}
Guy David.
\newblock {\em Wavelets and singular integrals on curves and surfaces}, volume
  1465 of {\em Lecture Notes in Mathematics}.
\newblock Springer-Verlag, Berlin, 1991.

\bibitem{MR3263957}
Kenneth~J. Falconer and Xiong Jin.
\newblock Exact dimensionality and projections of random self-similar measures
  and sets.
\newblock {\em J. Lond. Math. Soc. (2)}, 90(2):388--412, 2014.

\bibitem{MR3276605}
Andrew Ferguson, Jonathan~M. Fraser, and Tuomas Sahlsten.
\newblock Scaling scenery of {$(\times m,\times n)$} invariant measures.
\newblock {\em Adv. Math.}, 268:564--602, 2015.

\bibitem{He20}
Weikun He.
\newblock Orthogonal projections of discretized sets.
\newblock {\em J. Fractal Geom.}, 7(3):271--317, 2020.

\bibitem{Ho}
Michael Hochman.
\newblock On self-similar sets with overlaps and inverse theorems for entropy.
\newblock {\em Ann. of Math. (2)}, 180(2):773--822, 2014.

\bibitem{KM}
Robert Kaufman and Pertti Mattila.
\newblock Hausdorff dimension and exceptional sets of linear transformations.
\newblock {\em Ann. Acad. Sci. Fenn. Ser. A I Math.}, 1(2):387--392, 1975.

\bibitem{mar}
John~M. Marstrand.
\newblock Some fundamental geometrical properties of plane sets of fractional
  dimensions.
\newblock {\em Proc. London Math. Soc. (3)}, 4:257--302, 1954.

\bibitem{MS08}
Pertti Mattila and Pirjo Saaranen.
\newblock Ahlfors-{D}avid regular sets and bilipschitz maps.
\newblock {\em Ann. Acad. Sci. Fenn. Math.}, 34(2):487--502, 2009.

\bibitem{Or1}
Tuomas Orponen.
\newblock On the packing dimension and category of exceptional sets of
  orthogonal projections.
\newblock {\em Ann. Mat. Pura Appl. (4)}, 194(3):843--880, 2015.

\bibitem{MR4055989}
Tuomas Orponen.
\newblock An improved bound on the packing dimension of {F}urstenberg sets in
  the plane.
\newblock {\em J. Eur. Math. Soc. (JEMS)}, 22(3):797--831, 2020.

\bibitem{MR4218963}
Tuomas Orponen.
\newblock On the {A}ssouad dimension of projections.
\newblock {\em Proc. Lond. Math. Soc. (3)}, 122(2):317--351, 2021.

\bibitem{MR4388762}
Tuomas Orponen.
\newblock On arithmetic sums of {A}hlfors-regular sets.
\newblock {\em Geom. Funct. Anal.}, 32(1):81--134, 2022.

\bibitem{OS23}
Tuomas Orponen and Pablo Shmerkin.
\newblock On the {H}ausdorff dimension of {F}urstenberg sets and orthogonal
  projections in the plane.
\newblock {\em Duke Math. J.}, 172(18):3559--3632, 2023.

\bibitem{2023arXiv230110199O}
Tuomas {Orponen} and Pablo {Shmerkin}.
\newblock {Projections, Furstenberg sets, and the $ABC$ sum-product problem}.
\newblock {\em arXiv e-prints}, page arXiv:2301.10199, January 2023.

\bibitem{PS}
Yuval Peres and Pablo Shmerkin.
\newblock Resonance between {C}antor sets.
\newblock {\em Ergodic Theory Dynam. Systems}, 29(1):201--221, 2009.

\bibitem{2023arXiv230904097R}
Kevin {Ren}.
\newblock {Discretized Radial Projections in $\mathbb{R}^d$}.
\newblock {\em arXiv e-prints}, page arXiv:2309.04097, September 2023.

\bibitem{2023arXiv230808819R}
Kevin {Ren} and Hong {Wang}.
\newblock {Furstenberg sets estimate in the plane}.
\newblock {\em arXiv e-prints}, page arXiv:2308.08819, August 2023.

\bibitem{Sh}
Pablo Shmerkin.
\newblock On {F}urstenberg's intersection conjecture, self-similar measures,
  and the {$L^q$} norms of convolutions.
\newblock {\em Ann. of Math. (2)}, 189(2):319--391, 2019.

\bibitem{Wu24}
M.~Wu.
\newblock A sharp {M}arstrand type projection theorem for self-similar measures
  and {CP}-distributions, and applications.
\newblock {\em Forthcoming preprint}, 2024.

\end{thebibliography}

\end{document}